\documentclass{article}
\usepackage[a4paper,totalwidth=17cm,totalheight=25cm]{geometry}
\usepackage[table]{xcolor}
\usepackage[utf8]{inputenc}
\usepackage{amsmath,amsthm}
\usepackage{amsfonts}
\usepackage{amssymb}
\usepackage[all,arc]{xy}
\usepackage{tikz-cd}
\usetikzlibrary{arrows}
\usepackage{tikz}
\usetikzlibrary{positioning}
\usepackage{soul}
\usepackage[title,titletoc,toc]{appendix}
\usepackage{floatflt}
\usepackage{bbm}
\usepackage{enumitem}
\usepackage{floatflt}
 \usepackage{wrapfig}
\usepackage{ulem}
\usepackage{soul}
\usepackage{color}
\usepackage{makeidx}

\usepackage{etex}

\usepackage{microtype}%evite de underfull hbox

\usepackage[outercaption]{sidecap}

\def\@oddhead{\hfill \shorttitle \hfill \thepage}
\def\@evenhead{\thepage \hfill \shortauthor \hfill}
\def\@oddfoot{}
\def\@evenfoot{}

\swapnumbers % Avec cette ligne, le numéro des théorèmes s'affiche avant le mot "Théorème".
\newtheorem{df}{Definition}[section]
\newtheorem{fact}[df]{Fact}

\newtheorem{remark}[df]{Remark}

\newtheorem{theorem}[df]{Theorem}

\newtheorem{lemma}[df]{Lemma}

\newtheorem{cor}[df]{Corollary}
\newtheorem{proposition}[df]{Proposition}
\newcommand{\R}{\mathbb{R}}
\renewcommand{\S}{\mathcal{S}}
\renewcommand{\H}{\mathbb{H}}

\newcommand{\E}{\mathrm{E}}

\newcommand{\C}{\mathcal{C}}
\newcommand{\M}{\mathbb{M} \kern-0.1em \mathbbm{i} \kern-0.1em \mathbbm{n}}

\newcommand\itemrow[3]{%
  \item\makebox[8em][l]{#1}%
    \makebox[8em][r]{#2}%
    \makebox[8em][l]{#3}%
}

\DeclareFontFamily{OT1}{pzc}{}
\DeclareFontShape{OT1}{pzc}{m}{it}{<-> s * [0.99] pzcmi7t}{}
\DeclareMathAlphabet{\mathscr}{OT1}{pzc}{m}{it}

\newcommand{\cM}{\mathscr{co}\mathcal{M}\!\mathscr{in}}
\newcommand{\cE}{\mathscr{co}\mathcal{E}\!\mathscr{uc}}
\newcommand{\Eu}{\mathcal{E}\!\mathscr{uc}}

\renewcommand{\P}{\mathrm{P}}

\renewcommand{\E}{\mathbb{E}}
\newcommand{\Ell}{\mathbb{E} \kern-0.1em \mathbbm{l} \kern-0.1em \mathbbm{l}}
\newcommand{\Gal}{\mathbb{G} \kern-0.1em \mathbbm{a} \kern-0.1em \mathbbm{l}}

\newcommand{\Hess}{\operatorname{Hess}}

\newcommand{\AdS}{\mathbb{A}\mathrm{d}\mathbb{S}}
\newcommand{\dS}{\mathrm{d}\mathbb{S}}

\newcommand{\I}{\mathrm{I}}
\newcommand{\II}{\mathrm{I\hspace{-0.04cm}I}}
\newcommand{\III}{\mathrm{I\hspace{-0.04cm}I\hspace{-0.04cm}I}}

\renewcommand{\d}{\operatorname{d}\!}
\makeindex

\author{Fran\c{c}ois Fillastre and Andrea Seppi}
\title{Spherical, hyperbolic and other projective geometries: convexity, duality, transitions.}
\begin{document}

\date{\today }

\maketitle

\begin{abstract}

We give an elementary projective geometry presentation of the classical Riemannian model spaces (elliptic and hyperbolic spaces) and of the classical Lorentzian model spaces (de Sitter and Anti-de Sitter spaces). We also present some relevant degenerate model spaces (Euclidean and co-Euclidean spaces, Lorentzian Minkowski and co-Minkowski spaces), and geometric transitions. 

An emphasis is given to dimensions $2$ and $3$, convex subsets, duality, and geometric transitions between the spaces.
\end{abstract}



\tableofcontents
\normalem

\newpage

\section{Introduction}

Since the end of the 19th century, and after the works of F. Klein and H. Poincar\'e, it is well known that models of 
elliptic geometry\footnote{The elliptic space is the standard round sphere quotiented by the antipodal map.} and hyperbolic geometry can be given using projective geometry, and that Euclidean geometry can be seen as a ``limit'' of both geometries. (We refer to \cite{AP1,AP2} for historical aspects.)
Then, all the geometries that can be obtained in this way (roughly speaking by defining an ``absolute'', which is the projective quotient of the isotropic cone of a quadratic form) were classified, see  \cite{rosenfeld}.
Some of these geometries  had a rich development, most remarkably after the work of W.P.~Thurston at the end of the 1970's, see \cite{Thurcour1,thb}, which gave rise to a highly prosperous understanding of three-dimensional hyperbolic geometry (see among others \cite{benpet,bonotal,thurston2dim,agol}).
On the other hand, the seminal work of G. Mess (\cite{Mess}, see also \cite{mes+}) of 1990 motivated the study in dimension $(2+1)$ of the Lorentzian geometries of  Minkowski, de Sitter and Anti-de Sitter spaces (these spaces are known under different names 
 in the realm of projective geometries, see  \cite[p.375]{rosenfeld} or \cite{struve}), which in higher dimensions have attracted interest for a long time in mathematical physics and more precisely in General Relativity, see \cite{oneil, beem, hawking}. These geometries are close relatives of hyperbolic geometry and in fact Mess outlined their strong relation with Teichm\"uller theory, thus giving rise to new directions of development which are still very active \cite{notes,bms,bonschlfixed,bon_schl,bbzads,barbotzeghib,bscod,derthol,sephyp,seppimaximal,bstvolume,bsequivariant}.  

Moreover, some degenerate spaces appear naturally in the picture, namely the co-Euclidean space (the space of hyperplanes of the Euclidean space), and the co-Minkowski space (that we will restrict to the space of space-like hyperplanes of Minkowski space), first because of duality reasons, and second because they appear as limits of degeneration of classical spaces. In fact, co-Minkowski space recently reacquired interest under the name \emph{half-pipe geometry}\index{half-pipe geometry}, since the work of J. Danciger \cite{dan-thesis,dan}, which is very related to the idea of geometric transition \cite{hodker,cdkbook,kozai,dgk1,dgk2,bsmink,bsads}.
%\sout{\cite{dan-thesis,dan}.The interest of geometric transition and co-Minkowski space was settled in recent important geometric results, e.g. \cite{DGK}, \cite{DMS}.}

The purpose of this paper is to provide a survey of the properties of these spaces, especially in dimensions 2 and 3, from the point of view of projective geometry. Even with this perspective,  the paper does not aim to be an exhaustive treatment. Instead it is focused on the aspects which concern convex subsets and their duality, degeneration of geometries and some properties of surfaces in three-dimensional spaces. The presentation is intended to be elementary, hence containing no proofs of deep theorems, but trying to proceed by accessible observations and elementary proofs. Apart from some constructions in Sections \ref{sec connection volume} and \ref{sec geometry surfaces}, which have been obtained in \cite{andreathesis}, there is no claim of originality in the presented results. On the other hand, in this paper we attempt to use a modern mathematical language, thus possibly contrasting with the point of view of several presentations of classical topics in the literature.

Hopefully, this survey will provide a unified introduction to the aforementioned geometries, and at the same time it might fill to some extent the lack of references that, in the opinion of the authors, surrounds the  differential-geometric understanding of Minkowski, de Sitter and Anti-de Sitter geometries, quite differently, for instance, from the case of hyperbolic geometry where a large number of textbooks appeared since its modern development.

\section{Model spaces}\label{sec:mode spaces}

\subsection{Pseudo-spheres and model spaces}

The set of unit vectors of three-dimensional Euclidean space, endowed with the metric induced by the ambient space, is a useful model for the round sphere $\S^2$. Indeed, 
in this model,
 \begin{itemize}%[nolistsep]
\item  at a point $x\in \S^2$, the outward unit normal vector is the vector represented by  $x$ itself;
\item  the isometries of $\S^2$ are the restrictions of transformations in $\mathrm O(3)$, namely of linear isometries of Euclidean space;
 \item the geodesics are the intersections of $\S^2$ with linear planes, hence are great circles.\footnote{This follows by contradiction using the symmetry by reflection in the plane and the fact that an unparametrized geodesic from a given point is locally uniquely determined by its tangent vector.}
 \end{itemize} 
Therefore, 
 using this model, it is easy to show that:
 \begin{itemize}%[nolistsep]
 \item $\S^2$ is a smooth surface;
\item the isometry group  of $\S^2$  acts transitively 
on   points and on  orthonormal frames in the tangent spaces;
 \item the sectional curvature is $1$ at every point.
\end{itemize}  

 The above properties  still hold for the 
$n$-dimensional sphere $\S^n$, the set of unit vectors of
the Euclidean space of dimension $(n+1)$, which is a smooth hypersurface of constant sectional curvature if endowed with the induced metric.
Moreover, it is a simple but fundamental remark that the above properties  do not use the fact that the usual scalar product on the ambient space is positive definite, but only that it is a symmetric non-degenerate bilinear form.

 More precisely, if a smooth manifold $M$ is endowed with a smooth $(0,2)$-tensor which is a non-degenerate symmetric bilinear form, then there is an associated connection, the Levi-Civita connection,\footnote{The notion of Levi-Civita connection is recalled in Section \ref{sec connection volume}.}
 and the notions of curvature tensor and geodesic only depend on the connection.\footnote{Of course, if the tensor is not positive definite, it does not induce a distance on $M$, and in particular there is no notion of ``minimization of distance" for the geodesics.} For example, in $\R^{n+1}$ endowed with any non-degenerate symmetric bilinear form $b$, the Levi-Civita connection is given by the usual differentiation in each component, the curvature is zero, and the geodesics are affine lines, regardless of the signature of $b$. 
 
The \emph{isotropic cone}\index{isotropic! cone} (or null cone) of $b$ is the set of isotropic vectors for $b$:

$$\mathcal I(b)=\{ x \in \R^{n+1} \,|\, b(x,x)=0 \}~.$$

\begin{df}
A non-empty subset $\mathcal M$ of $\R^{n+1}$ is a \emph{pseudo-sphere}\index{pseudo-sphere}  if 
$$\mathcal M=\{ x\in \R^{n+1}  \,|\, b(x,x)=1 \} \qquad\text{or}\qquad\mathcal M=\{x\in \R^{n+1}  \,|\, b(x,x)=-1 \}~.$$
for a non-degenerate symmetric bilinear form $b$.
\end{df}

\begin{remark}{\rm
Of course, by replacing $b$ by $-b$, taking $1$ or $-1$ in the definitions leaves the topology unchanged, but
produces changes at a metric level.
More concretely, an \emph{anti-isometry}\index{anti-isometry} between two pseudo-Riemannian spaces\footnote{A smooth manifold is \emph{pseudo-Riemannian}\index{pseudo-Riemannian} if it is endowed with a non-degenerate smooth $(0,2)$-tensor. If the tensor is positive definite at each point, then the manifold is \emph{Riemannian}. If the tensor has a negative direction, but no more than one linearly independent negative direction at each point, then the manifold is \emph{Lorentzian}\index{Lorentzian}.} $(N_1,g_1)$ and $(N_2,g_2)$ is a diffeomorphism $f:N_1 \to N_2$ such that $f^*g_2=-g_1$. For example, if $b'=-b$, the pseudo-sphere 
$$\mathcal M=\{ x \,|\, b(x,x)=1 \}\,, $$
endowed with the metric induced from $b$,
is anti-isometric to
$$\mathcal N=\{ x \,|\, b' (x,x)=-1 \}\,, $$
when $\mathcal N$ is endowed with the metric induced from $b'$, 
the anti-isometry being the identity. As sets, $\mathcal M$ and $\mathcal N$ coincide, but if the signature of $b$ is $(p,q)$, the induced metric on $\mathcal M$ has signature
$(p-1,q)$, while the induced metric on $\mathcal N$ has signature $(q,p-1)$. In the following, given $\mathcal M$, the pseudo-sphere $\mathcal N$ obtained as above will be denoted by $\overline{\mathcal M}$. 
}\end{remark}

 Let $\mathcal M$ be a pseudo-sphere.
 The following facts can be checked similarly to the corresponding properties of the unit sphere in the Euclidean space:
\begin{itemize}
 \item  at a point $x\in \mathcal M$, a unit normal vector is the vector represented by  $x$ itself,
\item  the restrictions of the transformations of $\mathrm O(p,q)$, namely the linear isometries of $b$, are isometries of $\mathcal M$,\footnote{The isometry group of $\mathcal M$ is bigger than $\mathrm O(p,q)$ if $\mathcal M$ is not connected.} 
 \item totally geodesic submanifolds  are the intersections of $\mathcal M$ with vector subspaces.
\end{itemize}
It follows straightforwardly from the above properties that:
\begin{itemize}
 \item $\mathcal M$ is a smooth hypersurface,
\item the isometry group of $\mathcal M$ acts transitively 
on  points and on orthonormal frames in the tangent spaces,
 \item the sectional curvature is constantly equal to 
 $1$  if $\mathcal M=b^{-1}(1)  $ and equal to $-1$ if $\mathcal M=b^{-1}(-1) $.\footnote{This follows from the Gauss formula %\cite[4.20]{oneil} 
 and the fact that the shape operator on $\mathcal M$ is  $\pm$ the identity. The fact that the sectional curvature is constant also follows from the preceding item.}
 \end{itemize}

For $n\geq 2$, the manifold $b^{-1}(1)$ is diffeomorphic to $\mathcal{S}^{p-1} \times \R^q$, and $b^{-1}(-1)$ is diffeomorphic to $\mathcal{S}^{q-1} \times \R^p$, so it may  be disconnected. As for the remainder of this section, the reference is \cite{oneil} for more details. But from the definition, if $x\in \mathcal M$, then also $-x \in \mathcal M$, hence we will be mainly interested in the projective quotient of $\mathcal M$. 
As the action of the antipodal map is clearly isometric, the quotient will inherit in an obvious way many of the properties of $\mathcal M$.

\begin{df}
A \emph{model space}\index{model space} $\mathbb M$ is a
space $\mathbb M := \mathcal  M / \{\pm \mathrm{Id} \}$,
where $\mathcal M$ is a pseudo-sphere, together with the induced metric from $b|_{\mathcal M}$.
\end{df}

It is obvious that the projective quotient of $\overline{\mathcal M}$ is anti-isometric to $\mathbb M$, and will thus be denoted by $\overline{\mathbb M}$. Moreover, in the case where $\mathcal{M}$ is not connected, it has two connected components, which correspond under the antipodal map $x \mapsto -x$. In particular, $\mathbb M$ is connected.

By construction, $\mathbb M $ is a subset of the projective space $\R\P^n$.
 Indeed, topologically $\mathbb M $ can be defined as
 $$\mathbb M=\mathrm P\{x\in\R^ {n+1}\,|\,b(x,x)>0\}~,$$
 or the same definition with $>$ replaced by $<$, depending on the case.
Hence any choice of an affine hyperplane in $\R^{n+1}$ which does not contain the origin will give an affine chart of the projective space, and the image of $\mathbb M$ in the affine chart will be an open subset of an affine space of dimension $n$.

Recall that $\operatorname{PGL}(n+1,\R)$, the group of \emph{projective transformations}\index{projective transformation} (or homographies)\index{homography},  is the quotient of 
$\operatorname{GL}(n+1,\mathbb{R})$ by the non-zero scalar transformations. 
Even if $\mathcal{M}$ is not connected, an isometry of $\mathcal{M}$ passes to the quotient only if it is an element of $\mathrm{O}(p,q)$.
 Hence $\operatorname{Isom}(\mathbb{M})$, the isometry group of $\mathbb{M}$, is $\mathrm{PO}(p,q)$, the quotient of 
$\operatorname{O}(p,q)$ by $\{\pm\mathrm{Id}\}$ (indeed elements of $\operatorname{O}(p,q)$ have determinant equal to $1$ or $-1$).

\subsection{Lines and pseudo distance} \label{subsec lines pseudodistance}

\paragraph{Pseudo distance on $\mathcal M$.}

A geodesic $c$ of $\mathcal M$ is the non-empty intersection 
of $\mathcal M$ with a linear plane. This intersection might not be  connected.
%As a particular case of the characterization of totally geodesic submanifolds, a geodesic $c$ of $\mathcal M$  is a connected component of the intersection of $\mathcal M$  with  a linear plane of the ambient space (when the intersection is non-empty). 
The geodesic $c$ of $\mathcal M$ is said to be
\begin{itemize}
\item \emph{space-like} if $b(\dot c, \dot c)>0$,
\item \emph{light-like} (or \emph{null}) if $b(\dot c,\dot c)=0$,
\item \emph{time-like} if $b(\dot c,\dot c)<0$.
\end{itemize} 

An anti-isometry sends space-like lines onto time-like lines.
For example, assume that the restriction of $b$ to a plane $P$ of coordinates $(x_1,x_2)$ has the form $x_1^2+x_2^2$. Then the curve $$\{ x\in P\,|\,b(x,x)=1 \}=\{ x_1^2+x_2^2=1 \}$$ is a space-like geodesic. On the other hand, $b'=-b$ has the form 
$-x_1^2-x_2^2$ and thus $$\{ x\in P\,|\,b'(x,x)=-1 \}=\{ x_1^2+x_2^2=1 \}$$ is time-like.
Actually, up to a global isometry of the ambient space, a non-light-like geodesic of $\mathcal M$ can be written 
in one the the following forms:
\begin{enumerate}[label=a.\arabic*]
\itemrow {$\{x\in P\,|\,b(x,x)=1\}$}  {where } {$b|_P(x,x)=x_1^2+x_2^2$ ;} \label{form1}
\itemrow {$\{x\in P\,|\,b(x,x)=-1\}$}  {where } {$b|_P(x,x)=-x_1^2-x_2^2 $ ;} \label{form2}
\itemrow {$\{x\in P\,|\,b(x,x)=1\}$}  {where } {$b|_P(x,x)=x_1^2-x_2^2 $ ;}\label{form3}
\itemrow {$\{x\in P\,|\,b(x,x)=-1\}$}  {where } {$b|_P(x,x)=x_1^2-x_2^2$ .}\label{form4}
\end{enumerate}

%\begin{equation}\label{eq:type line}
%\begin{array}  {ll}
%1. \,\{x\in P\,|\,b(x,x)=1\} & \text{where } b|_P(x,x)=x_1^2+x_2^2~,\\
%2. \, \{x\in P\,|\,b(x,x)=-1\} & \text{where } b|_P(x,x)=-x_1^2-x_2^2~, \\
%3. \, \{x\in P\,|\,b(x,x)=1\} & \text{where } b|_P(x,x)=x_1^2-x_2^2~, \\
%4. \, \{x\in P\,|\,b(x,x)=-1\} & \text{where } b|_P(x,x)=x_1^2-x_2^2~.
%\end{array}
%\end{equation}

Up to a change of coordinates, a non-light-like geodesic is a circle or a branch of hyperbola.
If the geodesic is light-like, it is an affine line in the ambient space.
It follows from these descriptions that:

\begin{fact} Every pseudo-sphere $\mathcal M$  is \emph{geodesically complete}\index{geodesically complete}: each maximal (for the inclusion) geodesic is defined on the entire real line.
\end{fact}

Let $x,y$ be on a connected component of a geodesic $c$ of $\mathcal{M}$, and let us parametrize by $\gamma:[0,1]\to c$ a subsegment of $c$ such that  $\gamma(0)=x$ and $\gamma(1)=y$. The
\emph{pseudo distance}\index{pseudo distance} is  
\begin{equation}\label{eq:d lor}
 \tilde{d}(x,y)=\int_0^1 \sqrt{|b(\dot \gamma(t),\dot \gamma{(t)})|}dt~.\end{equation}

If $x$ and $y$ lie on a circle, there are actually two  choices for the subsegment of $c$ joining them.
We choose the one which gives the smaller distance.
In particular, in this case $\tilde{d}(x,y)\in [0, \pi]$. Moreover if $c$ is light-like, then $\tilde{d}(x,y)=0$.  
We have  $\tilde{d}(x,x)=0$ and $\tilde{d}(x,y)=\tilde{d}(y,x)$, but in general $\tilde{d}$ is not a distance.\footnote{It is a distance when $\mathcal M$ is Riemannian.}

Considering the cases in \ref{form1}-\ref{form4}, %\eqref{eq:type line},
explicit parameterizations of the circle or the hyperbola give:
\begin{enumerate}[label=b.\arabic*]
\item $b(x,y)=\cos (\tilde{d}(x,y))$ ; \label{formb1}
\item $- b(x,y)= \cos (\tilde{d}(x,y))$ ;
\item $b(x,y)=\cosh (\tilde{d}(x,y))$ ;
\item $-b(x,y)=\cosh (\tilde{d}(x,y))$ .
\end{enumerate}

%\begin{equation}\label{eq: b cos}
%\begin{split}
%1.& \,b(x,y)=\cos (\tilde{d}(x,y))~,\\
%2.&\, - b(x,y)= \cos (\tilde{d}(x,y))~, \\
%3.&\, b(x,y)=\cosh (\tilde{d}(x,y))~, \\
%4.&\, -b(x,y)=\cosh (\tilde{d}(x,y))~.
%\end{split}
%\end{equation}

\paragraph{Projective distance on $\mathbb M$.}

 We will be mainly interested in non-parametrized geodesics of $\mathbb M$.

\begin{df}
A \emph{line} (resp. \emph{plane}, \emph{hyperplane}) of $\mathbb M$ is the projective quotient  of a non-empty  intersection of $\mathcal M$ with a 2-dimensional  (resp. 3-dimensional, $n$-dimensional) linear vector space. 
\end{df}

A straightforward property  is that the image of a line (resp. plane, hyperplane)  of $\mathbb M$ in an affine chart is the intersection of the image of $\mathbb M$ with an affine line (resp. plane, hyperplane) of $\R^n$.

We say that a line of $\mathbb{M}$ is \emph{space-like} (respectively \emph{time-like}, \emph{light-like}) if its lift is a space-like (resp. time-like, light-like)\index{space-like line}\index{time-like line}\index{light-like line} geodesic of $\mathcal M$.
But also the lines can be classified with respect to the quadric given by the projective quotient of the isotropic cone.

\begin{df}
The projective quotient $\P \mathcal I(b)$ of the isotropic cone 
is called the \emph{absolute}\index{absolute}.\footnote{The absolute corresponds to what is  sometimes called the \emph{boundary at infinity} of $\mathbb{M}$. We will avoid this terminology  in the present paper, because some confusion may occur from the fact that in an affine chart, the boundary at infinity may be at infinity or not. Note that an absolute is often the boundary at infinity of two model spaces.} %$\partial_\infty\mathbb M$ of $\mathbb M$. 
\end{df}

\begin{df}\label{df:projective type}
A line  of $\mathbb M$ is called
\begin{itemize}
\item \emph{hyperbolic} if it intersects %$\partial_\infty \mathbb M$
the absolute in two distinct points, 
\item  \emph{parabolic} if it intersects %$\partial_\infty \mathbb M$
the absolute in one point,
\item \emph{elliptic}  it it does not intersect the absolute. % $\partial_\infty \mathbb M$. 
\end{itemize}
\end{df}

%We will produce below an evidence that the notions of light-like geodesics and parabolic lines coincide.XXXX
%
%Note that if $c$ is light-like, then the distance in Equation \eqref{eq:d lor} between two (distinct) points $x,y$ on $c$ is always $0$. 
%If $c$ is parabolic, then $c$ intersects $\partial_{\infty} \mathbb M$ in a single point $p$, $[x,y,p,p]=1$ and
%the distance in the sense of the right-hand side of Equation \eqref{eq:d hilb} is also $0$. A contrario, if $c$ is non-light like or non parabolic, then $d(x,y)\not=0$. In particular, parabolic lines and light-like lines coincide.

\begin{lemma}
Parabolic and light-like lines coincide.
\end{lemma}
\begin{proof}
A parabolic line is the projective quotient of a plane $P$ containing a unique line belonging to the isotropic cone of $b$.
But a light-like geodesic of $\mathcal M$ is the intersection of $\mathcal{M}$ with a
plane on which the restriction of $b$ is degenerate. 
\end{proof}

 Recall that an anti-isometry exchanges space-like and time-like lines, but it
preserves the  type of a line given by Definition~\ref{df:projective type}.
Actually the lines in $\mathbb M$ whose lifts are described in \ref{form1}-\ref{form4} have the following corresponding type:

\begin{enumerate}[label=c.\arabic*]
\item space-like  elliptic,
\item time-like elliptic,
\item time-like  hyperbolic,
\item space-like  hyperbolic.
\end{enumerate}

%\begin{equation*}
%\begin{split}
%1.& \,\mbox{space-like  elliptic},\\
%2.&\, \mbox{time-like elliptic}, \\
%3.&\, \mbox{time-like  hyperbolic},  \\
%4.&\, \mbox{space-like  hyperbolic}.
%\end{split}
%\end{equation*}

%
%1. space-like  elliptic,
%2. time-like elliptic,
%3. time-like  hyperbolic, 
%4.  space-like  hyperbolic.

If $x,y$ are on a line of $\mathbb M$, we will define a ``distance'' between them, which roughly speaking corresponds to the pseudo distance between the corresponding lifts. But one has to choose the lifts carefully.

\begin{df}
Let $x,y$ be on a line $l$ of $\mathbb{M}$. The \emph{projective distance}\index{projective distance} $d(x,y)$ is 
\begin{itemize}
\item $0$ if $l$ is parabolic,
\item $\tilde{d}(\tilde{x},\tilde{y})$ if $l$ is hyperbolic, and $\tilde{x},\tilde{y}$ are lifts of $x$ and $y$ on the same branch of hyperbola,
\item  $\tilde{d}(\tilde{x},\tilde{y})$ if $l$ is elliptic,  and $\tilde{x},\tilde{y}$ are lifts of $x$ and $y$ such that  $\tilde{d}(\tilde{x},\tilde{y})$ is the smallest.\footnote{In particular, in this case $d(x,y)\in [0,\pi/2]$.}
\end{itemize}
\end{df}

With the help of the cross-ratio, one can give a more direct definition of the projective distance.
Recall that if four points $x,y,q,p$ are given on the projective line $\mathrm \R\mathrm{P}^1\cong\R\cup\{\infty\}$, then there exists a unique projective transformation $h$ with $h(x)=\infty$, $h(y)=0$ and $h(q)=1$. Then the \emph{cross-ratio}\index{cross-ratio} can be defined as
$$[x,y,q,p]=h(p)\in \R\cup \{\infty\}\,.$$ Of course,  
$[x,y,q,p]=\infty$ if $p=x$, $[x,y,q,p]=0$ if $p=y$ and $[x,y,q,p]=1$ if $p=q$.
More explicitly, a formula is:
$$[x,y,q,p]=\frac{x-q}{y-q}\frac{y-p}{x-p}~. $$
Note that this last definition extends to complex numbers, i.e. the definition of the cross-ratio extends to $ \mathbb{C}\mathrm{P}^1$.

\begin{lemma}\label{lem:distances}
Given two distinct points $x,y\in \mathbb M$ on a line $l$, then $l$ intersects the absolute  at two points $\operatorname{I},\operatorname{J}$ (may be not distinct, may be in  $\mathbb{C}\P^1$), and 
\begin{equation}\label{eq:d hilb}d(x,y)=\left|\frac{1}{2}\ln [x,y,\operatorname{I},\operatorname{J}]\right|~.\end{equation}
Here $\ln$ denotes the branch of the complex logarithm $\ln:\mathbb C\setminus\{0\}\to\{z\in{\mathbb{C}}:\operatorname{Im}(z)\in(-\pi,\pi]\}$.
\end{lemma}
\begin{proof}
If $l$ is parabolic, then $\operatorname{I}=\operatorname{J}$,  
$[x,y,\operatorname{I},\operatorname{J}]=1$ and the result follows.

Let us consider the case when $l$ is elliptic. A lift of the line containing $x$ and $y$ is a connected component of the intersection of a $2$-plane $P$ with $\mathcal M$. The restriction of the bilinear form $b$ to the plane $P$ has necessarily signature $(+,+)$ or $(-,-)$, otherwise $P$ would meet the isotropic cone.  Without loss of generality, let us consider the case $(+,+)$, and let us  introduce coordinates such that
$$b(x,x)=x_1^2+x_2^2 $$
 and  $c=P \cap \mathcal M=\{ x\in P\mid b(x,x)=1\}$. In the complex plane, the equation
$x_1^2+x_2^2=0$  is solved by the two complex lines\footnote{In the text, we will use brackets to designate a projective   equivalence class.  } 
\begin{equation*}\label{eq:conj pt}\operatorname{I}=\left[\begin{matrix} i \\ 1 \end{matrix}\right] \mbox{ and } \operatorname{J}=\left[\begin{matrix} -i \\ 1 \end{matrix}\right]~.\end{equation*}
For more about the points $\operatorname{I},\operatorname{J}$
% defined by \eqref{eq:conj pt},
  see e.g. \cite{richter}.

%,  it intersects $\partial_\infty\mathbb M$ at two conjugate complex points $p$ and $q$, and we will see that $[x,y,q,p]$ is a complex number of modulus $1$. 
%
%Let us prove the case when $l$ is elliptic, the hyperbolic case being similar.
%
% The line $x+\lambda y$, $\lambda\in \mathbb{C}$, will meet the isotropic cone of $b$ at $p=-b(x,y)\pm i \sqrt{1-b(x,y)^2}$ and $q=\bar p$. 
% 
% 
% 
% Using \eqref{eq: b cos}, we obtain $z/\bar z=e^{\pm 2 i d(x,y)}$. 
% Then \eqref{eq:complex cross ratio} leads to the result.
% 
%
%{\color{red}{ Another question/curiosity is: what is the right setting to do all this? I mean, one should embed $P^1\R$ into $P^1\C$, then extend the quadratic form $b$ to a complex quadratic form in the obvious way, and take $\mathcal M$ still defined in the same way? Maybe understanding this helps to understand the case of euclidean space?}
%
%.  Then $[x,y,p,q]$ is a complex number of modulus $1$, }
% 

  Let us parametrize $c$ as 
$( \cos t,  \sin t)$, and let us choose 
$t_x$ and $t_y$ such that 
 $d(x,y)=|t_x-t_y|$.
In the affine chart $x_2=1$, the line $\operatorname{I}$ (resp. $\operatorname{J}$, $x$, $y$) is represented by the point  
$i$ (resp. $-i$, $ \operatorname{cotan} t_x$,  $ \operatorname{cotan} t_y$). Then we compute:
\begin{eqnarray*}
 \frac{1}{2} \ln [x,y,\operatorname{I},\operatorname{J}]  &=&  \frac{1}{2} \ln [\operatorname{cotan} t_x,\operatorname{cotan} t_y,i,-i]  \\
%&=& \left|\frac{1}{2i}\ln \left( \frac{\operatorname{cotan} t_x-i}{ \operatorname{cotan} t_y-i}\frac{\operatorname{cotan} t_y+i}{\operatorname{cotan} t_x+i}   \right) \right| \\
&=& \frac{1}{2}\ln \left( \frac{i\operatorname{cotan} t_x+1}{ i\operatorname{cotan} t_y+1}\frac{i\operatorname{cotan} t_y-1}{i\operatorname{cotan} t_x-1}   \right)  \\
&=& \frac{1}{2} \ln \left( \frac{ i\operatorname{cotan} t_x+1}{i\operatorname{cotan} t_x-1}   \right) - \frac{1}{2} \ln \left( \frac{ i\operatorname{cotan} t_y+1}{i\operatorname{cotan} t_y-1}  \right)\\
&=&    \operatorname{arcoth} (i \operatorname{cotan} t_x) -  \operatorname{arcoth} (i \operatorname{cotan} t_y) =  it_x-it_y =i(t_x-t_y)\,,
\end{eqnarray*}
where the last equation follows from
$\cos(x)=\cosh(ix),\, \sin(x)=({1}/{i})\sinh(ix)$, therefore $i \operatorname{cotan} x=\operatorname{arcoth}(ix)$. This shows that, if $l$ is elliptic, then the expression inside the modulus of \eqref{eq:d hilb} is purely imaginary. Taking the modulus we obtain
$$\left|\frac{1}{2} \ln [x,y,\operatorname{I},\operatorname{J}]\right|=d(x,y)\,.$$

If $l$ is hyperbolic,
it intersects the absolute in two distinct points $\operatorname{I},\operatorname{J}$, such that $\operatorname{I}$ and $\operatorname{J}$ do not separate $x$ and $y$. Similarly to the previous case, one can assume $b(x,x)=x_1^2-x_2^2$ on $P$ and $c=P \cap \mathcal M=\{ x\in P\mid b(x,x)=-1\}$. Picking the lifts of $x$ and $y$ on the same branch of the hyperbola of $c$, which is parameterized by $(\cosh t,\sinh t)$, one sees that in the affine line $x_2=1$ the line $\operatorname{I}$ (resp. $\operatorname{J}$, $x$, $y$) is represented by the point  
$-1$ (resp. $1$, $ \tanh t_x$,  $ \tanh t_y$). Hence by an analogous computation,
$$
 \frac{1}{2} \ln [x,y,\operatorname{I},\operatorname{J}]  =  \frac{1}{2} \ln [\operatorname{tanh} t_x,\operatorname{tanh} t_y,-1,1] 
%&=& \left|\frac{1}{2i}\ln \left( \frac{\operatorname{cotan} t_x-i}{ \operatorname{cotan} t_y-i}\frac{\operatorname{cotan} t_y+i}{\operatorname{cotan} t_x+i}   \right) \right| \\
= \frac{1}{2} \ln \left( \frac{ \operatorname{tanh} t_x+1}{\operatorname{tanh} t_x-1}   \right) - \frac{1}{2} \ln \left( \frac{ \operatorname{tanh} t_y+1}{\operatorname{tanh} t_y-1}  \right)\\
=t_x-t_y\,.
$$
In this case $[x,y,\operatorname{I},\operatorname{J}]$
is a real number, and the formula holds in a way analogous to the elliptic case.
\end{proof}

\subsection{Usual model spaces}

We will often consider the symmetric bilinear form $b$ on $\R^{n+1}$ in a standard form, i.e. such that the standard basis of $\R^{n+1}$ is an orthonormal basis for $b$, and if the signature of $b$ is $(p,q)$, then the quadratic form is positive on the first $p$ vectors of the standard basis, and negative on the other $q$ vectors. In this case, we will denote the bilinear form by $b_{p,q}$.

\paragraph{Euclidean and Elliptic spaces.}

In our notation  $\R^{n+1}$ is the real coordinate space of dimension $(n+1)$, and we will use the notation $\E^{n+1}$ for the  Euclidean space of dimension $(n+1)$, that is, $\R^{n+1}$ endowed with 
$b_{n+1,0}$. 
The \emph{elliptic space}\index{elliptic space}  $\Ell^n$ is the projective space of dimension $n$ endowed with the spherical metric inherited from $\S^n$, namely:
$$\Ell^n=\{x\in \R^{n+1} \,|\, b_{n+1,0}(x)=1\}/\{\pm \mathrm{Id}\}\,.$$
It is a  compact Riemannian manifold of dimension $n$. Any affine chart corresponds to a central projection of the sphere onto a hyperplane. All its lines are space-like  elliptic, in particular the distance induced by the Riemannian structure is given by the projective distance \eqref{eq:d hilb}.

\paragraph{Minkowski, Hyperbolic and de Sitter spaces.}

The space $\R^{n+1}$  endowed with $b_{n,1}$ is called the \emph{Minkowski space}\index{Minkowski space} and denoted by $\M^{n+1}$.
 An affine hyperplane $Q$ of direction $P$ in $\M^{n+1}$
 is called 
 \begin{itemize}
 \item space-like, if $(b_{n,1})|_P$ is positive definite;
 \item light-like or null if $(b_{n,1})|_P$  is degenerate;
 \item time-like if $(b_{n,1})|_P$  has signature $(n-1,1)$. 
 \end{itemize}
  (See Figure~\ref{fig:plans mink} for the 
 $n=2$ case.) 
 
 \begin{SCfigure}%\begin{center}
 \includegraphics[scale=0.5]{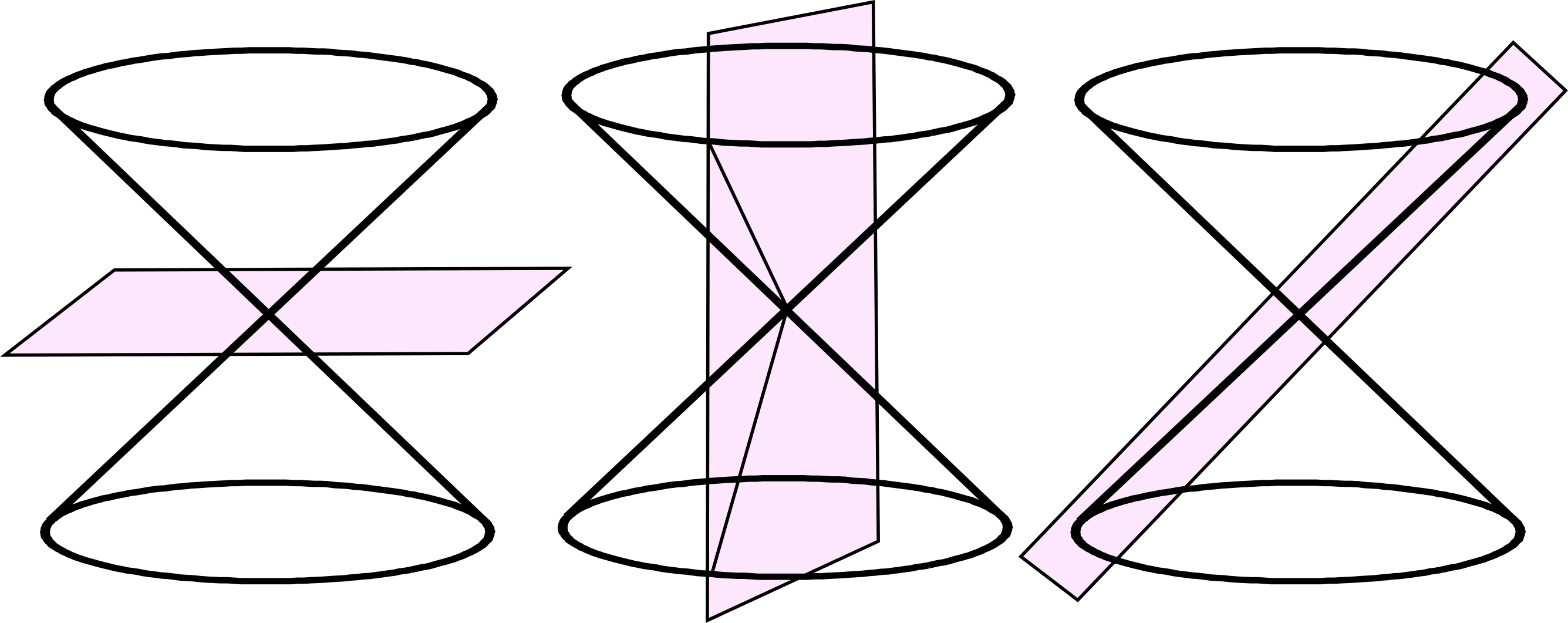}
 \caption{\label{fig:plans mink}  A plane in $\M^3$ is space-like (resp. time-like, light like) if its intersection with the isotropic cone of $b_{2,1}$ is reduced to a point (resp. two lines, a single line). 
}%\end{center}
 \end{SCfigure}

The \emph{hyperbolic space}\index{hyperbolic space}  is 
$$\H^n=\mathcal{H}^n / \{\pm \mathrm{Id}\}\,,$$
with
$$\mathcal{H}^n= \{x\in \R^{n+1} \,|\, b_{n,1}(x)=-1\}\,,$$
endowed with the induced metric.
The hyperbolic space is  a  simply connected complete Riemannian surface of sectional curvature $-1$. All its line are space-like hyperbolic, in particular the distance induced by the Riemannian structure is given by the projective distance \eqref{eq:d hilb}.

The other pseudo-sphere of Minkowski space is the \emph{de Sitter space} \index{de Sitter space}
$$\dS^n=\{x\in \R^{n+1} \,|\, b_{n,1}(x)=1 \}/\{\pm \mathrm{Id} \} $$
endowed with the induced metric. On any tangent plane to $\dS^n$, the restriction of $b_{n,1}$ has Lorentzian signature, i.e. it is non-degenerate with  signature $(n-1,1)$. De Sitter space has constant curvature $1$, and  is homeomorphic to $\S^{n-1}\times\R$.

Let us focus our attention to the case $n=2$. The de Sitter plane $\dS^2$ has space-like elliptic, null and time-like hyperbolic  lines.  This is more easily seen by taking affine models. The most common model is 
$\{x_3=1\}$. Then $\H^2$  is  the unit open disc, and $\dS^2$ is the complement of the closed disc in the  plane, see Figure~\ref{fig:H2dS2}.\footnote{This is the famous \emph{Klein model} for the hyperbolic plane.} 

\begin{SCfigure}%\begin{center}
\includegraphics[scale=0.5]{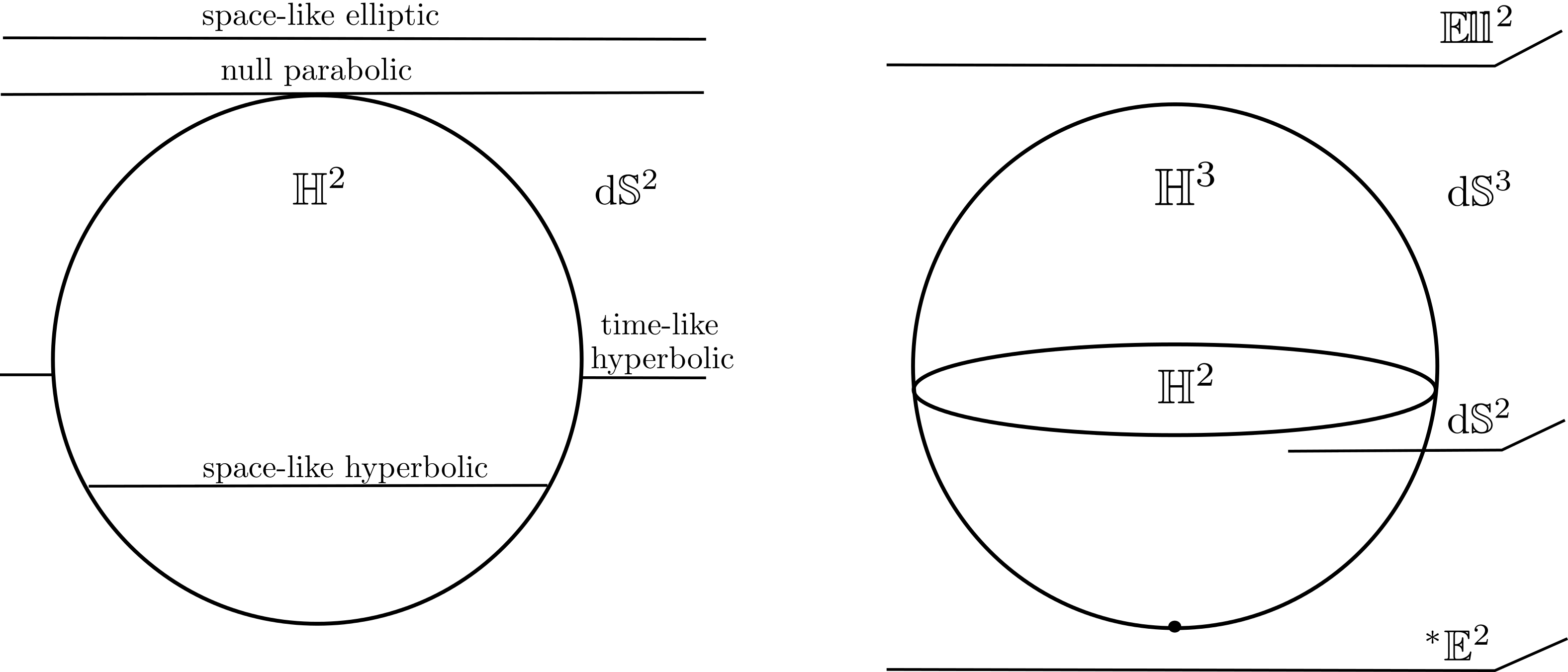}
\caption{\label{fig:H2dS2}On the left: types of lines in an affine model of the hyperbolic and de Sitter planes. On the right: types of planes in an affine model of the hyperbolic and de Sitter spaces.}%\end{center}
\end{SCfigure}

By any affine transformation,  any ellipse instead of the circle gives another affine model --- these affine transformations are given by the choice of any space-like affine plane in $\M^3$ which does not contain the origin. 
Taking as affine chart a time-like or a light-like affine plane of $\M^{3}$, we obtain other models of the hyperbolic plane and de Sitter plane:

\begin{fact}The convex side of any  conic of the plane (i.e. a parabola, ellipse or a hyperbola) endowed with the projective distance is an affine model of $\H^2$ or $\overline{\H^2}$. The other side is an affine model of $\dS^2$ or $\overline{\dS^2}$.\end{fact}

In dimension $3$, we obtain:
%
%
%The generalization to the higher dimension is immediate: the hyperbolic space is
%$$\H^3=\{x\in \R^{3,1} | b(x)=-1\}/\{\pm Id\} $$
%endowed with the induced metric. It is a simply connected complete Riemannian space of curvature $-1$.
% The de Sitter space is 
%$$dS^3=\{x\in \R^{3,1} | b(x)=1 \}/\{\pm Id \} $$
%and it is a simply connected Lorentzian manifold of curvature $1$. We have
\begin{fact}Any convex side of a non-degenerate quadric of $\R^3$ (i.e. an ellipsoid, an elliptic paraboloid or a two-sheeted hyperboloid) endowed with the projective distance is a model for $\H^3$ or $\overline{\H^3}$. The other side is a model for $\dS^3$ or $\overline{\dS^3}$.\end{fact}

It is then immediate that  a  plane of $\H^3$ is a hyperbolic plane, and that a  plane of $\dS^3$ may be isometric to $\dS^2$, to the elliptic plane $\Ell^2$, or to the co-Euclidean plane $^*\E^2$ (see \ref{cominkowski plane}), see Figure~\ref{fig:H2dS2}.

%\begin{figure}[h]\begin{center}
%\includegraphics[scale=0.5]{planesHdS.jpg}
%\caption{\label{fig:H3dS3} {\color{red}{Sorry, I introduced the symbol $\dS$, so that should be changed in figures as well}}}\end{center}
%\end{figure}

\paragraph{Anti-de Sitter space.}

By considering the non-degenerate symmetric bilinear form $b_{n-1,2}$ on
$\R^{n+1}$, for $n\geq 2$, we define the \emph{Anti-de Sitter space}\index{anti-de Sitter space}
%Observe that Anti-de Sitter space is not anti-isometric to de Sitter space unless $n=2$.}
 as $\AdS^n=\mathcal{A}d\mathcal{S}^n / \{\pm \mathrm{Id}\}$, with
$$\mathcal{A}d\mathcal{S}^n=\{x\in \R^{n+1} \,|\, b_{n-1,2}(x)=-1 \} $$
endowed with the metric induced from $b_{n-1,2}$. %(As a warning for the reader, anti-de Sitter space is not anti-isometric to the de Sitter space.) 
Anti-de Sitter space is a Lorentzian manifold of curvature $-1$. It is not simply connected as it contains a closed time-like geodesic.

\begin{remark}{\rm 
It is readily seen that $\AdS^2$ is anti-isometric to 
 $\dS^2$. This is not the case for  $n>2$.
}\end{remark}

Let us now focus on the $n=3$ case.
First note that $\mathcal{A}d\mathcal{S}^3$ is anti-isometric to
$b_{2,2}^{-1}(1)$. 
% Note fist that  $\AdS^3$ is anti-isometric to  
%\begin{equation}\label{eqads3}\{x\in \R^{2,2} \,|\, b_{2,2}(x)=1 \} / \{\pm \mathrm{Id}\}\,.\end{equation} 
Up to affine transformations, $\AdS^3$ has two models described as follows:
\begin{fact}Any part of $\R^3$ delimited by a double-ruled quadric (i.e. a hyperbolic paraboloid or a one-sheeted hyperboloid) endowed with the projective distance is a model of $\AdS^3$ or $\overline{\AdS^3}$.\end{fact}

A plane in $\AdS^3$ can be isometric to $\H^2$, to the co-Minkowski plane $^*\M^2$ (cf. \ref{cominkowski plane}), or  to $\AdS^2$, see Figure~\ref{adslineetplan}.
%, and Figure~\ref{fig:AdS3} for the different types of lines in the hyperboloid model.

\begin{SCfigure}%\begin{center}
\includegraphics[scale=0.5]{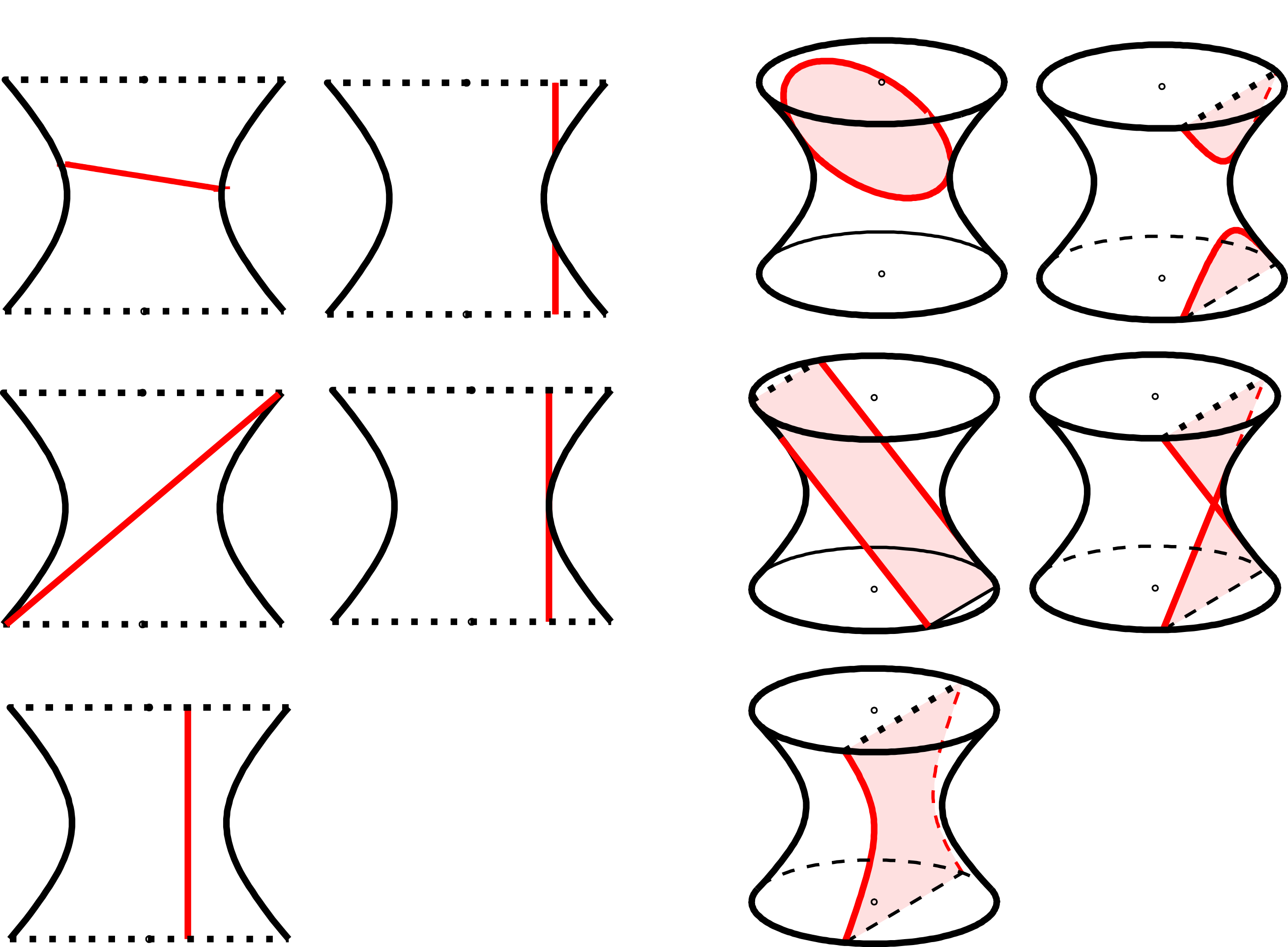}
\caption{\label{adslineetplan} On the left-hand side, some {non-affinely equivalent} lines in an affine model of $\AdS^2$. First row on the left: space-like hyperbolic, second row: light-like parabolic, third row: time-like elliptic. On the right-hand side, some {non-affinely equivalent} planes in an affine model of $\AdS^3$. First row: hyperbolic, second row: co-Minkowski, third row: $\AdS^2$.
}%\end{center}
\end{SCfigure}

%
%\begin{figure}\begin{center}
%\includegraphics[scale=0.3]{ads-lines.jpg}
%\caption{\label{fig:AdS3} Some lines in $\AdS$. First row: space-like hyperbolic, second row: light-like parabolic, third row: time-like elliptic.}
%\end{center}\end{figure}
%
%
%\begin{figure}\begin{center}
%\includegraphics[scale=0.3]{ads-plans.jpg}
%\caption{\label{fig:AdS32} Some planes in $\AdS$. First row: hyperbolic, second row: co-Minkowski, third row: (anti-isometric to) de Sitter.}
%\end{center}\end{figure}

\subsection{Convex sets and duality}\label{duality}

\paragraph{Dual cones.} A non-empty subset $K$ of $\R^{n+1}$ is \emph{convex} if it contains the segment between any two of its points (note that this notion only uses the affine structure of $\R^{n+1}$).
% In the following, we will always ely assume that $K$ is also a closed set.
A  convex cone\index{convex! cone} $\C$ of $\R^n$ is a convex set such that  $\lambda \C \subset \C$ for any $\lambda>0$. We will also suppose that $\C$ is pointed:   the only linear subspace it contains is $\{0\}$. In general, we will also assume implicitly that the cone is closed.

The dual $(\R^{n+1})^*$ of  $\R^{n+1}$ is  the set of linear forms on $\R^{n+1}$, and  is naturally endowed with a vector space structure of dimension $(n+1)$. Note that the notion of convexity also holds in  $(\R^{n+1})^*$. Given a convex cone $\C$ in $\R^n$, its \emph{dual}\index{dual} is 

\begin{equation}\label{eq:dual}\C^*=\{x^* \in (\R^n)^* \,|\, x^*(y) \leq 0, \forall y\in \C\}~. \end{equation}

It is readily seen that the dual of a convex cone is a convex cone, and that
\begin{equation}\label{eq:inclusion cone}
A\subset B \Rightarrow B^*\subset A^ *~.
\end{equation}

Recall that a \emph{support space} of a closed convex set $K$ is a half space containing $K$ and bounded by an affine hyperplane $H$. If moreover $K\cap H \not= \emptyset$, $H$ is a \emph{support plane}\index{support plane} of $K$.
A convex set is also the intersection of its support spaces. If $K=\C$ is a convex cone, then its support planes are linear hyperplanes. 
Now if $x^*$ is a non trivial linear form on $\R^{n+1}$, then its kernel is a linear hyperplane, and 
$\{y \,|\, x^*(y)\leq 0 \}$ is a half-space bounded by this hyperplane. So 
$\C^*$ can be interpreted as the set of
 support planes of $\C$. As a convex set is the intersection of its support spaces, 

\begin{equation}\label{eq:dualdual}\C=\{ x \in \R^n \,|\, y^*(x)\leq 0, \,\forall y^* \in \C^* \}~. \end{equation}

From the formal duality between the definitions \eqref{eq:dual} and \eqref{eq:dualdual}, it is readily seen that 
\begin{equation}\label{autoduality}(\C^*)^*=\C~. \end{equation}

   If $\R^{n+1}$ is endowed with a non-degenerate symmetric bilinear form $b$, then $b$ induces an isomorphism between 
 $\R^{n+1}$ and $(\R^{n+1})^*$. Once $b$, and thus the isomorphism, are fixed, we will still denote by $\C^*$ the image of $\C^*$ in $\R^{n+1}$:

$$\C^*=\{x \in \R^n \,|\, b(x,y) \leq 0, \forall y\in \C\}~.$$
See Figure~\ref{fig:dualmink} for an example with $b=b_{1,1}$.
Let us denote by $ \mathrm{C}$ the projective quotient of 
$\lambda \C$, $\forall \lambda \in \R$, i.e.
$$\mathrm{C}= (\C\cup-\C)  / \{\pm \mathrm{Id}\}~.$$

\begin{remark}{\rm
We could have defined as a dual for $\C$ the convex set
$$\C^{*'}=\{x \in \R^n \,|\, b(x,y) \geq 0, \forall y\in \C\}~,$$
as both definitions agree in the projective quotient: $\mathrm{C}^*=\mathrm{C}^{*'}$.}
\end{remark}

\begin{SCfigure}%\begin{center}
   \includegraphics[width=10cm]{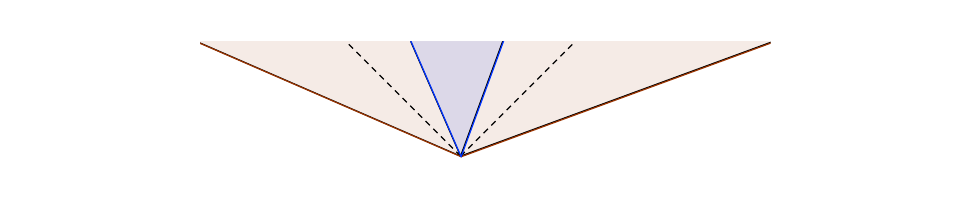}
\caption{\label{fig:dualmink} Dual cones in the Minkowski plane.}
%\end{center}
\end{SCfigure}

\paragraph{Convex sets in model spaces.}

Let $\mathcal M$ be a pseudo-sphere,  $\C$ be a convex cone, and suppose that  $\mathcal M \cap \C$ is non-empty.
Recall that $\mathbb M$ is the projective quotient of $\mathcal M$, that is, $\mathbb M := \mathcal  M / \{\pm \mathrm{Id} \}$.

\begin{df}
A \emph{convex set}\index{convex!set} of $\mathbb M$ is
 the intersection $\mathbb M \cap \mathrm{C}$. 
\end{df}

The image of a convex cone  in an affine chart may not be an affine convex set. For example
the intersection of $\{z^2\leq x^2+y^2\}$ 
with a vertical plane.
 But  the image of a convex cone in an affine chart is an affine convex set if the vector hyperplane parallel to the hyperplane which defines the affine chart meets the cone only at $\{0\}$.
Moreover in an affine chart, even if $ \mathrm{C}$ is convex, 
$ \mathbb M\cap \mathrm{C}$ is not necessarily an affine convex set (the  image of $\mathbb M$ in the affine chart may be non convex).

Also, a convex set in $\mathbb{M}$ is not necessarily 
 a geodesically convex subset of $\mathbb{M}$.
For example, take any convex cone in $\R^3$ which contains the isotropic cone of $b_{2,1}$ in its interior, and is on one side of the horizontal plane. In the affine chart given by $\{x_3=1\}$, this gives a convex set in $\dS^2$, which is an affine convex set. But it contains the absolute  in its interior, so it is not a geodesically convex set in $\dS^2$ (see the big triangle in the left-hand side of Figure~\ref{fig:dualH2dS2}).

%\item Even if $\P C$ is a convex affine set in an affine chart, a convex hypersuface in $\P M$ does not necessary bounds a geodesically convex set if $\P M$
%Even if $\P C$ is a convex affine set in an affine chart, a convex hypersuface in $\P M$ does not necessary bounds a geodesically convex set in  convex set (cas de sitter douteux, expliquer intersecte boule avec hyhperbolide).
%\end{itemize}  
 
%\begin{figure}\begin{center}
%\includegraphics[scale=0.7]{convexeds.jpg}
%\caption{\label{fig-non conv}$S$ is a convex hypersurface in the de Sitter plane $\operatorname{dS}^2$, which bounds an affine convex set, but it does not bound a geodesically convex set in $\operatorname{dS}^2$. Faux ? il suffit que les geodesiques aillent a l'infini?? il faut imposer une distance comme dans la sphere ??}
%\end{center}\end{figure}

\paragraph{Duality.}

The notion of duality for convex sets  follows easily from the one for convex cones.
Let us focus on the more relevant cases. 
Let $K$ be a convex set in $\H^n$, defined by a convex cone $\C(K)$.
Using the bilinear form $b_{n,1}$ to identify the ambient space $\R^{n+1}$ with its dual,
 the support planes of $\C(K)$ are time-like or light-like,  and the boundary of $\C(K)^*$ is made of space-like or light-like lines, and hence its intersection $K^*$ with $\dS^n$ is not empty. The set  $K^*$ is a convex set in de Sitter space, the \emph{dual} of $K$.  As the lines of $\C(K)$ are orthogonal to the support planes of $\C(K)^*$ by construction, $K^*$ is \emph{space-like}: it has only space-like or light-like support planes.
Conversely, if  $K$ is a space-like convex set  in $\dS^n$, then we can define in the same way its dual convex set $K^*$, which is a convex set in $\H^n$.

If $K$ is a convex set in $\Ell^n$, we obtain similarly a dual convex set $K^*$ which is in $\Ell^n$. Analogously, the dual of a space-like convex set $K$ in $\AdS^3$ is a space-like convex set $K^*$ in $\AdS^3$. 
In any case, we have  $(K^*)^*=K$.

\begin{SCfigure}%\begin{center}
\includegraphics[scale=0.15]{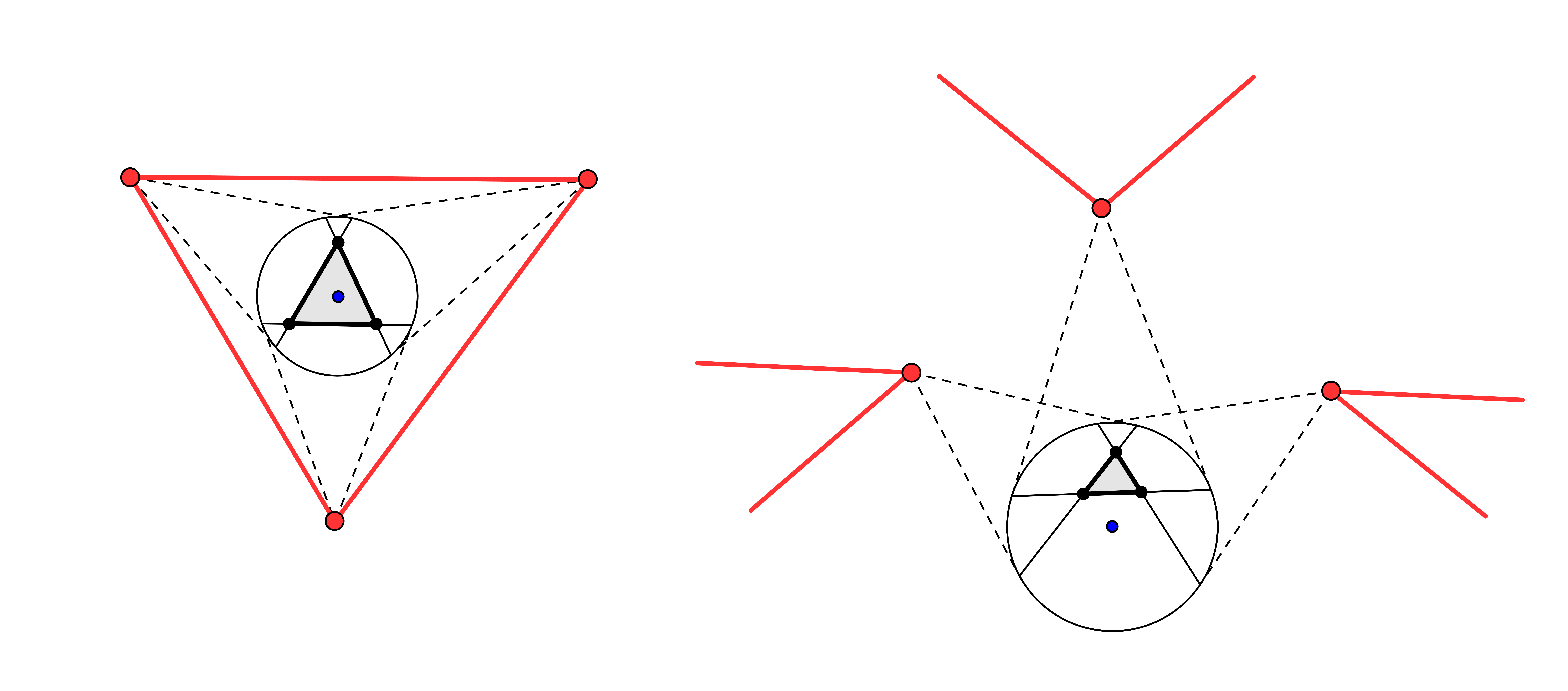}
\caption{\label{fig:dualH2dS2}Triangles in $\dS^2$  dual to triangles in $\H^2$ in an affine chart. Note that on the right-hand side, the dual triangle is not the convex hull in the affine chart of the vertices.}
%\end{center}
\end{SCfigure}

 The duality is also a duality between points and hyperplanes. 
For  $x\in\mathcal{M}=b^{-1}(1)$, $x^*$ is the intersection (if non-empty) of
$\mathcal{N}=b^{-1}(-1)$ with the hyperplane of $\R^{n+1}$ orthogonal to $x$ for $b$.

This duality points/hyperplane can also be seen in an affine manner.
 Let us begin with the hyperbolic case. A hyperplane in $\H^n$ meets the absolute along a topological sphere $S$ of dimension $n-2$. Then the point dual to the hyperplane is the apex of the cone formed by the lines tangent to the absolute along $S$. This is the usual notion of 
polarity transformation with respect to a (proper) quadric, which is an affine notion, see Figure~\ref{fig:coniques} and Figure~\ref{fig:dualh2}. 

 We provide an argument, in the $n=2$ case, to fix ideas. In the double cover, a  plane of $\H^2$ corresponds to the intersection of the pseudo-sphere $$\mathcal H^2=\{x\in \R^3\,|\, b_{2,1}(x,x)=-1\}$$ with a time-like plane $ P$ of $\M^{3}$, which meets the isotropic cone $\mathcal I(b_{2,1})$ of $\M^3$ along two light-like vectors $v_1,v_2$. The light-like planes tangent to the isotropic cone containing $v_1,v_2$ meet along a line directed by a space-like vector $v$. It is easy to see that $v$ is orthogonal to $v_1$ and $v_2$, and hence to $ P$.

 The same holds in the Anti-de Sitter space: the dual of a space-like plane is a point of $\AdS^3$, and vice versa, see Figure~\ref{fig:dualityads}.

 The duality points/hyperplanes suffices to recover the dual of a convex set, see Figure~\ref{fig:dual-convexe}.   Also,
 it gives the following description of de Sitter space.
 
 \begin{fact}
 The de Sitter space $\dS^n$ is the space of (unoriented) hyperplanes of $\H^n$.
 \end{fact}
%\begin{figure}
%   \begin{minipage}[r]{.3\linewidth}
%     \includegraphics[scale=0.3]{dualH2.jpg}
%     \caption{\label{fig:dualh2}The hyperplane containing $P$ meets the absolute along a hypersphere $S$. The point $P^*$ is the common point of all the lines tangent to the absolute at points of $S$.}
%   \end{minipage} \hfill
%   \begin{minipage}[l]{.65\linewidth}
%     \includegraphics[scale=0.3]{coniques.jpg}
%     \caption{\label{fig:coniques} The dual point of a line in different affine model of the hyperbolic plane.}
%   \end{minipage}
%\end{figure} 

\begin{SCfigure}
 \includegraphics[scale=0.3]{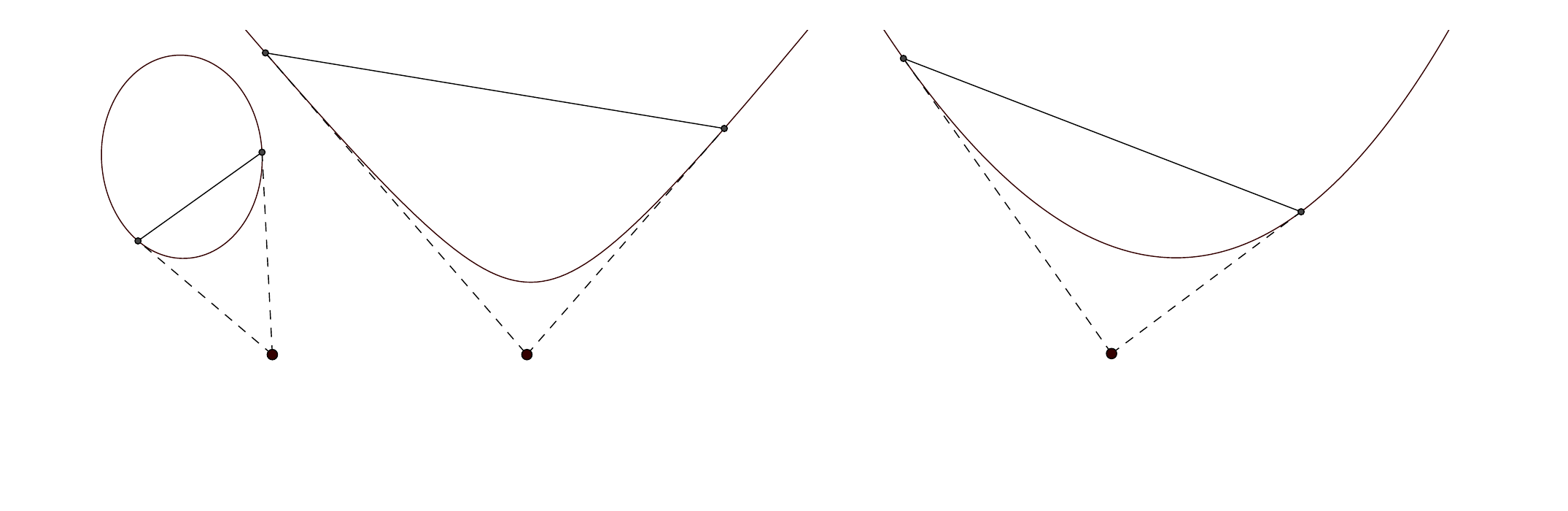}
     \caption{\label{fig:coniques} The dual point of a line in different affine models of the hyperbolic plane.}
\end{SCfigure}

\begin{SCfigure}
       \includegraphics[scale=0.6]{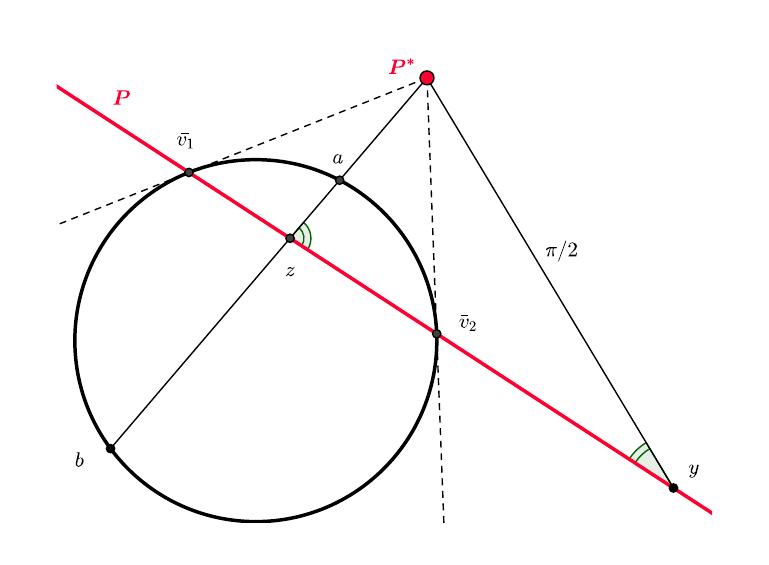}
     \caption{\label{fig:dualh2}The hyperplane containing $P$ meets the absolute along a hypersphere $S$. The point $P^*$ is the common point of all the lines tangent to the absolute at points of $S$.}
\end{SCfigure}

\begin{SCfigure}
   \includegraphics[width=5cm]{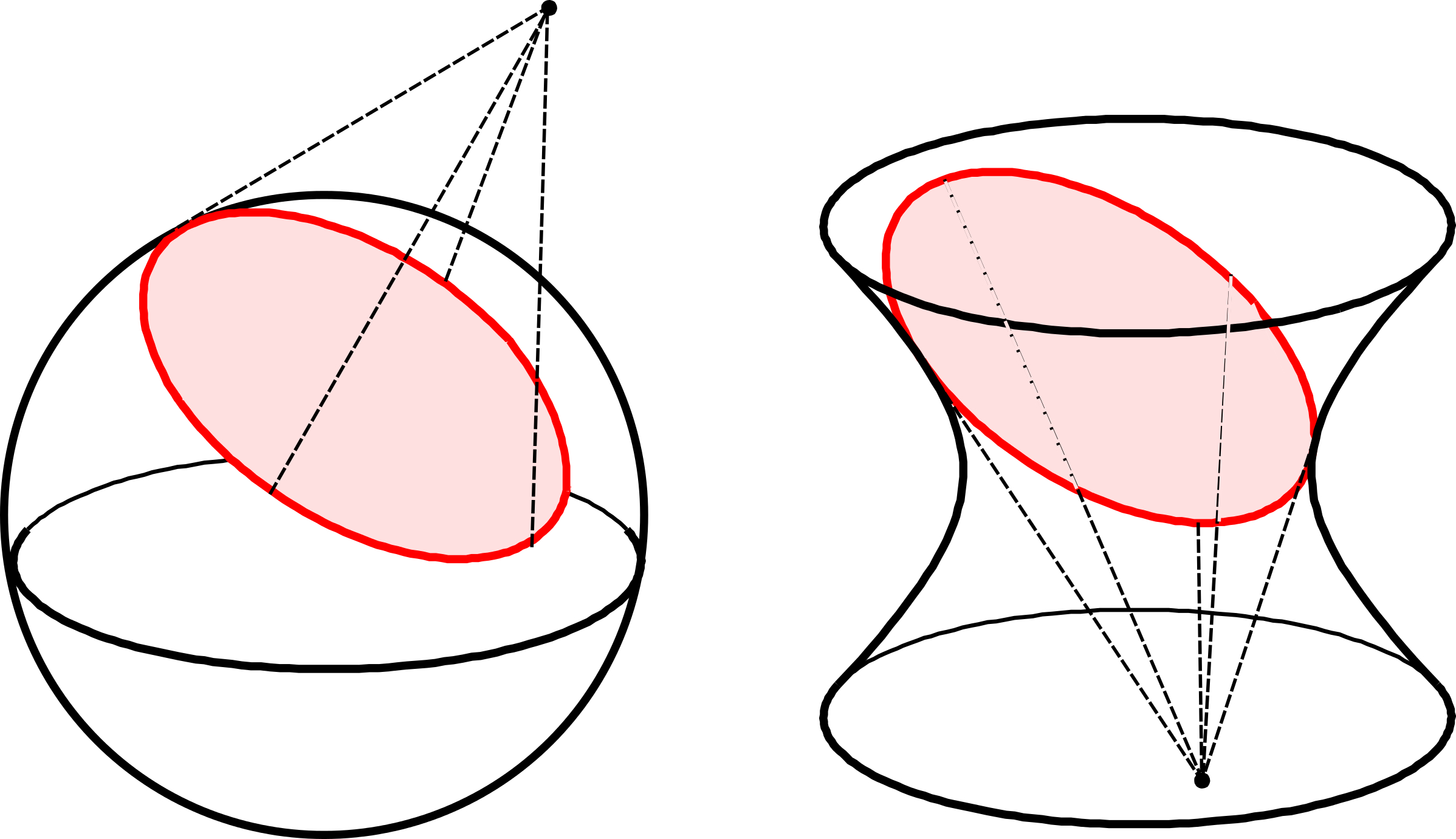}
\caption{\label{fig:dualityads} The dual of a plane of the hyperbolic space and of Anti-de Sitter space in an affine model. }
\end{SCfigure}

 \begin{SCfigure}
%\begin{center}
\includegraphics[scale=0.15]{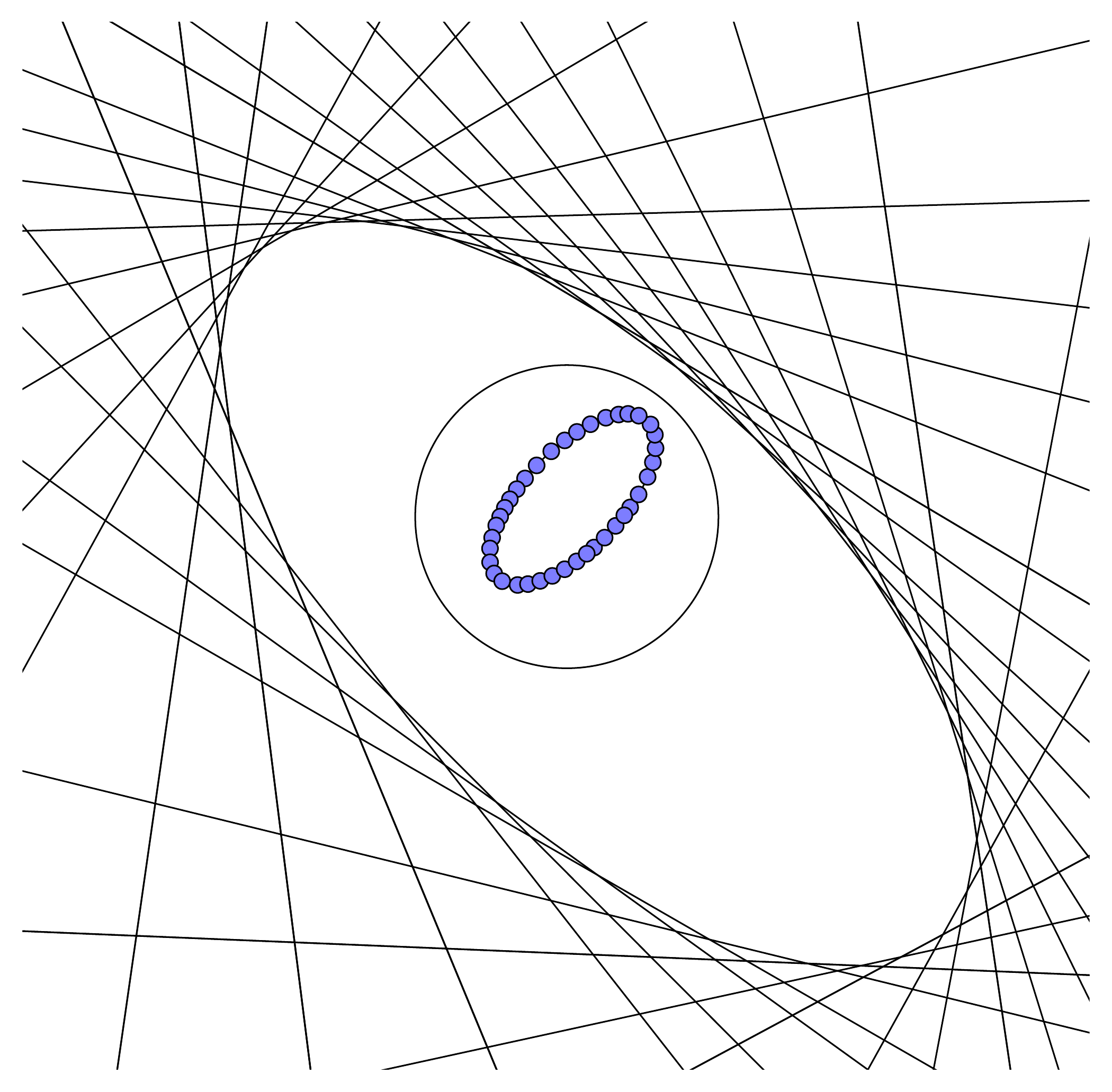}
\caption{\label{fig:dual-convexe} The dual of a convex set with respect to the sphere is the envelope of the hyperplanes dual to its boundary points.}
%\end{center}
\end{SCfigure}

%
% 
%\begin{floatingfigure}[r]{2.5cm}
%   \includegraphics[width=2cm]{dualityH.jpg}
%\caption{\label{fig:dualityH}
%All the lines from $P^*$ meet $P$ orthogonally. Using a straighfroward generalization of the Hilbert distance  \eqref{eq:d hilb}, one can see that if $z\in P$ and the line $P^*z$ meets the absolute in two points $a,b$ as in , then the points $P^*azb$ forms a \emph{harmonic division} of the line $P^*z$, i.e.
%$[P^*,a,z,b]=1$.}
%\end{floatingfigure}
% 

% 
%\begin{wrapfigure}{l}{0.05\textwidth}
%    \centering
%    \includegraphics[width=0.25\textwidth]{dualityH.jpg}
%\end{wrapfigure}

\begin{remark} \label{remark duality}{\rm
In the elliptic plane $\Ell^2$, it is well known that 
the duality can also be expressed in a metric way: the plane $x^*$ is the set of points at distance $\pi/2$ from $x$. 
This is readily seen because $x^*$ is the projective quotient of a linear plane $P$ in $\E^3$, and $x$ is the projective quotient of a vector orthogonal to $P$. The orthogonality for $b_{3,0}$ immediately gives that the projective distance is equal to $\pi/2$ due to \ref{formb1} in Subsection \ref{subsec lines pseudodistance}.

A similar argument leads to the following in $\AdS^ 3$:  the plane $x^*$ dual to the point $x$  is the set of points at distance $\pi/2$ from $x$ (see Figure~\ref{fig:dualityads}).

The same computation also occurs in de Sitter space,
if one considers the duality in the following way.
As before, let $P$ be a hyperplane in $\H^n$ and let $P^*$ be its dual point in $\dS^n$. Actually $P^*$ is also dual to a time-like hyperplane of $\dS^n$, the one defined by the same affine hyperplane as $P$, and that we still denote by $P$. Then in $\dS^n$, the distance between $P^*$ and $P$ is $\pi/2$, see Figure~\ref{fig:dualh2}. 
% An immediate way to see this is because for any $y\in P$, $0=b(\tilde{P^*},\tilde y)=\cos (d(P^*,y))$, where $b$ is the bilinear form on $\M^{n+1}$. {\color{red}{I would prefer to write it for the AdS case before, and then mention the dS case which is more weird}}

}\end{remark}

\paragraph{Comments and references} \small
\begin{itemize}
\item When the pseudo-metric is written under the form given by Lemma~\ref{lem:distances}, it is usually called a \emph{Hilbert} or \emph{Cayley--Klein} metric\index{Hilbert metric}\index{Cayley--Klein metric}. It can be defined on any convex sets, and not only on the ones bounded by quadrics as here,  see e.g. \cite{papadopoulos} for more details.
\item As Euclidean space is made only of parabolic lines, it is sometimes called \emph{parabolic space}\index{parabolic space}. In dimension $2$, the projective distance can be defined directly on a quadric: this is the Poincar\'e quadric geometry \cite{poincare,AP2}. In particular, this justifies the term  ``parabolic''.
\item Remarkably, in dimension $3$, the sphere $\mathcal{S}^3$ has a group structure, namely $\operatorname{SU}(2)$, with the determinant as the quadratic form on the ambient space. Similarly, $\mathcal{A}d\mathcal{S}^3$ is $\operatorname{SL}(2,\mathbb{R})$, with minus  the determinant as the
quadratic form on the ambient space. Note that the Lie algebra $\mathfrak{su}(2)$ (resp. $\mathfrak{sl}(2,\R)$) together with its Killing form is naturally identified with $\E^3$ (resp. $\M^3$). 
\item The following fact  comes readily from the definition of 
duality:
\emph{any line from $x$ meets $x^ *$ orthogonally}.
This is useful in practice, as the  affine models  are certainly not conformal (to the Euclidean metric), and computations of angles may be cumbersome, but orthogonality is easily seen, see Figure~\ref{fig:dualh2}. See Figure~\ref{fig:hexagon} for an application.  In particular, the Klein model of hyperbolic space is not a conformal model, as a striking difference with the other famous \emph{Poincar\'e model}, which will not be used in this survey.
\begin{SCfigure}%[h]%\begin{center}
\includegraphics[scale=0.6]{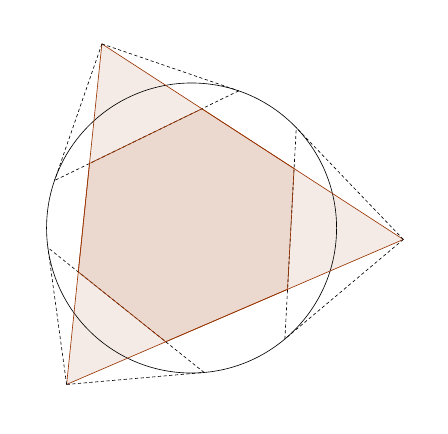}
\caption{\label{fig:hexagon}  A hyperbolic right-angled hexagon is an affine triangle, hence the corresponding moduli space has dimension $3$. Right-angled hexagons  are fundamental pieces to construct compact hyperbolic surfaces \cite{thb,Thurcour1}.}
%\end{center}
\end{SCfigure}
\item 
%\paragraph{A remark about angles and lengths} 
The duality in $\S^2$ induces also a correspondence between angles and length, that has a great importance as many statements have a straightforward dual analogue. The most basic ones are  spherical trigonometric laws for triangles. 

 This is also true  for a general duality with respect to a symmetric non-degenerate bilinear form $b$. Actually this property is contained in the definition of angle. Let $P_1,P_2$  be two hyperplanes in $\mathcal{M}$, and let $x\in P_1\cap P_2$. The outward unit normals $n_1$ and $n_2$ to $P_1$ and $P_2$ in $T_x\mathcal{M}$ define two points on a pseudo-sphere $S$ of $T_x\mathcal{M}$, which is naturally isometric to a pseudo-Euclidean space. The (exterior dihedral) \emph{angle} between $P_1$ and $P_2$   is the pseudo distance on $S$ between those two points. This does not depend on the choice of the point $x\in P_1\cap P_2$. {But $T_x\mathcal{M}$ is also identified with a hyperplane in the ambient $\R^{n+1}$, the one orthogonal to $x$ for $b$.
 In particular, $n_1$ and $n_2$ are identified with points in a pseudo-sphere $\mathcal{N}$,  dual to $P_1$ and $P_2$. The distance between those points of $\mathcal{N}$ is exactly the distance between the points in $S$, hence equal to the angle between $P_1$ and $P_2$.}

The point is  that in $T_x\mathcal{M}$, $n_1$ and $n_2$ may belong to two different pseudo-spheres, but it is possible 
to define a pseudo distance between points belonging to two different pseudo-spheres defined by a same symmetric non-degenerate bilinear form. We will not use this in the present paper. 
For trigonometric laws for hyperbolic/Anti-de Sitter triangles, see \cite{Thurcour1,Cho09}.
\item Let $\operatorname{Isom}_0$ be the connected component of the identity of the isometry group. 
There are famous identifications between  $\operatorname{Isom}_0(\H^2)$ and  $\operatorname{PSL}(2,\mathbb{R})$, and between
the absolute and $\R \mathrm{P}^1$. In dimension $3$, there is an identification between
$\operatorname{Isom}_0(\H^3)$  and $\operatorname{PSL}(2,\mathbb{C})$. The last group also acts on the absolute, which is naturally identified with $\mathbb{C}\mathrm{P}^1$. 
In contrast, 
$\operatorname{Isom}_0(\AdS^3)=\operatorname{PSL}(2,\mathbb{R})\times \operatorname{PSL}(2,\mathbb{R})$, and the absolute is naturally identified with $\mathbb{R}\mathrm{P}^1\times \mathbb{R}\mathrm{P}^1$.

\end{itemize}

\normalsize

\section{Degenerate cases}

\subsection{Geometries}

A \emph{geometry}\index{geometry (as $(X,G)$ structure)}  is a pair $(X,G)$ where $X$ is a manifold and $G$ is a Lie group acting transitively by diffeomorphisms on $X$.
The model spaces we introduced in Section~\ref{sec:mode spaces} are geometries, with $X=\mathbb{M}$ and $G=\operatorname{PO}(p,q)$. Note that here the group $G$ is a subgroup of the group of projective transformations that can be characterized in two ways: $G$ leaves $\mathbb{M}$ invariant in the projective space and $G$ is the isometry group of $\mathbb{M}$. %The couples $(\operatorname{M},G)$ may be called \emph{model geometries}. Sometimes the group $G$ will be implied, so for example we could call $\Ell^n$ the elliptic geometry.
We will now be interested by \emph{degenerate} model spaces. This means that they are defined by a degenerate bilinear form (or equivalently, their absolute is not a proper quadric in an affine chart).\footnote{In particular, the absolute will contain lines, called \emph{isotropic}\index{isotropic! line}.}

A degenerate model space gives a geometry  $(\mathbb{M},G)$, with $G$ the subgroup of the group of projective transformations that leaves $\mathbb{M}$ invariant. Equivalently, $G$ is a subgroup of the group of projective transformations that preserves
 a degenerate $(0,2)$ tensor on $\mathbb{M}$. But we will consider the geometries 
 $(\mathbb{M},H)$, with $H$ a proper subgroup of $G$. The choice of $H$ will be justified by a duality argument in this section, and next justified on the one hand by the process of degeneration introduced in Section \ref{sec geometric transition} and on the other hand by the definition of connections and volume forms in Section \ref{sec connection volume}.

Unlike the preceding section, we do not attempt to give a unified treatment of the degenerate geometries, but we focus on the Euclidean and Minkowski spaces, together with their dual spaces: the co-Euclidean space and the co-Minkowski space.

%
%
%So far we looked at the geometry obtained from pseudo-sphere of non-degenerate quadratic forms. 
%We will focus on degenerate cases that will be used in the next sections about degeneracy of convex surfaces. Moreover they give an extrinsic description of well-known dualities between convex sets of $\R^n$.

%A quadratic form $b_{p,q,k}$ is such that the standard basis of $\R^{n+1}$ is an orthogonal basis for $b_{p,q,k}$, and with  signature
%$(p,q,k)$: the quadratic form is positive on the first $p$ vectors, negative on the other $q$ vectors and null on the last $k$ vectors.
\subsection{Euclidean space}\label{sec dual eucl}

\paragraph{Euclidean space as a degenerate model geometry.}\index{Euclidean space}
 The Euclidean space $\E^n$ may be considered as a projective quotient of a pseudo-sphere $\Eu^n$ of $\R^{n+1}$ defined by the following degenerate bilinear form:  
\begin{equation}\label{eq:b01}b(x,y)
=x_{n+1}y_{n+1}~, \end{equation}
i.e.
$$\Eu^n=\{x \in \R^{n+1} \,|\, x_{n+1}^2=1 \}~.$$

Indeed, $\Eu^n/\{\pm\operatorname{Id} \}$
is the complement of a hyperplane in $\R\mathrm{P}^n$, i.e.  it is identified with the affine space $\mathbb{A}^n$ of dimension $n$. {Note that in this model, the absolute (the projective quotient of the isotropic cone of $b$) is the hyperplane at infinity. In particular, all the lines are
parabolic, and hence the projective distance is zero, so we do not recover the Euclidean metric. We need more information to recover Euclidean geometry from this model.}

 The group $G$ of projective transformations leaving $\Eu^n$ invariant is the group of transformation of the form

\begin{equation} \label{eq: rep eucl}
\left(
\begin{array}{ccc|c}
  
  & & & t_1 \\
   & \raisebox{-4.5pt}{{\huge\mbox{{$A$}}}}  & & \vdots  \\
  & & & t_n \\ \hline
  0 & \cdots & 0 & \lambda
\end{array}
\right)%~,
\end{equation}
where $A$ belongs to $\operatorname{GL}(n,\R)$ and $\lambda\neq 0$, quotiented by the subgroup of multiples of the identity, i.e. it is the group of transformations of the form
\begin{equation*} %\label{eq: rep eucl}
\left[
\begin{array}{ccc|c}
  
  & & & t_1 \\
   & \raisebox{-4.5pt}{{\huge\mbox{{$A$}}}}  & & \vdots  \\
  & & & t_n \\ \hline
  0 & \cdots & 0 & \lambda
\end{array}
\right]~.
\end{equation*}
 Of course one can always choose a representative with $\lambda=1$. Under the identification of $\Eu^n/\{\pm \operatorname{Id}\}$ with the affine space $\mathbb{A}^n$ given by taking the chart $\{x_{n+1}=1\}$, the action of $G$ is identified with the action on $\mathbb{A}^n\times\{1\}$ of matrices of the form

\begin{equation} \label{eq: rep eucl2}
\left(
\begin{array}{ccc|c}
  
  & & & t_1 \\
   & \raisebox{-4.5pt}{{\huge\mbox{{$A$}}}}  & & \vdots  \\
  & & & t_n \\ \hline
  0 & \cdots & 0 & 1
\end{array}
\right)~.
\end{equation}

Hence $(\Eu^n/\{\pm \operatorname{Id}\},G)$ is the affine geometry $(\mathbb{A}^n,\mathrm{GL}(n,\R)\rtimes\R^n)$. The Euclidean geometry $(\mathbb{\E}^n,\mathrm{Isom}(\E^n))$
is $(\Eu^n/\{\pm \operatorname{Id}\},H)$, where $H$ is the subgroup of $G$ such that $A$ in \eqref{eq: rep eucl} belongs to $\operatorname{O}(n)$.

%But something is missing in this definition, as we only get the affine space of dimension $n$ 
%Actually in this case it does not suffice to know the absolute $\partial_\infty \E^n$ to identify 
%the Euclidean space. The space $P^n\R\setminus \partial_\infty \E^n$ 
%%is only the definition of the affine space of dimension $n$, that can be also identified as the Minkowski space, as we will see in \ref{sec mink} ---as all the lines are parabolic, so we cannot hope to recover the metric.
%
%Instead of $\E^n$, let us consider its  group of isometries  $\operatorname{Isom}(\E^n)$, which is composed of  linear isometries, that form the group  $\operatorname{O}(n)$, and translations, that are identified with $\R^n$.  There is a faithful representation $\rho$ of $\operatorname{Isom}(\E^n)$ into the subgroup of $\operatorname{GL}(n+1,\R)$ which preserves $b$. This representation is defined as follows: for
% $M=(A,(t_1,\ldots,t_n))\in \operatorname{O}(n)\times \R^n$ an Euclidean isometry, 
%\begin{equation}
%\rho(M)=\left(
%\begin{array}{c|c}
%  \raisebox{-14pt}{{\huge\mbox{{$A$}}}} & t_1 \\[-3.5ex] & \vdots  \\
%  & t_n \\ \hline
%  0 \cdots 0 & 1
%\end{array}
%\right)~.
%\end{equation}
%Because of the $1$ in the right bottom corner, 
% $\rho$ decends to a faithful representation of $\operatorname{Isom}(\E^n)$ into  $\operatorname{PGL}(n+1,\R)$.

An element of the form \eqref{eq: rep eucl} also acts on $\R^n\times \{0\}$, and the action reduces to the action of $A$.
As  $A\in \operatorname{O}(n)$, passing to the projective quotient, the hyperplane at infinity of $\mathbb{A}^n$ is endowed with the  elliptic geometry $\Ell^{n-1}$. 
 Conversely, suppose that $(\mathbb{A}^n,H)$ is a geometry,  where $H$ is a group of affine transformations, whose action on the hyperplane at infinity  is the one of $\operatorname{PO}(n)$. Then necessarily the part $A$, in a representative   of the form \eqref{eq: rep eucl} for an element of $H$, must belong to $\operatorname{O}(n)$.
 In other words, the Euclidean geometry can be characterized as follows.
% 
%Moreover, $\rho(M)$ also acts on $\R^n\times\{0\}$, and the action is the one of $A\in \operatorname{O}(n)$. 
%So identifying $\operatorname{Isom}(\E^n)$ and $\rho(\operatorname{Isom}(\E^n))$, the Euclidean isometries extends to $\partial_\infty \E^n$ as isometries of $\Ell^{n-1}$. Moreover, any 
%isometry of $\Ell^{n-1}$ can be written as the projective quotient 
%of an element of $\operatorname{GL}(n+1,\R)$ of the form \eqref{eq: rep eucl}.

\begin{fact}\label{fact:euclidean}
The Euclidean space is the projective space minus  an elliptic hyperplane.
\end{fact}

%For example in two dimension:
%
%\begin{fact}
%The Euclidean plane is the projective plane minus an elliptic line.
%\end{fact}

\paragraph{Duality of convex sets.}
Recall that a \emph{convex body}\index{convex! body} of $\E^n$ is a compact convex set in $\E^n$ with non-empty interior.\footnote{For some authors, a convex body is only a compact convex set.}
We will add the following assumption.

\begin{df}
A convex body is \emph{admissible} if it contains the origin in its interior.
\end{df}

  Actually every convex body is admissible  up to a translation. It is suitable to consider $K$
in $\E^n\times \{1\} \subset \E^{n+1}$, in order to introduce 
   the cone $\mathcal{C}(K)$ in $\E^{n+1}$ over $K$:
$$\mathcal{C}(K)=\left\{\left. \lambda \binom x 1 \, \right| \,  x \in K, \lambda \geq 0 \right\}~.$$

Let $\mathcal{C}(K)^*$ be its dual in $\E^{n+1}$ for the scalar product $b_{n+1,0}$:
 \begin{equation}\label{eq:CK* eucl}\mathcal{C}(K)^*=\{(y,y_{n+1}) \in \R^{n+1}\, |\,
 b_{n+1,0}\left( (y,y_{n+1}),(x,x_{n+1})\right) \leq 0, \forall (x,x_{n+1})\in \mathcal{C}(K) \}~.\end{equation}

We will  denote by $K^*$ the intersection of 
$\mathcal{C}(K)^*$ with $\{y_{n+1}=-1\}$. We identify  $\{y_{n+1}=-1\}$ with $\E^n$, so  that $K^*$ is a closed convex set in $\E^n$. 
It is readily seen that 
\begin{equation}\label{eq:K* eucl} K^*=\{y\in \R^n \,|\, b_{n,0}( x,y) \leq 1, \forall x \in K \}~. \end{equation}

The expression 
\eqref{eq:K* eucl} corresponds to the affine duality with respect to the unit sphere: a point $rv$, with $v\in \mathcal S^n$, on the boundary of $K$ will correspond to a support plane of $K^*$ of direction orthogonal to $v$ and at distance $1/r$ from the origin. Compare  Figure~\ref{fig:dual-cercle} and Figure~\ref{fig:coEplan}. In particular, the dual $B_r$ of a ball centred at the origin with radius $r$ is the ball $B_{1/r}$ centred at the origin. 
%\begin{figure}[h]
%   \begin{minipage}[c]{.46\linewidth}
%     \includegraphics[scale=0.4]{dual-cercle.pdf}
%\caption{\label{fig:dual-cercle} Euclidean duality= projective duality/circle}
%   \end{minipage} \hfill
%   \begin{minipage}[c]{.46\linewidth}
%   \includegraphics[scale=0.1]{dualeuclidienplan.pdf}
% \caption{\label{fig:coEplan}Effects on the duality on point/plane.}
%   \end{minipage}
%\end{figure}

\begin{SCfigure}
\includegraphics[scale=0.5]{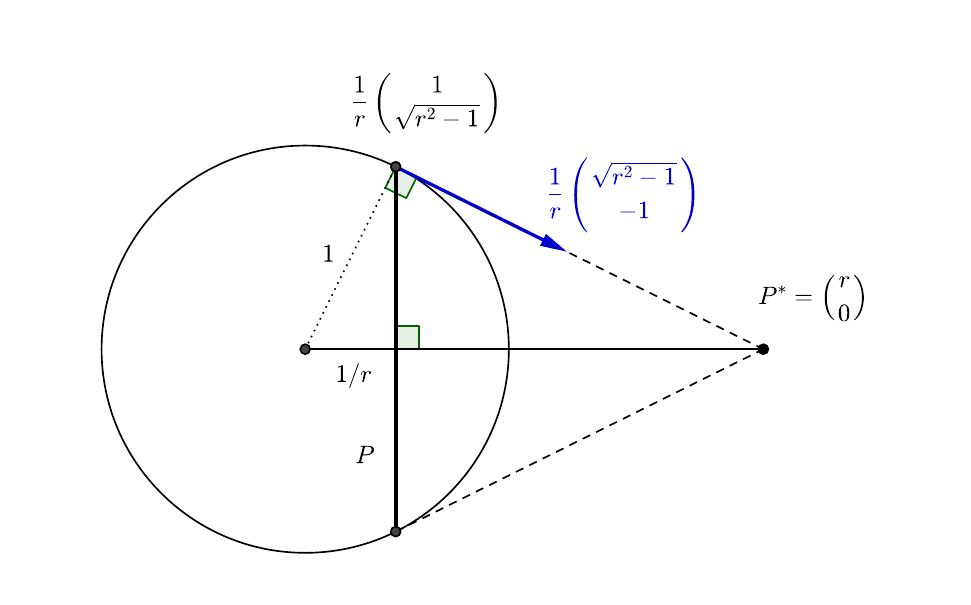}
\caption{\label{fig:dual-cercle} The Euclidean duality is the duality with respect to the unit sphere. }
\end{SCfigure}

\begin{SCfigure}
\includegraphics[scale=1.4]{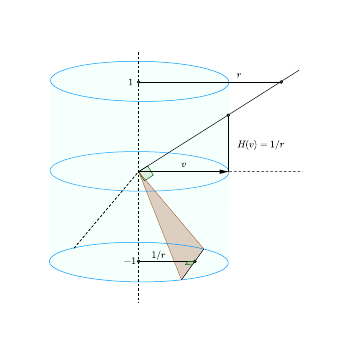}
 \caption{\label{fig:coEplan}The duality between points and lines in the Euclidean plane expressed in terms of orthogonality in a higher-dimensional Euclidean space.}
\end{SCfigure}

\begin{lemma}
The dual of an admissible convex body in Euclidean space is an admissible convex body in Euclidean space.
\end{lemma}
\begin{proof}
If $K$ is a convex body, by definition 
there exists $s,t>0$ such that $B_s\subset K \subset B_t$. Hence by \eqref{eq:inclusion cone}, $B_{1/t}\subset K^* \subset B_{1/s}$:
$K^*$ is bounded and contains the origin in its interior. 
\end{proof}

%
%\begin{figure}[h]\begin{center}
%\includegraphics[scale=0.4]{dual-cercle.png}
%\caption{\label{fig:dual-cercle} Euclidean duality= projective duality/circle}
%\end{center}\end{figure}
%
%\begin{figure}[t]\begin{center}
%\includegraphics[scale=0.1]{dualeuclidienplan.jpg}
% \caption{Effects on the duality on point/plane.}\label{fig:coEplan}
%\end{center}\end{figure}

\paragraph{Support functions.}\index{support function}

Let $H: \R^n\to \R$ be such that $\mathcal{C}(K)$ is the epigraph of $H$, i.e.
\begin{equation}\label{eq:ck H eucl}\mathcal{C}(K)=\{(x,x_{n+1})\in\R^{n+1} \,|\, x_{n+1} \geq H(x) \}~.\end{equation}

The function $H$ is convex, positive outside the origin, and homogeneous of degree $1$: 
$H(\lambda x)=\lambda H(x)$ for $\lambda>0$.
Hence using \eqref{eq:ck H eucl} and \eqref{eq:CK* eucl},
$$\mathcal{C}(K)^*=\{(y,y_{n+1}) \in \R^{n+1} \,|\,
b_{n,0}(y,x) \leq -y_{n+1}x_{n+1}, \forall (x,x_{n+1}), x_{n+1}\geq H(x) \}$$
i.e.
$$\mathcal{C}(K)^*=\{(y,y_{n+1}) \in \R^{n+1} \,|\,
 b_{n,0} (y,x) \leq -y_{n+1}H(x),\, \forall x\in \R^n \}$$ 
so $K^*$ is determined by $H$:
\begin{equation}\label{eq:supp fct euc}K^*=\{y \in \R^n \,|\, b_{n,0}( y,x) \leq H(x),\, \forall x\in \R^n  \}~. \end{equation}

\begin{df}
The function $H$ such that $\mathcal{C}(K)$ is the epigraph of $H$ is the \emph{support function} of $K^*$.
\end{df} 
 
 The support function  has also the following interpretation. 
Let $v$ be a unit vector of $\E^n\times \{0\}$. Then $H(v)$ is the distance in $\E^n \times \{-1\}$ between $\{0\}\times\{-1\}$ and the support plane of $K^*$ directed by $v$ (see Figure~\ref{fig:coEplan}). Hence
\eqref{eq:supp fct euc} expresses the fact that $K^*$ 
is  the envelope of its support planes. 

The following fact follows easily from the construction.

\begin{fact}
The support function provides a  bijection between the space of admissible convex bodies of $\R^n$  and the space of positive convex $1$-homogeneous functions on $\R^n$.
\end{fact}

%By homogeneity, $H$ is determined by its restriction, $h$, to the sphere $\S^{n-1}$, denoted by $h$ and called the \emph{restricted support function}.

%
%\textbf{Remark} 
%This suggests a projective definition of the Euclidean space $\E^n$: in 
%an affine chart, $\E^n$ is the affine $n$-dimensional space whose 
% absolute is at infinity and complex (given by $b$ in the hyperplane $\{x_{n+1}=0\}$ in $\R^{n+1}$). Actually by the Laguerre formula, angles can be recovered using the Hilbert distance. Up to fix some unit, the Euclidean distance can also be given by the Hilbert metric \cite{richter}.

\subsection{Co-Euclidean space}\label{coeucl}

\paragraph{Definition.}\index{co-Euclidean space}

On $\R^{n+1}$, let $b^*$ be the following degenerate bilinear form

$$b^*(x,y)=x_1y_1+\cdots+x_{n}y_{n}~. $$
Let $\cE^n$ be the unit sphere for $b^*$: $$\cE^n=\{x\in\R^{n+1}\,|\,b^*(x,x)=1 \}\,,$$ endowed with the restriction of $b^*$ to its tangent space. Topologically, $\cE^n$ is  $\S^n\times \mathbb{R}$. 
%We call a line  \emph{vertical} if it is of the form $\{x\}\times \R$.
%Vertical lines are light-like, and any section of $\mathcal M$ by a hyperplane not containing a vertical line is isometric to $\S^n$. 

\begin{df}
The space  $^*\E^n=\cE^n/\{\pm \mathrm{Id}\}$ is the \emph{co-Euclidean space}.
\end{df}

%An alternative description of $^*\E^n$ is the following.
The isotropic cone $\mathcal{I}(b^*)$  is the line $\tilde\ell=\{(0,\ldots,0,\lambda) \}$. Let 
$\ell=\mathrm{P}\tilde\ell$. Then $^*\E^n=\R \mathrm{P}^n\setminus\{\ell\}$.
Lines of $^*\E^n$ are parabolic if they contain $\ell$, and
elliptic otherwise.
%
% The projective distance can then be defined on $\P \R^{n+1}\setminus \{\ell\}$. In an affine chart, 
%we have $\R^n$ minus the point $\P\ell$, that may be at infinity. A line is parabolic if 
%it meets $\P\ell$, and space-like elliptic otherwise.

\paragraph{Duality.}

The co-Euclidean space is dual to the Euclidean space in the sense that it can be described as the set of hyperplanes of $\E^n$.
Let $P$ be an affine hyperplane in $\E^n$, and $v$  its unit normal vector, pointing towards the side of $P$ which does not contain the origin $0$ of $\E^n$.  Let  $h(v)$ be the distance from $0$ to $P$, i.e. $P$ has equation $b_{n,0}( \cdot,v) -h(v)=0$. 
The vector
$\binom{v}{h(v)}\in \R^{n+1}$ is   orthogonal in $\E^{n+1}$ to the linear hyperplane  containing $P \times \{-1\}$. Its projective quotient 
defines a point $P^*$ in $^*\E^n$, see Figure~\ref{fig:coEplan}.
One could also consider the other unit normal vector
$-v$ of $P$. The (signed) distance from the origin is then 
$-h(v)$, and the point $-\binom{v}{h(v)}$ has the same projective quotient as $\binom{v}{h(v)}$.

\begin{fact}
The co-Euclidean space $^*\E^n$ is the space of (unoriented) hyperplanes of $\E^n$.
\end{fact}

Conversely, an elliptic hyperplane $P$ of $^*\E^n$ (i.e. a hyperplane which does not contain $\ell$) is dual to a point $P^*$ of $\E^n$. 
The point $P^*$ is the intersection of all the hyperplanes $x^*$, for $x \in P$.
A co-Euclidean hyperplane of $^*\E^n$ (i.e. a hyperplane containing $\ell$) is dual to a point at infinity.

Let $K$ be an admissible convex body of $\E^n$, and 
let $H$ be its support function, i.e. $\mathcal{C}(K)^*$ is the epigraph of $H$.
By abuse of notation, let us also denote 
by $K^*$ the intersection of $\mathcal{C}(K)^*$ with
$\cE^n$. The set $K^*$ is the epigraph of the restriction $h$ of 
$H$ to $\cE^n\cap \{x_{n+1}=0\}$, that we identify with  $\mathcal{S}^{n-1}$. Note that  by homogeneity,  $H$ is determined by its restriction  to $\S^{n-1}$, which is actually $h$:
\begin{equation*}%\label{eq:one homegenous extension}
H(x)=\|x\|h(x/\|x\|)~,
\end{equation*} 
with $\|x\|=\sqrt{b_{n,0}(x,x)}$. 
\begin{fact}
The convex set $K^*$ of $\cE^n$ is the epigraph of $h$, the restriction to $\S^{n-1}$ of the support function of $K$.
\end{fact}

\begin{remark}\label{remark:angles coeucl}{\rm
 The duality between hyperplanes of $\E^n$ and points in $^*\E^n$  leads to the following relation between angles and length. If $P'$ is another affine hyperplane of $\E^n$, orthogonal to the unit vector $v'$, then the $^*\E^n$ segment between $P^*$ and $P'^*$ is elliptic (or equivalently, the restriction of $b^*$  is positive definite on any non-vertical hyperplane of $\R^{n+1}$). By the expression \ref{formb1} in Subsection \ref{subsec lines pseudodistance}, $\cos d(P^*,P'^*)= b^*(P^*,P'^*)$ and it is readily seen that the last quantity is equal to $b_{n,0}( v,v' )$. Thus,   the 
(exterior dihedral) angles between intersecting  affine hyperplanes of $\E^n$ are equal to the distance between their duals in $^*\E^n$.
}\end{remark}

\paragraph{A model geometry.}

Let $\mathrm{Isom}(b^*)$ be the subgroup of projective transformations that preserve $b^*$. There is a natural injective morphism
$^*:\mathrm{Isom}(\E^n) \to  \mathrm{Isom}(b^*)$ which is defined as follows.
Let $$P(v,h)=\{z \in \R^n\,|\, b_{n,0}(z,v)=h \}$$ be an affine hyperplane of $\E^n$, $v\in \S^{n-1}, h\in \R^*$. For $A\in O(n)$ and $\vec{t}\in \R^n$,
we have $$AP(v,h)+\vec{t}=P(Av,h+b_{n,0}(v,A^{-1}\vec{t})),$$ 
from  which we define
\begin{equation} \label{eq: etoile eucl}
\begin{array}{c}
^*  \\
\\
  \\
  \\ 
  \end{array}
  \hspace{-3mm}
\left[
\begin{array}{ccc|c}
  
  & & &  \\
   & \raisebox{-4.5pt}{{\huge\mbox{{$A$}}}}  & &\vec{t} \\
  & & &  \\ \hline
  0 & \cdots & 0 & 1
\end{array}
\right]~ := \left[
\begin{array}{ccc|c}
  
  & & & 0 \\
   & \raisebox{-4.5pt}{{\huge\mbox{{$A$}}}}  & & \vdots  \\
  & & & 0 \\ \hline
   & A^{-1}\vec{t} & & 1
\end{array}
\right]~.
\end{equation}

\begin{df} \label{defi isom coeucl}
The isometry group of the co-Euclidean space, $\mathrm{Isom}(^*\E^n)$, is the group of projective transformations of the form
\begin{equation}\label{eq:isom co}\left[
\begin{array}{ccc|c}
  
  & & & 0 \\
   & \raisebox{-4.5pt}{{\huge\mbox{{$A$}}}}  & & \vdots  \\
  & & & 0 \\ \hline
   & \vec{t} & & 1
\end{array}
\right]~ \end{equation}
for $A\in \operatorname{O}(n)$, $\vec{t}\in \R^n$.
\end{df}

Note that $\mathrm{Isom}(^*\E^n)$ is a proper subgroup of the group of isometries of $b^*$. For instance the latter  also contains \emph{homotheties} of the form
$[\operatorname{diag}(1,\ldots,1,\lambda)]$. They correspond to displacement along 
parabolic lines, see Figure~\ref{fig:isom coeucl}.

The co-Euclidean space is naturally endowed with a degenerate metric $g^*$, which is the restriction on $\cE^n$ of $b^*$, pushed down to the projective quotient, and coincides with the elliptic metric on the elliptic hyperplanes, and is zero on the parabolic lines. 

\begin{SCfigure}%[h]\begin{center}
     \includegraphics[scale=0.4]{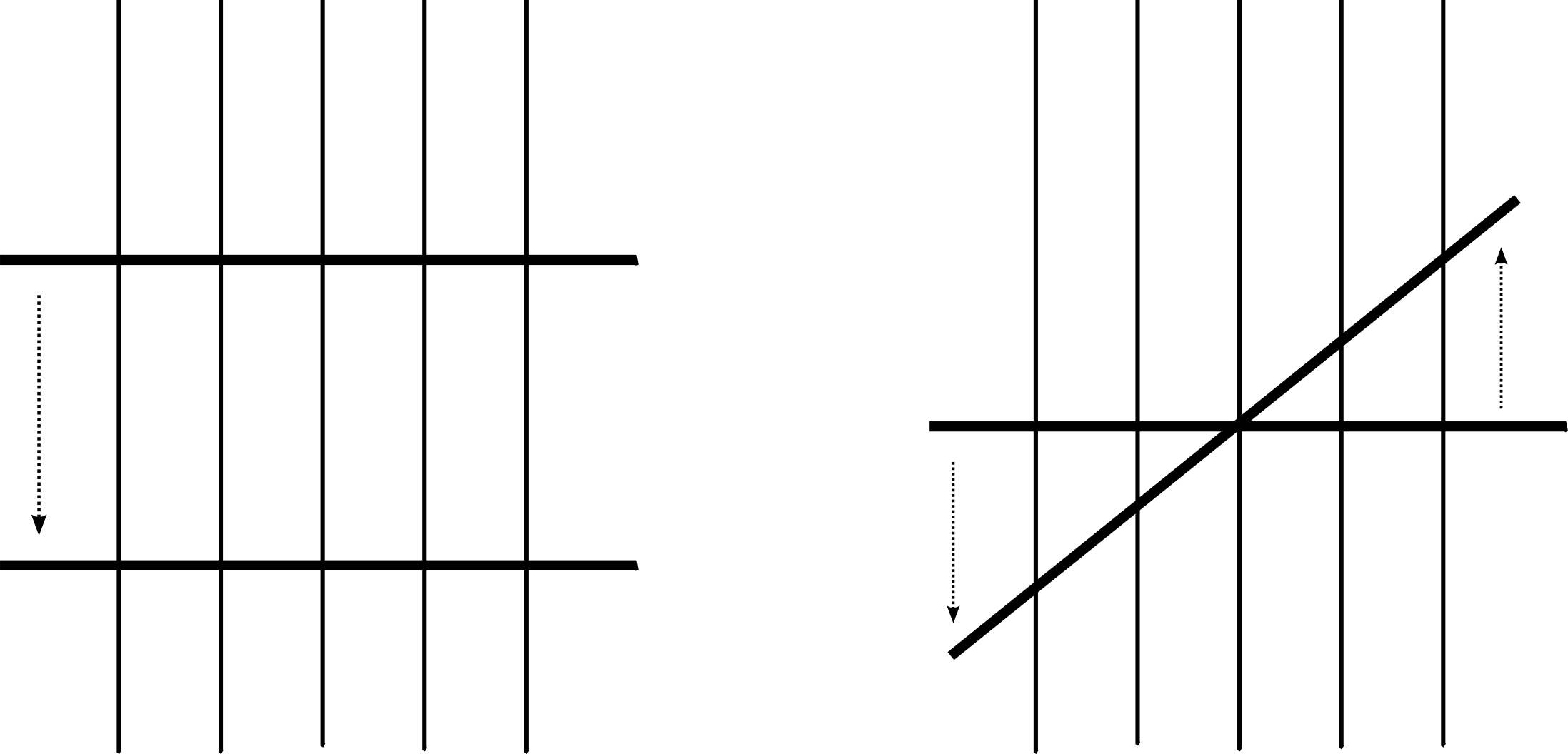}
\caption{\label{fig:isom coeucl} In this affine model of the co-Euclidean plane, the absolute $\{\ell\}$ is at infinity, in such a way that the vertical lines are parabolic. Also the quotient of $\{1,-1\}\times\R\subset \mathcal S^1\times\R$ is at infinity. On the left is the effect of the dual of a Euclidean homothety on an elliptic line of the co-Euclidean plane, and on the right is the effect of the dual of a Euclidean translation on an elliptic line of the co-Euclidean plane. }
%\end{center}
\end{SCfigure}

The isotropy group  $I(p)$ of a point $p\in {^*\E^n}$ corresponds to the isometries of $\E^n$ that fix the dual hyperplane $p^*$, i.e. all the translations by vectors in the direction of $p^*$. 
The group $I(p)$ actually pointwise fixes the parabolic line passing through $p$, because each point on the line is dual to a hyperplane parallel to $p^*$ in $\E^n$. By a similar argument, $I(p)$ preserves every other parabolic line (not through $p$) and acts on it by translation. Finally, the group $I(p)$ acts simply transitively on all the elliptic hyperplanes containing $p$, because two such hyperplanes $a,b$
are dual to two points $a^*, b^*$ contained in the hyperplane $p^*$ of $\E^n$, and there is a unique translation in the direction of $p^*$ sending $a^*$ to $b^*$. The stabilizer in $I(p)$ of any such elliptic hyperplane clearly acts on it as $\operatorname{O}(n-1)$.

As a result of this discussion, let us show that any $(0,2)$-tensor on $^*\E^n$ invariant for $\mathrm{Isom}(^*\E^n)$ must necessarily be degenerate. Let $g$ be a bilinear form on $T_p \, ^*\E^n$ invariant under the action of $I(p)$, and fix a vector $X\in T_p\, ^*\E^n$ which is not tangent to the parabolic line through $p$. 
%Let $V$ be a vector tangent to a parabolic line at $p$. Consider a tangent vector $X$ at $p$, non collinear to $V$, such that $g(V,X)=0$. As $I(p)$ acts transitively on the non-parabolic directions, and fixes $V$, then $g(V,X)=0$ for any $X$ not collinear to $V$. Hence $V$ is in the kernel of $g$, thus giving a contradiction.
By the above discussion, there is an element of $I(p)$ which maps $X$ to $X+V$, where $V$ is a vector tangent to the parabolic line. Hence if $I(p)$ preserves the bilinear form $g$, then $V$ is null for $g$. Hence $g$ cannot be a scalar product. Since $I(p)$ acts transitively on lines spanned by vectors $X$ as above, all such vectors $X$ must be of the same type, and therefore $g$ is degenerate. 

\begin{fact}\label{fact no met coeucl}
There is no pseudo-Riemannian metric on $(^*\E^n,\mathrm{Isom}(^*\E^n))$.
\end{fact}

In fact, the argument above essentially shows that the metric must be of the form described above (up to a factor), namely when lifted to  the double cover $\cE^n$, it restrict to the spherical metric on every hyperplane transverse to the parabolic line.

\subsection{Minkowski space}\label{sec mink}

\paragraph{Minkowski space as a degenerate model geometry.}\index{Minkowski space}

Like Euclidean space, the Minkowski space is a subgeometry of the affine geometry,
and isometries of Minkowski space are of the form

\begin{equation*} 
\left[
\begin{array}{ccc|c}
  
  & & & t_1 \\
   & \raisebox{-4.5pt}{{\huge\mbox{{$A$}}}}  & & \vdots  \\
  & & & t_n \\ \hline
  0 & \cdots & 0 & 1
\end{array}
\right]~,
\end{equation*}
with $A\in \operatorname{O}(n-1,1)$.
Such a transformation also acts on $\R^n\times \{0\}$, and the action reduces to the  action of $A$.
As  $A\in \operatorname{O}(n-1,1)$, passing to the projective quotient, the hyperplane at infinity of $\mathbb{A}^n$ is endowed with the  hyperbolic and the de Sitter geometries. 
 Conversely, suppose that $(\mathbb{A}^n,H)$ is a geometry,  where $H$ is a group of affine transformations, whose action on the hyperplane at infinity  is the one of $\operatorname{PO}(n-1,1)$. Then necessarily the part $A$  for the representative of an element of $H$ must belong to $\operatorname{O}(n-1,1)$.
 This can be summarized as follows.
% 
%Moreover, $\rho(M)$ also acts on $\R^n\times\{0\}$, and the action is the one of $A\in \operatorname{O}(n)$. 
%So identifying $\operatorname{Isom}(\E^n)$ and $\rho(\operatorname{Isom}(\E^n))$, the Euclidean isometries extends to $\partial_\infty \E^n$ as isometries of $\Ell^{n-1}$. Moreover, any 
%isometry of $\Ell^{n-1}$ can be written as the projective quotient 
%of an element of $\operatorname{GL}(n+1,\R)$ of the form \eqref{eq: rep eucl}.

\begin{fact}\label{fact:mink}
The Minkowski space is the projective space minus  a hyperbolic-de Sitter hyperplane.
\end{fact}

%For example in two dimension:
%
%\begin{fact}
%The Euclidean plane is the projective plane minus an elliptic line.
%\end{fact}

\paragraph{Duality of convex sets.}

 %We will restrict our attention to the following class of convex sets of the Minkowski space $\M^n$.
  An affine space-like hyperplane $P$ splits $\M^n$ into two half-spaces. The time-orientation of $\M^n$ allows to speak about the future side of $P$. A convex set is \emph{future convex} if it is the intersection of the future of space-like hyperplanes.
Note that a future convex set may have also light-like support planes (e.g. the future cone of a point),\footnote{The \emph{future cone of a point} is the  union of all the future directed time-like or light-like half lines from the point.} but no time-like support plane.

\begin{fact}\label{fact:future cone}
A future convex set contains the future cone of any of its points.
\end{fact}

Let us denote by $\mathcal{F}$ the closure of the future cone of the origin minus the origin, i.e.
\begin{equation}\label{eq F}\mathcal{F}:=\{x \in \R^{n} \,|\, b_{n-1,1}(x,x)\leq 0, x_{n}>0 \}~. \end{equation}

\begin{df}
An  \emph{admissible convex set} of $\M^n$ is 
a future convex set contained in  $\mathcal{F}$. 
\end{df}

Up to translation, any future convex set contained in the future cone of a point is an admissible convex set. But not all future convex sets are admissible, even up to translation, for example consider the future of a single space-like hyperplane.

Let $K$ be an admissible convex set in $\M^n$, and 
let us identify $\M^n$ with $\{-1\}\times\M^n$ in $\M^{n+1}$. 
 Then
   the cone $\mathcal{C}(K)$ in $\M^{n+1}$ over $K$ is
$$\mathcal{C}(K)=\left\{\left.\lambda \binom{-1}{x} \, \right| \, x \in K, \lambda \geq 0 \right\}~.$$
{Note that $C(K)$ is not closed as it contains points with zero first coordinate in its closure.} 

Let $\mathcal{C}(K)^*$ be its dual in $\M^{n+1}$ for  $b_{n,1}$:
 \begin{equation}\label{eq:CK* mink}\mathcal{C}(K)^*=\{(y_1,y) \in \R^{n+1} \,|\,
 b_{n,1}\left( (y_1,y),(x_1,x)\right) \leq 0, \forall (x_1,x)\in \mathcal{C}(K) \}~.\end{equation}

  We will  denote by $K^*$ the intersection of 
$\mathcal{C}(K)^*$ with $\{y_1=-1\}$. As we identify 
  $\{y_1=-1\}$ with $\M^n$,
 $K^*$ is a closed convex set in $\M^n$, and 
it is readily seen that 
\begin{equation}\label{eq:K* mink} K^*=\{y\in \R^n \,|\, b_{n-1,1}( x,y) \leq -1, \forall x \in K \}~. \end{equation}
This  corresponds to the affine duality with respect to the unit hyperboloid, compare Figure~\ref{fig:dual-hyp} and Figure~\ref{fig:comink}. The dual of $H_r$, the future convex side  of a branch of hyperboloid with radius $r$,  
is  $H_{1/r}$. The fact that $H_1^ *=H_1$ comes also from the following:
 the boundary of $\mathcal{C}(H_1)$ is (a part of) the isotropic cone $\mathcal{I}(b_{n,1})$, which is dual to itself for $b_{n,1}$.
%
%\begin{figure}[h]
%   \begin{minipage}[c]{.46\linewidth}
%     \includegraphics[scale=0.2]{dual-hyperbole.jpg}
%\caption{\label{fig:dual-hyp} Minkowski duality= affine duality/hyperbola}
%   \end{minipage} \hfill
%   \begin{minipage}[c]{.46\linewidth}
%   \includegraphics[scale=0.7]{cominkowski.jpg}
% \caption{Effects on the duality on point/plane.}\label{fig:comink}
%   \end{minipage}
%\end{figure}

Let us call \emph{admissible cone} the future cone of a point contained in the interior of $\mathcal{F}$, and  \emph{admissible truncation} the intersection of $\mathcal{F}$ with the future side of a space-like hyperplane $P$ such that the origin is in the past of $P$, see Figure~\ref{fig:fundpiece}.

\begin{SCfigure}
    \includegraphics[scale=0.2]{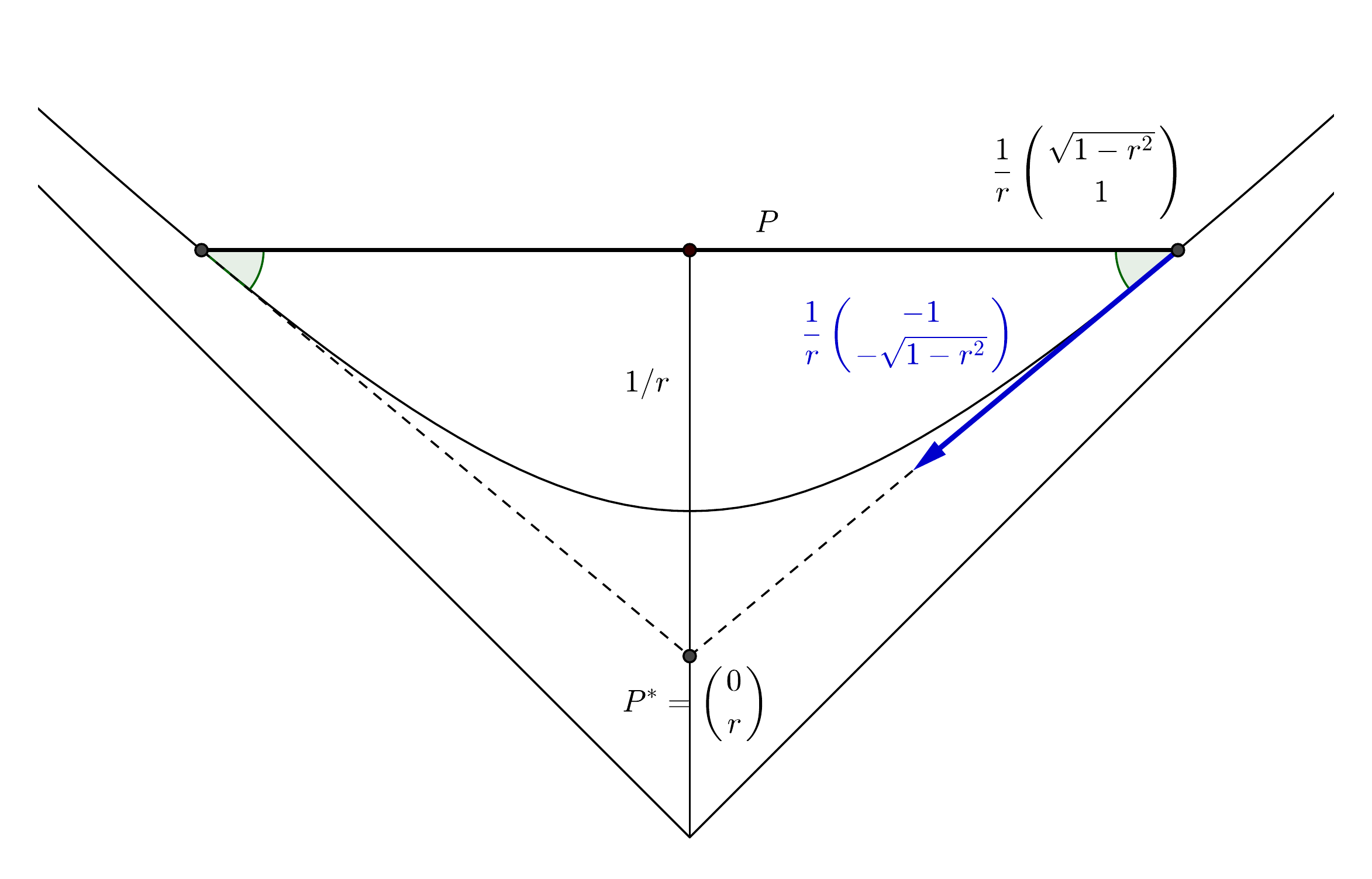}
\caption{\label{fig:dual-hyp} The duality of admissible convex sets in Minkowski space is the affine duality with respect to the upper hyperboloid.}
\end{SCfigure}

\begin{SCfigure}
    \includegraphics[scale=0.9]{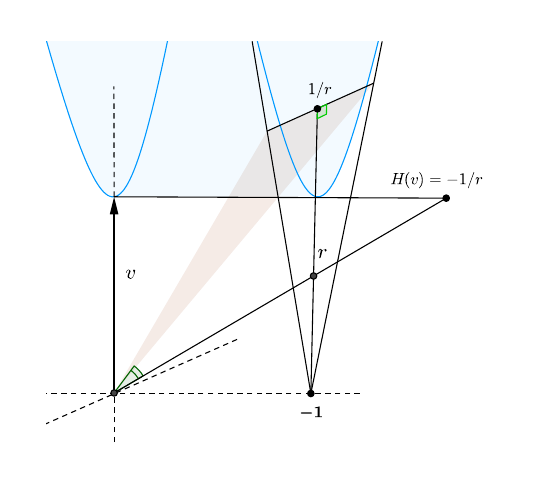}
 \caption{The duality between points and lines in the Minkowski plane expressed in term of orthogonality in a higher-dimensional Minkowski space.}\label{fig:comink}
\end{SCfigure}

\begin{SCfigure}%\begin{center}
\includegraphics[scale=1]{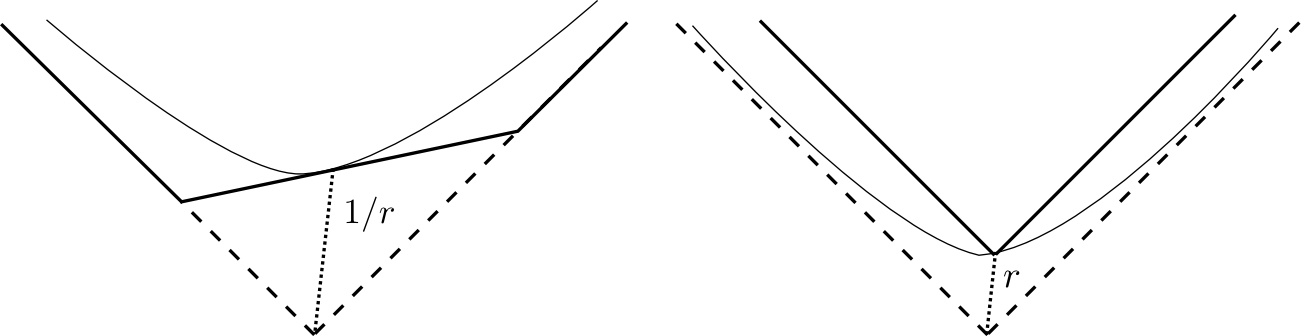}
\caption{\label{fig:fundpiece}The dual of an admissible truncation is an admissible cone.}
%\end{center}
\end{SCfigure}

\begin{lemma}\label{lem: fundamental pieces}
The dual of an admissible truncation is an admissible cone, and vice versa.
\end{lemma}
\begin{proof}
Let $K$ be an admissible truncation, i.e. if $v$ is the unit future vector orthogonal to $P$ and if $r$ is the distance from the origin to $P$, then
$$K=\{x\in \M^{n} \,|\, b_{n-1,1}(x,x)<0, b_{n-1,1}(x-rv,v)\leq 0 \}~. $$
From \eqref{eq:K* mink}, we see that $\frac{1}{r}v \in K^*$. From Fact~\ref{fact:future cone}, if $C$ is the future cone of $\frac{1}{r}v$, then $C \subset K^*$.

On the other hand, as $\frac{1}{r}v \in C$, if $y \in C^*$, then by \eqref{eq:K* mink}, $b_{n-1,1}(y,v)\leq -r$, which implies that $y\in K$, hence $C^*\subset K$. The result follows \eqref{autoduality} and {\eqref{eq:inclusion cone}}. 
\end{proof}
\begin{lemma}
The dual of an admissible  convex set in Minkowski space is an admissible  convex set in Minkowski space.
\end{lemma}
\begin{proof}
Let $K$ be an admissible convex set. By Fact~\ref{fact:future cone}, $K$ contains the future cone of a point, and by {\eqref{eq:inclusion cone}} and Lemma~\ref{lem: fundamental pieces}, $K^*$ is contained in an admissible truncation, in particular $K^*$ is in $\mathcal{F}$.

As $K$ is in $\mathcal{F}$, there exists a hyperplane separating $K$ from the origin, in particular $K$ is contained in an admissible truncation. By {\eqref{eq:inclusion cone}} and Lemma~\ref{lem: fundamental pieces}, $K^*$ contains the future cone of a point. From this is it easy to see that the closed convex set $K^*$ must be a future convex set.
\end{proof}

 \paragraph{Support function.}
\index{support function}
Let $\tilde H : \R^n \to \R \cup \{+\infty\}$ be the convex function whose graph is the cone over the boundary of $K$. In particular, $\mathcal{C}(K)$ is the intersection of the epigraph of $\tilde H$ with $\{x_1>0\}$:
\begin{equation}\label{eq:epgraphe mink}
\mathcal{C}(K)=\{(x_1,x)\in \R_{>0}\times \R^{n} \,|\, x_1\geq \tilde{H}(x) \}~.
\end{equation}
 The function $\tilde H$ is non-positive, convex and homogeneous of degree one.
Let $\operatorname{dom}\tilde H$ be the  domain of $\tilde{H}$,  the set of points where $\tilde H$ takes finite values.
By Fact~\ref{fact:future cone}, it is easy to see that any 
future time-like ray from the origin meets the boundary of $K$ exactly once, hence 
$$\operatorname{int}\mathcal{F}\subset \operatorname{dom}\tilde H \subset \mathcal{F}~,
 $$
 (here $\mathcal{F}$ is the set \eqref{eq F} 
 for the Minkowski structure induced on $\{0\}\times \R^n$ by that of the ambient $\M^{n+1}$). Let $v\in \partial \mathcal{F}$, i.e. $v$ is a future light-like vector. If there exists a $\lambda >0$ such that $\lambda v \in K$, then 
 $v \in \operatorname{dom}\tilde H$. Otherwise, 
 as from \eqref{eq:epgraphe mink},
\begin{equation}\label{eq K}K=\{x \in \mathcal{F} \,|\, \tilde H(x) \leq -1 \}~, \end{equation}
we would have $\tilde H(\lambda v)>-1$.  So, by homogeneity, $0\geq \tilde H(v) \geq -1/\lambda$
for all $\lambda >0$, hence $\tilde{H}(v)=0$, and the domain of $\tilde{H}$ is $\mathcal{F}$.

As the epigraph of $\tilde H$ is closed, $\tilde H$ is lower semi-continuous, hence it is determined by its restriction to $\operatorname{int}\mathcal{F}$. {See Section~7 in \cite{roc} for details.}
We will denote by $H$ the restriction of $\tilde H$ to  $\operatorname{int}\mathcal{F}$. The function $H$ is negative.

\begin{df}
The function $H$  is the \emph{support function} of $K^*$.
\end{df} 

The exact relation between $H$ and $K^*$ is given by the following lemma.

\begin{lemma}
Let $H$  be the support function of $K^*$.
Then 
\begin{equation}\label{eq:supp fct mink}K^*=\{y \in \R^n \,|\, b_{n-1,1}( y,x) \leq H(x),\, \forall x\in \operatorname{int}\mathcal{F}  \}~. \end{equation}
\end{lemma}
\begin{proof}
By \eqref{eq:K* mink} and \eqref{eq K}
$$K^*=\{y \in \R^n \,|\, b_{n-1,1}(x,y) \leq -1 \; \forall x\in \mathcal{F}, \tilde H(x) \leq -1 \}~, $$
so it is straightforward that the set 
$$A=\{y\in \R^n \,|\, b_{n-1,1}(x,y) \leq \tilde H(x)\; \forall x \in \mathcal{F} \} $$
is contained in $K^*$. Conversely, suppose that $y\in K^*$ and let $x\in \mathcal{F}$. If $\tilde{H}(x)=0$, then $y\in A$. Otherwise, 
by homogeneity,
$$\tilde H \left(\frac{x}{-\tilde H(x)} \right)=-1 $$
so
$$b_{n-1,1}\left(\frac{x}{-\tilde H(x)},y \right) \leq -1 $$
i.e. $b_{n-1,1}(x,y)\leq \tilde H(x)$, so $y \in A$. By lower-semi continuity, the right-hand side of \eqref{eq:supp fct mink} is equal to $A$.

\end{proof}

 The support function also has the following interpretation. 
Let $v$ be a unit  vector of $\{0\}\times \operatorname{int}\mathcal{F} $. Then $H(v)$ is the distance in $\{-1\}\times \M^n $ between the origin and the space-like support plane of $K^*$ directed by $v$ (see Figure~\ref{fig:comink}).\footnote{In other terms, the Lorentzian orthogonal projection of the origin onto the support plane is the point $-H(v)v$.} Hence
\eqref{eq:supp fct mink} expresses the fact that $K^*$ 
is  the envelope of its space-like support planes.

The following fact follows easily from the construction, 

\begin{fact}
The support function provides a  bijection between the space of admissible convex subsets  of $\M^n$  and the space of negative convex $1$-homogenous functions on the cone $\operatorname{int}\mathcal{F}$.
\end{fact}

\subsection{Co-Minkowski space}\label{cominkowski plane}

\paragraph{Definition.}
\index{co-Minkowski space}

On $\R^{n+1}$, let $b_-^*$ be the following degenerate bilinear form

$$b^*_-(x,y)=x_2y_2+\cdots+x_{n}y_{n}-x_{n+1}y_{n+1}~. $$

Let $\cM^n$ be the unit sphere for $b_-^*$: $$\cM^n=\{x\,|\,b_-^*(x,x)=-1 \}\,,$$ endowed with the restriction of $b_-^*$ on its tangent space. Topologically,  $\cM^n$ is the real line times  an open disc. Any section 
of $\cM^n$ by a hyperplane not containing a horizontal line is isometric to $\mathcal{H}^n$. 

\begin{df}
The space  $^*\M^n=\cM^n/\{\pm \mathrm{Id}\}$ is the \emph{co-Minkowski space}.
\end{df}

The isotropic cone $\mathcal{I}(b^*)$ contains the line 
$\tilde\ell=\{(\lambda,0,\ldots,0) \}$. Let us call $\ell=\mathrm{P}\tilde\ell$ the \emph{vertex} of the absolute.
Hyperplanes of co-Minkowski space which do not contain the vertex of the absolute are hyperbolic. Lines passing through the vertex of the absolute are parabolic.

\begin{fact}
A strip between two parallel lines or the interior of a cone in the plane 
with the projective distance is a model of $^*\M^2$ or
$\overline{^*\M^2}$. 
\end{fact}

The two lines in the fact above meet at the vertex of the absolute.
In an affine chart, in dimension $3$, we have the following, see Figure~\ref{fig:com2}.

\begin{fact}The convex side of a ruled  quadric in $\R^3$ (i.e.  elliptic cone or elliptic cylinder or hyperbolic cylinder) with the projective distance is a model of $^*\M^3$ or
$\overline{^*\M^3}$. .\end{fact}

\begin{SCfigure}%[h]\begin{center}
\includegraphics[scale=0.7]{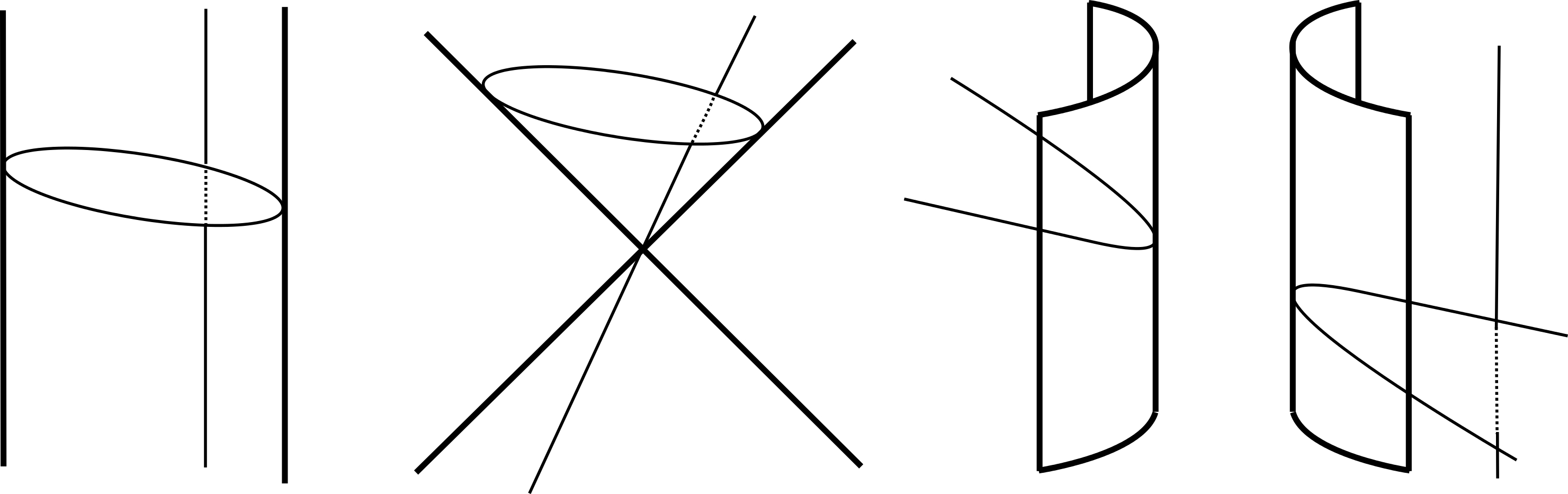}
\caption{\label{fig:com2} Projectively equivalent affine models of $3$d co-Minkowski space. In each case, a hyperbolic plane 
and a parabolic line are drawn.}
%\end{center}
\end{SCfigure}

\paragraph{Duality.}

Let $P$ be an affine space-like hyperplane in $\M^n$, and $v$  its future unit normal vector.  Let  $h(v)$ be the Lorentzian distance from $0$ to $P$, i.e. the future of $P$ has equation $b_{n-1,1}( \cdot,v)< -h(v)$. 
The vector
$\binom{h(v)}{v}\in \R^{n+1}$ is   orthogonal in $\M^{n+1}$ to the linear hyperplane  containing $P \times \{-1\}$. Its projective quotient 
defines a point $P^*$ in $^*\M^n$, see Figure~\ref{fig:comink}.
One could also consider the past unit normal vector
$-v$ of $P$. The (signed) Lorentzian distance from the origin is then 
$-h(v)$, and the point $-\binom{h(v)}{v}$ has the same projective quotient as $\binom{h(v)}{v}$.

\begin{fact}\label{fact:comink space plane}
The co-Minkowski space $^*\M^n$ is the space of (unoriented) space-like hyperplanes of $\M^n$.
\end{fact}

Conversely, a hyperbolic hyperplane of $^*\M^n$ (i.e. a hyperplane of 
$^*\M^n$ which does not contain $\ell$) is dual to a point of $\M^n$. 
A co-Minkowski hyperplane of $^*\M^n$ (i.e. a hyperplane of 
$^*\M^n$ which contains $\ell$) is dual to a point at infinity.

{
Let $K$ be an admissible convex subset of $\M^n$, and 
let $H$ be its support function.
By abuse of notation, let us also denote 
by $K^*$ the intersection of $\mathcal{C}(K)^*$ with
$\cM^n$. The set $K^*$ is the closure of the epigraph of the restriction $h$ of 
$H$ to $\cM^n\cap \{x_{1}=0\}$, which we identify with  the upper part of the hyperboloid:$$\mathcal{H}_+^{n-1}=\mathcal{H}^{n-1} \cap \{x_{n}>0 \}~.$$ Note that  by homogeneity,  $H$ is determined by its restriction $h$ to $\mathcal{H}_+^{n-1}$:
\begin{equation*}H(x)=\|x\|_-h(x/\|x\|_-)~,
\end{equation*} 
with $\|x\|_-=\sqrt{-b_{n-1,1}(x,x)}$. 
\begin{fact}
The convex set $K^*$ of $\cM^n$ is the epigraph of $h$, the restriction  to $\mathcal{H}_+^{n-1}$ of the support function of $K$.
\end{fact}
Note that the function $h:\mathcal{H}_+^{n-1}\to \R$ can be extended by symmetry to $\mathcal{H}^{n-1}$, in such a way that the} projective quotient is well defined on $\H^{n-1}$.

\paragraph{Cylinder model.}

There is another convenient way to describe the dual of admissible convex sets of Minkowski space in the co-Minkowski space, by looking at the affine chart $\{x_{n+1}=1\}$. In this model, $^*\M^n$ is an elliptic cylinder $B^{n-1}\times \R$, with $B^{n-1}$ the open unit disc.

Let $H$ be the support function of $K^*$.
The homogeneous function $H$ is also uniquely determined  by its restriction $\bar h$ to 
$B^{n-1}\times \{1\}$. The intersection of $\mathcal{C}(K)$ with 
$\{x_{n+1}=1\}$ is the graph of  $\bar h$. 
As a restriction of a convex function to a hyperplane, $\bar h$ is a convex function. Conversely, it can be easily seen that the 1-homogeneous extension of a convex function $B^{n-1}\times \{1\}$ is a convex function on the future cone of the origin. {See e.g. \cite[Lemma 2.6]{bf}.} 

\begin{fact}\label{fact:cvxe function disc}
There is a bijection between negative convex functions 
$B^{n-1} \to \R$ and admissible convex subsets of Minkowski space $\M^n$.
\end{fact}

Let us denote  by $K^\#$ the intersection of $C(K)^*$ with $\{x_{n+1}=1\}$.  By \eqref{eq:supp fct mink}, the convex function $\bar h$ determines $K^\#$:
\begin{equation*}\label{Kdiese}
K^\#=\left\{y \,|\, b_{n-1,1}\left(\binom x 1, \binom{ y}{ y_{n}}\right)\leq \bar h(x) \right\}~. \end{equation*}
The way to go from $K^*$ to $K^\#$ is by a projective transformation sending the origin to infinity, see Figure~\ref{fig:com}.

\begin{SCfigure}%[h]\begin{center}
\includegraphics[scale=1.2]{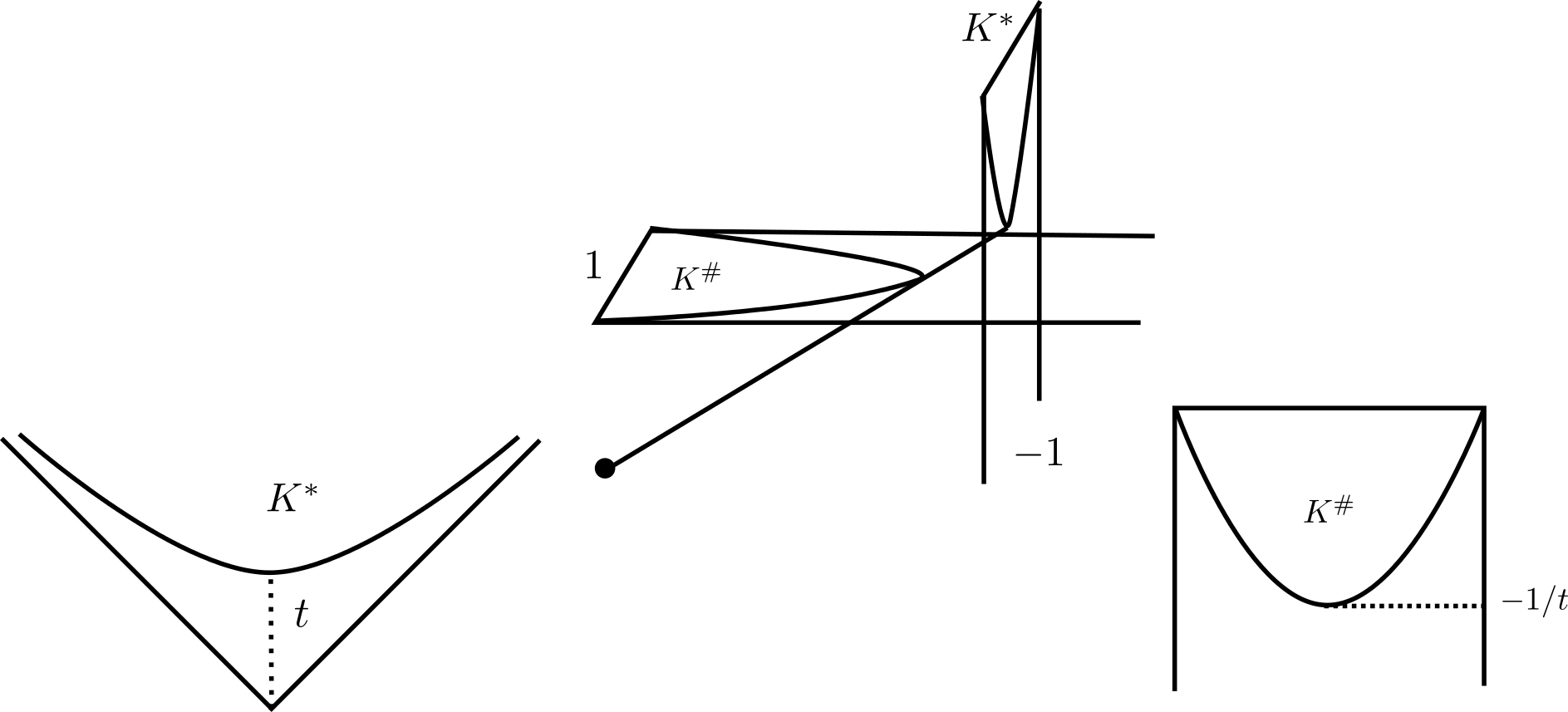}
\caption{\label{fig:com} The projective transformation $\binom x y \mapsto \binom{x/y}{-1/y}$ gives $K^\#$ from $K^*$. A hyperbola is sent to a parabola.}
%\end{center}
\end{SCfigure}

\paragraph{A model geometry.}

Let $\mathrm{Isom}(b_-^*)$ be the subgroup of projective transformations that preserves $b^*$. There is a natural injective morphism
${^*}:\mathrm{Isom}(\M^n) \to  \mathrm{Isom}(b_-^*)$ which is defined as follows.
Let $$P(v,h)=\{z\in \R^n \,|\, b_{n-1,1}(z,v)=h \}$$ be a space-like affine hyperplane of $\M^n$, $v\in \mathcal{H}_+^{n-1}, h\in \R^*$. For $A\in \operatorname{O}(n-1,1)$ and $\vec{t}\in \R^n$,
we have $$AP(v,h)+\vec{t}=P(Av,h+b_{n-1,1}(v,A^{-1}\vec{t})),$$ 
from which we define $^*I$ for $I\in \mathrm{Isom}(\M^n)$ as in \eqref{eq: etoile eucl}, with $A\in \operatorname{O}(n-1,1)$.

\begin{df} \label{defi isom comink}
The isometry group of the co-Minkowski space, $\mathrm{Isom}(^*\M^n)$, is the group of projective transformations of the form \eqref{eq:isom co}
for $A\in \operatorname{O}(n-1,1)$, $\vec{t}\in \R^n$.
\end{df}

Note that $\mathrm{Isom}(^*\M^n)$ is a proper subgroup of the group of isometries of $b^*$. For instance the latter also contains \emph{homotheties} of the form
$[\operatorname{diag}(1,\ldots,1,\lambda)]$. As in the co-Euclidean case, these homotheties correspond to displacement along 
parabolic lines, see Figure~\ref{fig:isom comink}.

The  co-Minkowski space is naturally endowed with a degenerate metric $g^*$, which is the hyperbolic metric on the hyperbolic hyperplanes, and zero on the parabolic lines. 

\begin{SCfigure}
%[h]
%\begin{center}
     \includegraphics[scale=0.4]{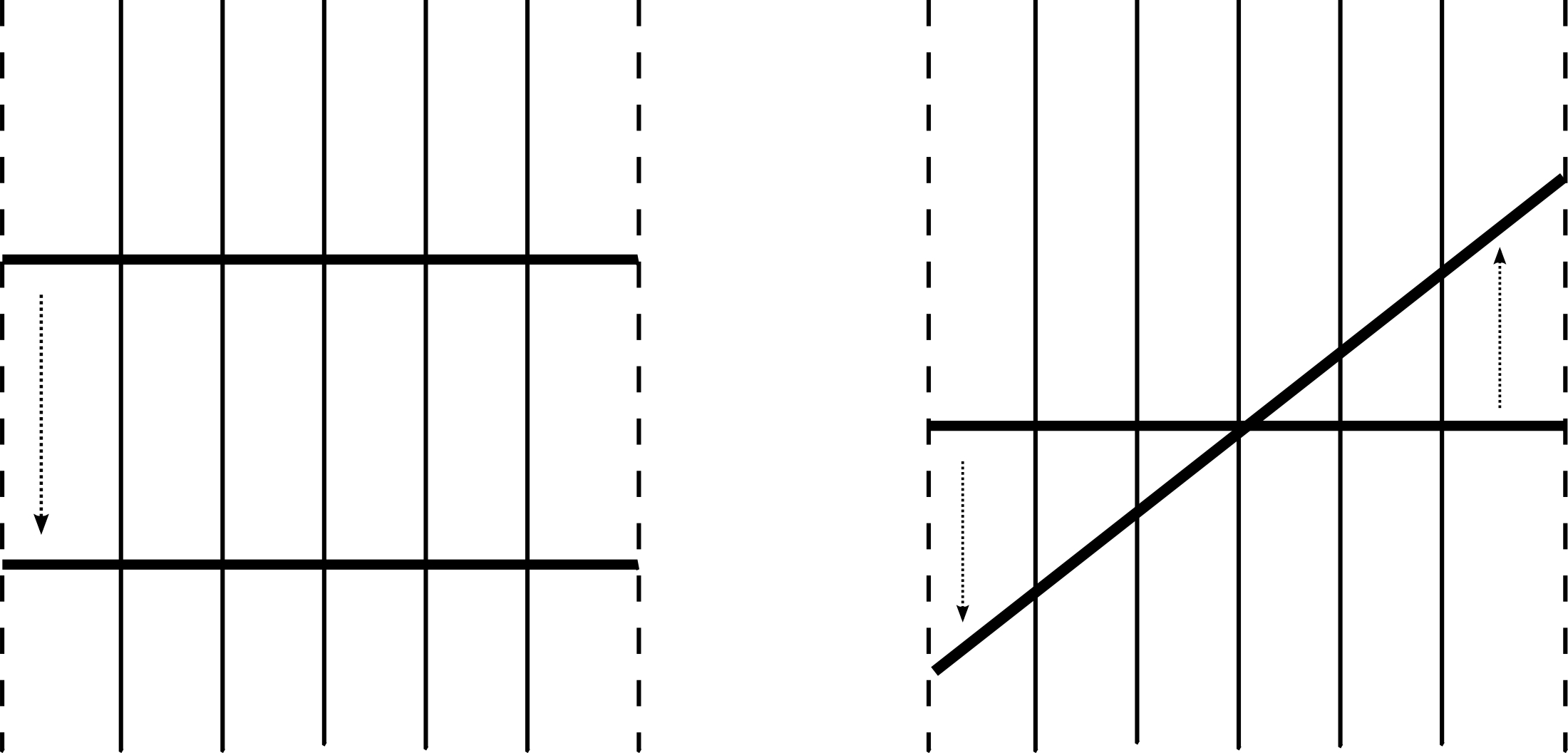}
\caption{\label{fig:isom comink} In this affine model of the co-Minkowski plane, the absolute is composed of two vertical lines. The other vertical lines are parabolic. On the left is the effect of a Minkowski homothety on a hyperbolic line of the co-Minkowski plane, and on the right is the effect of a Minkowski translation on a hyperbolic line of the co-Minkowski plane.}
%\end{center}
\end{SCfigure}

Following the same reasoning as in the co-Euclidean case, we see that 
any $(0,2)$-tensor field on the co-Minkowski space which is invariant under the action of $\mathrm{Isom}(^*\M^n)$  must have parabolic directions in its kernel.

\begin{fact}
There is no pseudo-Riemannian metric on $(^*\M^n,\mathrm{Isom}(^*\M^n))$.
\end{fact}

\paragraph{Comments and references} \small
\begin{itemize}

\item One fundamental property of the support functions is that they behave well under addition. More precisely, the Minkowski sum $A+B=\{a+b \,|\,A\in A,b\in B\}$ of two admissible convex bodies is an admissible convex body. If $H_A$ and $H_B$ are the corresponding support functions, then $H_{A+B}=H_A+H_B$. Classical references for convex bodies are for example \cite{bf} and the up-to-date 
 \cite{schneider}.
\item  From Remark~\ref{remark:angles coeucl} and \eqref{eq:d hilb}, the angle between two lines in the Euclidean plane can be written as the logarithm of a cross-ratio. This is the \emph{Laguerre formula}\index{Laguerre formula}, see e.g. \cite{richter}.
\item Similarly to Fact \ref{fact:euclidean} and \ref{fact:mink}, on has that $(\R^{n},b_{n-2,2})$ can be identified with the projective space of dimension $n$ minus a hyperplane endowed with the Anti-de Sitter geometry.
\item The duality between Euclidean (resp. Minkowski) and co-Euclidean (resp. co-Minkowski) spaces is expressed using a different formalism in \cite{CDW,struve}. 
\item Similarly to Fact~\ref{fact:comink space plane}, it is easy to see that  the outside of the co-Minkowski space (i.e. the projective quotient of $(b^ *_-)^ {-1}(1)$) is the space of time-like hyperplanes of Minkowski space. 
\item The boundary of an admissible convex set $K$ in Minkowski space is the graph of a $1$-Lipschitz convex function $u:\R^{n-1}\to \R$. 
From \eqref{Kdiese}, $\bar h=\sup_{(p,p_n)\in K}b_{n-1,1}\left( \binom x 1,\binom{p}{p_n} \right) $ i.e.
$\bar h=\sup_{(p,p_{n})\in K}\left\{ b_{n-1,0}(x,p) - p_{n}\right\}$ i.e. $\bar h =\sup_{p\in \R^d}\left\{ b_{n-1,0}( x,p) - u(p)\right\}$, i.e. $u$ is  the dual of $\bar h$  for the Legendre--Fenchel duality, see e.g. \cite[2.3]{bf} and references therein for more details.
\item By Fact~\ref{fact:cvxe function disc}, a negative convex function on the open unit ball $B$ is the support function $\bar h$ of an admissible convex set $K$ in Minkowski space. 
Let us suppose that the lower semi-continuous extension $g$ of $\bar h$ to the boundary of $B$ is continuous (this is not always the case, for example $g$ could attain no maximum, see p.~870 of \cite{GKR}). Let $h_0$ be the convex function on $B$ such that the lower boundary of the convex hull of the graph of $g$ is the graph of $h_0$. Then $h_0\geq h$.
The function $h_0$ is the support function of a convex subset of Minkowski space which is the \emph{Cauchy development}\index{Cauchy development} of the convex  set whose support function is $h$.  The function $g$ encodes the light-like support planes of $K$
(and $\partial B \times \R$, the absolute of the co-Minkowski space in a suitable affine model, is sometimes called the \emph{Penrose boundary}\index{Penrose boundary}).
Note that a light-like support plane may not touch the boundary of $K$, as in the case of the upper branch of the hyperboloid.
%see e.g. \cite[Proposition 4.15]{BMS13}.

Let $h_0^+$ be the concave function on $B$ such that the upper boundary of the convex hull of the graph of $g$ is the graph of $h_0^+$. It is possible to 
consider $h_0^+$ as the support function of a \emph{past} convex set in Minkowski space. At the end of the day, at a projective level, it would be better to consider as ``convex sets'' of Minkowski space the data of a future convex set and a past convex set, with disjoint interior and same light-like support planes.
\item The $3$ dimensional spaces $\H^3$, $\AdS^3$ and $^*\M^3$ can be defined in a unified way as spaces of matrices with coefficient in $\R + \kappa \R$, with
$\kappa \notin \R$, $\kappa^2 \in \R$, see \cite{DMS}.
\item The reference \cite{DMS} also contains a Gauss--Bonnet formula for $^*\M^2$ geometry.
\end{itemize}
\normalsize

%\begin{figure}\begin{center}
%\includegraphics[scale=1]{exampledual.png}\caption{The effect of the maps $C,D,P$ (see Figure~\ref{fig:composition}) on an admissible convex set bounded by a hyperboloid, and on the future cone of a point.}
%\end{center}\end{figure}

%
%\newpage
%
%\begin{figure}\begin{center}
%\includegraphics[scale=1]{composition.png}\caption{The map $D$ sends an admissible convex set of Minkowski plane to its dual in Minkowski plane. The map $P$ is the  mapping $(x,y)\mapsto (x/y,-1/y)$ which sends the line a infinity to  vertical line. The map $C$ sends an admissible convex set of the Minkowski plane to the half-pipe model of the co-Minkowski plane, and $P\circ D=C$.}\label{fig:composition}
%\end{center}\end{figure}

\section{Geometric transition} \label{sec geometric transition}
\index{geometric transition}
In this section, we will study the so-called \emph{geometric transition} of model spaces  as subsets of projective space. Recall that a model space is an open subset $\mathbb{M}$ of $\R \P^n$, endowed with a closed subgroup $\mathrm{Isom}(\mathbb{M})$ of $\mathrm{PGL}(n+1,\R)$ which preserves the geometric structure of $\mathbb{M}$. 

Moreover, recall that given a model space $\mathbb{M}$, by applying a projective transformation $g\in \mathrm{PGL}(n+1,\R)$ one obtains another model space which is equivalent to $\mathbb{M}$. Indeed, the group of isometries of $g\mathbb{M}$ is precisely $g \mathrm{Isom}(\mathbb{M}) g^{-1}$.

We say that a model space $(\mathbb{N},\mathrm{Isom}(\mathbb{N}))$ is a \emph{conjugacy limit} or \emph{rescaled limit} of another model space $(\mathbb{M},\mathrm{Isom}(\mathbb{M}))$ if there exists a sequence of projective transformations $g_n\in \mathrm{PGL}(n+1,\R)$ such that:
\begin{itemize}
 \item[i)]\label{i} The sequence $g_k \mathbb M$ converges to $\mathbb N$  as $k\to\infty$;
 \item[ii)]\label{ii} The sequence of closed subgroups $g_k \mathrm{Isom}(\mathbb{M}) g_k^{-1}$ converges to $ \mathrm{Isom}(\mathbb{N})$.
\end{itemize}

The convergence here should be meant as the Hausdorff convergence, for instance. Thus the conditions i) and ii) essentially mean that:
\begin{itemize}
 \item[i)] Every $x\in \mathbb{N}$ is the limit of a sequence $\{g_k x_k\}$, for some $x_k\in\mathbb{M}$, as $k\to\infty$;
 \item[ii)] Every $h\in \mathrm{Isom}(\mathbb{N})$ is the limit of a sequence $\{g_k h_k g_k^{-1}\}$, for some $h_k\in\mathrm{Isom}(\mathbb{M})$, as $k\to\infty$.
\end{itemize}
Of course, in general, the sequences $g_k$ and $h_k$ are not  compact sequences in $\mathrm{PGL}(n+1,\R)$. 

A toy model of geometric transition is the 1-dimensional case, namely, the degeneration of projective lines. We already know that a line $c$ in a model space is a subset of a copy of $\R\mathrm{P}^1$, with $0$, $1$ or $2$ points in the absolute if $c$ is elliptic, parabolic or hyperbolic respectively. More precisely, composing with a projective transformation, we can assume that $c$ is $\R\mathrm{P}^1$ (if elliptic), $\R\mathrm{P}^1\setminus\{[1:0]\}$ (parabolic) or $\R\mathrm{P}^1\setminus\{[1:1],[1:-1]\}$ (hyperbolic).  The stabilizer of $c$ in the isometry group of the model space is identified to a subgroup of $\mathrm{PGL}(2,\R)$ of the form:
$$\left\{R_\theta=\begin{bmatrix} \cos\theta & -\sin\theta \\ \sin\theta & \cos\theta  \end{bmatrix}\right\}\,,\qquad\left\{T_a=\begin{bmatrix} 1 & a \\ 0 & 1  \end{bmatrix}\right\}\,,\qquad\left\{S_\varphi=\begin{bmatrix} \cosh\varphi & -\sinh\varphi \\ \sinh\varphi & \cosh\varphi  \end{bmatrix}\right\}\,,$$
for $c$ elliptic, parabolic of hyperbolic respectively. By applying the projective transformations 
$$g_k=\begin{bmatrix} k & 0 \\ 0 & 1  \end{bmatrix}\,,$$
one sees that given a sequence $\theta_k$, the sequence of conjugates 
$$g_k R_{\theta_k}g_k^{-1}=\begin{bmatrix} \cos\theta_k & -k\sin\theta_k \\ (1/k)\sin\theta_k & \cos\theta_k  \end{bmatrix}$$
converges to a projective transformation of the form
$T_a$ provided $\theta_k\sim a/k$. In an analogous way, $g_k S_{\varphi_k}g_k^{-1}$ converges to $T_a$ as $k\to\infty$ if $\varphi_k\sim a/k$, and $g_t$ maps the two points at infinity $[1:1]$ and $[1:-1]$ of the hyperbolic line to $[1:1/k]$ and $[1:-1/k]$, which converge to $[1:0]$ as $k\to\infty$. Hence: 
\begin{fact} Elliptic and hyperbolic lines converge --- in the sense of geometric transition --- to a parabolic line.
\end{fact}
See also Figure \ref{fig:deg1d}.
We will meet this phenomenon again in higher dimensions. We start by discussing some examples of geometric transition of 2-dimensional model spaces in $\R\mathrm{P}^3$.

\begin{SCfigure}
     \includegraphics[scale=0.08]{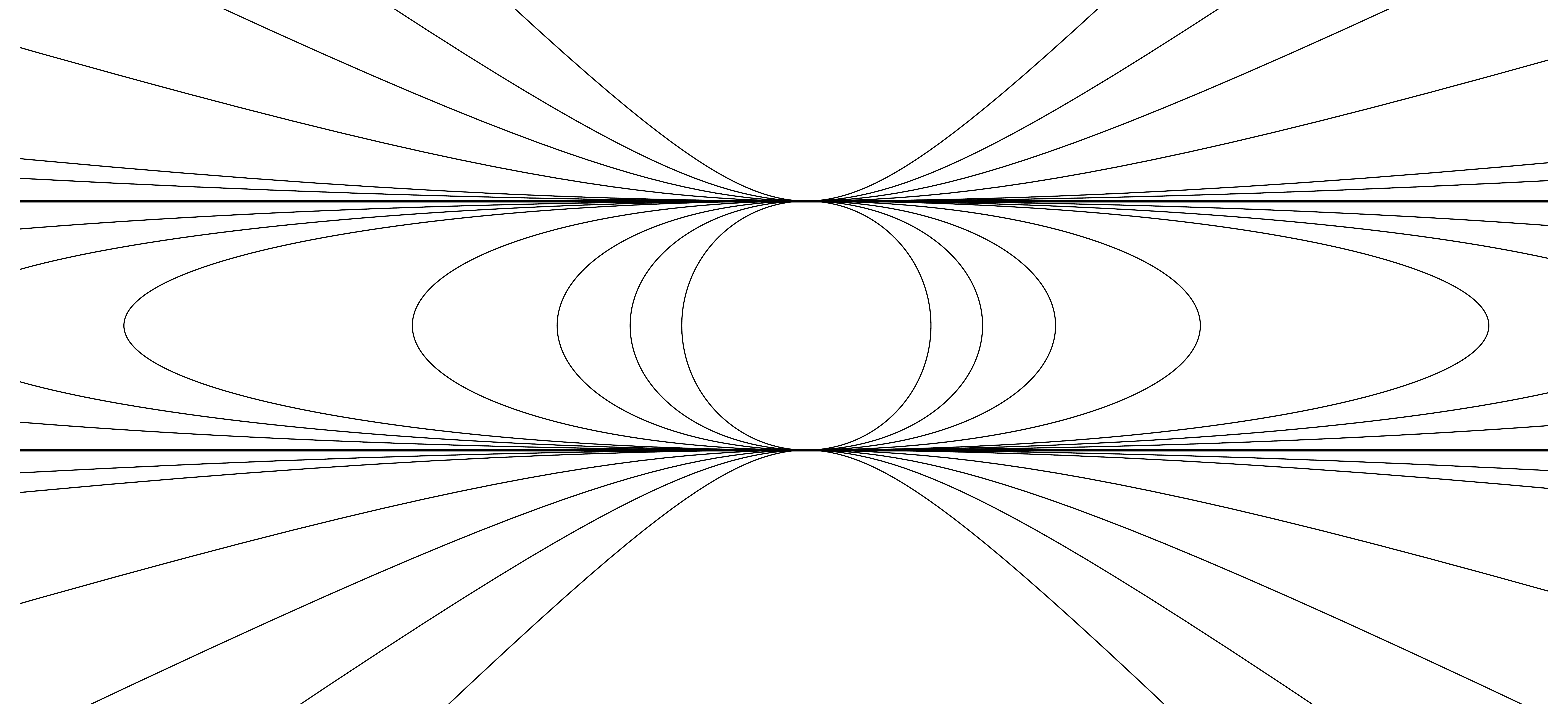}
\caption{\label{fig:deg1d} Geometric transition of elliptic line and hyperbolic line to Euclidean line, represented in the double cover.}
\end{SCfigure}

\subsection{Limits of 2-dimensional model spaces} \label{subsec Limits of 2-dimensional model spaces}

Recall that, in dimension 2, we have already introduced the 
 model spaces corresponding to the elliptic plane $\Ell^2$, the hyperbolic plane $\H^2$, the de Sitter plane $\dS^2$ (which is anti-isometric to $\AdS^2$) and the degenerate model spaces {of} Euclidean and Minkowski plane $\E^2$ and $\M^2$, co-Euclidean and co-Minkowski plane $^*\E^2$ and $^*\M^2$.

\paragraph{Elliptic and hyperbolic planes limit to Euclidean plane.}
Let us consider the unit sphere $\mathcal S^2$ in the Euclidean space $\E^3$.
We shall construct a \emph{geometric transition} which has the Euclidean plane as a conjugacy limit. For this purpose, consider the following projective transformations $g_t\in \operatorname{GL}(3,\R)$, for $t\in(0,1]$:
$$g_t=\begin{pmatrix} 1/t & 0 & 0 \\ 0 & 1/t & 0 \\ 0 & 0 & 1  \end{pmatrix}\,.$$
The transformations $g_t$ map the unit sphere $\mathcal S^2$ to an ellipsoid. 

Observe that  $g_t$ fixes the point $(0,0,1)\in\E^3$ and ``stretches'' the directions $x_1,x_2$. More precisely, given a differentiable path $x(t)\in \mathcal S^2$, for $t\in[0,1]$, such that $x(0)=(0,0,1)$, one has 
$$\lim_{t\to 0}g_t x_t=\lim_{t\to 0}\begin{pmatrix} x_1(t)/t \\ x_2(t)/t \\ x_3(t)  \end{pmatrix}=\begin{pmatrix} \dot x_1(0) \\ \dot x_2(0) \\ 1  \end{pmatrix}\,.$$
Roughly speaking, the rescaled limit is a point of the affine chart $\{x_3=1\}$, which encodes the first order derivative of the  $x_1,x_2$ coordinates. Heuristically, this rescaling procedure is a \emph{blow-up of a point}\index{blow-up! of a point}, in the sense that the point $(0,0,1)$ is preserved, and the transverse directions are ``blown-up''.

This procedure is actually well-defined in projective space. Indeed the transformations $g_t$ descend to projective transformations of $\R\mathrm{P}^2$, which we will still denote by $g_t$.
The points of the Euclidean plane $\E^2$, which is defined as
$$\E^2=\{x\in \E^{3} \,|\, x_3^2=1\}/\{\pm \mathrm{Id}\}$$
are conjugacy limits as $t\to 0$ of sequences of points of the elliptic plane
$$\Ell^2=\{x\in \E^{3} \,|\, b_{3,0}(x,x)=1\}/\{\pm \mathrm{Id}\}\,,$$
thus satisfying condition $i)$ in the definition of conjugacy limit, for instance with $t=1/k$,   {{see Figure~\ref{fig:blow2db} }}. To check that the Euclidean plane is a conjugacy limit of the elliptic space, one has to check also the condition $ii)$ on the isometry groups. That is, given an isometry of $\Ell^2$, namely an element in
$\mathrm{PO}(3)$, let us choose a representative with determinant $\pm 1$ of the form
$$h=\begin{bmatrix} a_{11} & a_{12} & a_{13} \\ a_{12} & a_{22} & a_{23} \\ a_{13} & a_{23} & a_{33}  \end{bmatrix}\,.$$
By a direct computation, one checks that
$$g_th g_t^{-1}=\begin{bmatrix} a_{11} & a_{12} & a_{13}/t \\ a_{12} & a_{22} & a_{23}/t \\ ta_{13} & ta_{23} & a_{33} \end{bmatrix}\,.$$
In fact, $g_th g_t^{-1}$ is a transformation which preserves the quadratic form $(1/t)^2  x_1^2+(1/t)^2  x_2^2+ x_3^2$. Thus every accumulation point of sequences of the form $g_t h(t) g_t^{-1}$ is necessarily of the form
$$h_\infty=\begin{bmatrix} a_{11}(0) & a_{12}(0) & \dot a_{13}(0) \\ a_{12}(0) & a_{22}(0) & \dot a_{23}(0) \\ 0 & 0 & \pm 1 \end{bmatrix}\in\mathrm{Isom}(\E^2)\,,$$
where
$$\begin{pmatrix} a_{11}(0) & a_{12}(0) \\ a_{12}(0) & a_{22}(0) \end{pmatrix}\in \mathrm{O}(2)\,.$$
and
$$h(0)=\begin{bmatrix} a_{11}(0) & a_{12}(0) & 0 \\ a_{12}(0) & a_{22}(0) & 0 \\ 0 & 0 & \pm 1 \end{bmatrix}\,.$$
This shows that the conjugacy limit of the group of isometries of $\Ell^2$ is precisely the group of isometries of $\E^2$, embedded in $\mathrm{PGL}(3,\R)$ in the usual way (see  {\eqref{eq: rep eucl}} and {\eqref{eq: rep eucl2}}). By a completely analogous proof, using the same transformations $g_t$, one sees that $\E^2$ is a conjugacy limit of the hyperbolic plane $\H^2$,  {see Figure~\ref{fig:blow2db}}. Thus one can imagine that the Euclidean plane (at $t=0$) is an interpolation of the elliptic plane (for $t>0$) and the hyperbolic plane (for $t<0$). 

\paragraph{Elliptic and de Sitter plane limit to co-Euclidean plane.}

We now describe a different procedure which permits to obtain a different limit from the elliptic plane, namely, we will obtain the co-Euclidean space as a conjugacy limit of the elliptic plane. Thus, consider
$$g_t^*=\begin{bmatrix} 1 & 0 & 0 \\ 0 & 1 & 0 \\ 0 & 0 & 1/t  \end{bmatrix}\in \mathrm{PGL}(3,\R)$$
for $t\in (0,1]$. As a remarkable difference with the previous case, the projective transformation $g_t$ leaves invariant a geodesic line of $\Ell^2$, namely the line which is defined by the plane $\{x_3=0\}$ of $\E^3$. So the ``stretching'' occurs only in the transverse directions  to $\{x_3=0\}$, and indeed the rescaled limit of a differentiable path of points of the form $x(t)=[x_1(t):x_2(t):x_3(t)]$ such that $x_3(0)=0$ is:
$$\lim_{t\to 0}g_t^* x_t=\lim_{t\to 0}\begin{bmatrix} x_1(t) \\ x_2(t) \\ x_3(t)/t  \end{bmatrix}=\begin{bmatrix} x_1(0) \\  x_2(0) \\ \dot x_3(0)  \end{bmatrix}\,.$$
We will indeed call this transition the \emph{blow-up of a line}\index{blow-up! of a line}. This shows that points of the co-Euclidean plane, defined by 
$$^*\E^2=\{x\in \E^{3} \,|\, b^*(x,x)=1\}/\{\pm \mathrm{Id}\}$$
(where $b^*(x,x)=x_1^2+x_2^2$) are rescaled limits of sequences in 
$$\Ell^2=\{x\in \E^{3} \,|\, b_{3,0}(x,x)=1\}/\{\pm \mathrm{Id}\}\,.$$
For what concerns the isometry groups, we will give again a computation which shows that $\mathrm{Isom}(^*\E^2)$ is the limit of $g_t \mathrm{Isom}(\Ell^2)(g_t^*)^{-1}$. As before, choose a representative of determinant $\pm 1$, say
$$h=\begin{bmatrix} a_{11} & a_{12} & a_{13} \\ a_{12} & a_{22} & a_{23} \\ a_{13} & a_{23} & a_{33}  \end{bmatrix}\in\mathrm{PGL}(3,\R)\,.$$
By a direct computation,
$$g_t^*h (g_t^*)^{-1}=\begin{bmatrix} a_{11} & a_{12} & ta_{13} \\ a_{12} & a_{22} & ta_{23} \\ a_{13}/t & a_{23}/t & a_{33} \end{bmatrix}\longrightarrow \begin{bmatrix} a_{11}(0) & a_{12}(0) & 0 \\ a_{12}(0) & a_{22}(0) & 0 \\ \dot a_{13}(0) & \dot a_{23}(0) & \pm 1 \end{bmatrix}\,,$$
provided $a_{13}(0)=a_{23}(0)=0$ and $$\begin{pmatrix} a_{11}(0) & a_{12}(0) \\ a_{12}(0) & a_{22}(0) \end{pmatrix}\in \mathrm{O}(2)\,.$$
In conclusion the limit of $g_t^*h (g_t^*)^{-1}$ is an element of the group $\mathrm{Isom}(^*\E^2)$ in Definition \ref{defi isom coeucl}.

Recall that we have introduced a duality of $\Ell^2$ to itself, which maps lines of $\Ell^2$ to points of $\Ell^2$ and vice versa. In principle, the dual space is a model space in the dual projective space. The dual space of $\Ell^2$ is thus identified to $\Ell^2$ itself if one chooses the scalar product of $\E^3$ to identify $\R\mathrm{P}^2$ to its dual projective space. Now, observe that the transformations $g_t$ and $g_t^*$, which have been used to rescale $\Ell^2$ to get $\E^2$ and $^*\E^2$ respectively, are well-behaved with respect to this duality. Namely, $b_{3,0}(x,y)=0$ if and only if $b_{3,0}(g_t x,g_t^* y)=0$. In other words, if $*$ denotes the duality point-line induced by the ambient scalar product, the following diagram is commutative:

\[
\xymatrix{
(\Ell^2,\mathrm{Isom}(\Ell^2)) \ar[r]^-{g_t} \ar@{<->}[d]^-{*}  & (g_t\Ell^2,g_t\mathrm{Isom}(\Ell^2)g_t^{-1}) \ar@{<->}[d]^-{*} \\
(\Ell^2,\mathrm{Isom}(\Ell^2)) \ar[r]^-{g_t^*}  & (g_t^*\Ell^2,g_t^*\mathrm{Isom}(\Ell^2)(g_t^*)^{-1}) 
}
\]

When $t\to 0$, we have already observed that $(g_t\Ell^2,g_t\mathrm{Isom}(\Ell^2)g_t^{-1})$ converges to  $(\E^2,\mathrm{Isom}(\E^2))$, while on the other hand $(g_t^*\Ell^2,g_t^*\mathrm{Isom}(\Ell^2)(g_t^*)^{-1}$ converges to 
$(^*\E^2,^*\mathrm{Isom}(^*\E^2))$. The natural duality of $\E^2$ and $^*\E^2$ is again induced by the scalar product of $\E^3$ and the commutativity of the diagram passes to the limit. Thus, we have shown:

\begin{fact} \label{rescaleEllE2}
The dual in $^*\E^2$ of a rescaled limit $x_\infty\in\E^2$ of points $x(t)\in\Ell^2$ is the rescaled limit of the dual lines $x(t)^*$ in $\Ell^2$ and vice versa (by exchanging the roles of points and lines).
\end{fact}

In a completely analogous way, one can define a geometric transition which permits to obtain the co-Euclidean plane $^*\E^2$ as a conjugacy limit of the de Sitter plane $\dS^2$. Indeed, observe that $\Ell^2$ and $\dS^2$ are the constant curvature non-degenerate model spaces in dimension $2$ which contain an elliptic line $\Ell^1$, whereas $^*\E^2$ is the degenerate space having an embedded copy of $(\Ell^1,\mathrm{Isom}(\Ell^1))$. Therefore the following fact holds:

\begin{fact} \label{rescaleHE2}
The dual in $^*\E^2$ of a rescaled limit $x_\infty\in\E^2$ of points $x(t)\in\H^2$ is the rescaled limit of the dual lines $x(t)^*$ in $\dS^2$. The dual in $\E^2$ of a rescaled limit $x_\infty\in\,^*\E^2$ of points $x(t)\in\dS^2$ is the rescaled limit of the dual lines $x(t)^*$ in $\H^2$. The converse is also true, by exchanging the role of points and lines.
\end{fact}

We thus have the following diagram, which encodes the possible transitions and dualities involving $\E^2$ and $^*\E^2$, {see Figures \ref{fig:blow2db} and  \ref{fig:blow2d}}:

\begin{equation}\label{eq:deg2d}
\xymatrix{
(\Ell^2,\mathrm{Isom}(\Ell^2)) \ar[rr]^-{\text{blow-up point}} \ar@{<->}[d]^-{*} & & (\E^2,\mathrm{Isom}(\E^2)) \ar@{<->}[d]^-{*} & & (\H^2,\mathrm{Isom}(\H^2)) \ar[ll]_-{\text{blow-up point}} \ar@{<->}[d]^-{*} \\
(\Ell^2,\mathrm{Isom}(\Ell^2)) \ar[rr]^-{\text{blow-up line}} & & (^*\E^2,\mathrm{Isom}(^*\E^2)) & & (\dS^2,\mathrm{Isom}(\dS^2)) \ar[ll]_-{\text{blow-up line}}
}
\end{equation}

\begin{SCfigure}
     \includegraphics[scale=0.08]{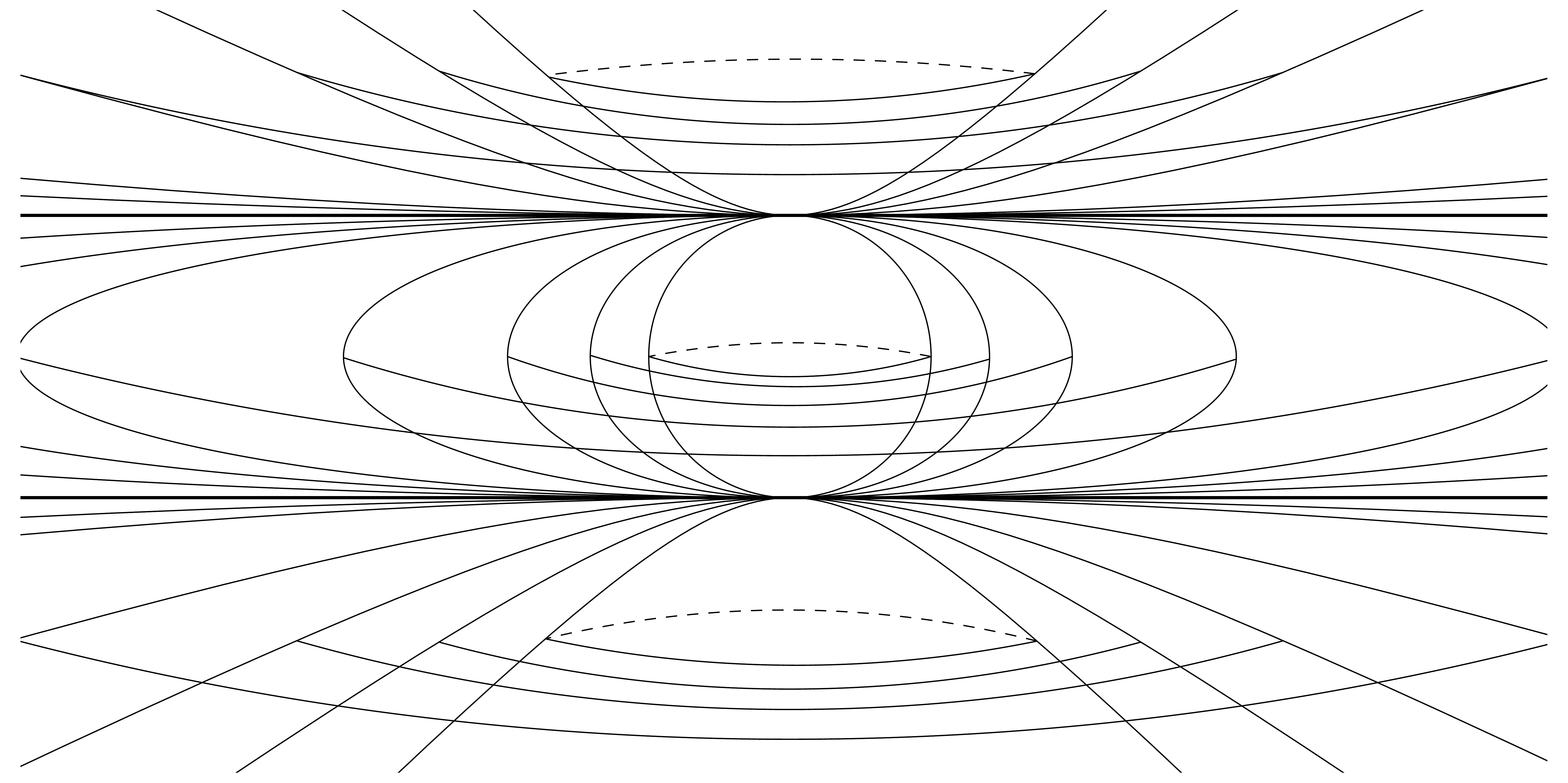}
\caption{\label{fig:blow2db} Degenerations of the top line in Scheme~\eqref{eq:deg2d}, inside the ambient space, in the double cover.}
\end{SCfigure}

\begin{SCfigure}
     \includegraphics[angle=270,origin=c,scale=0.07]{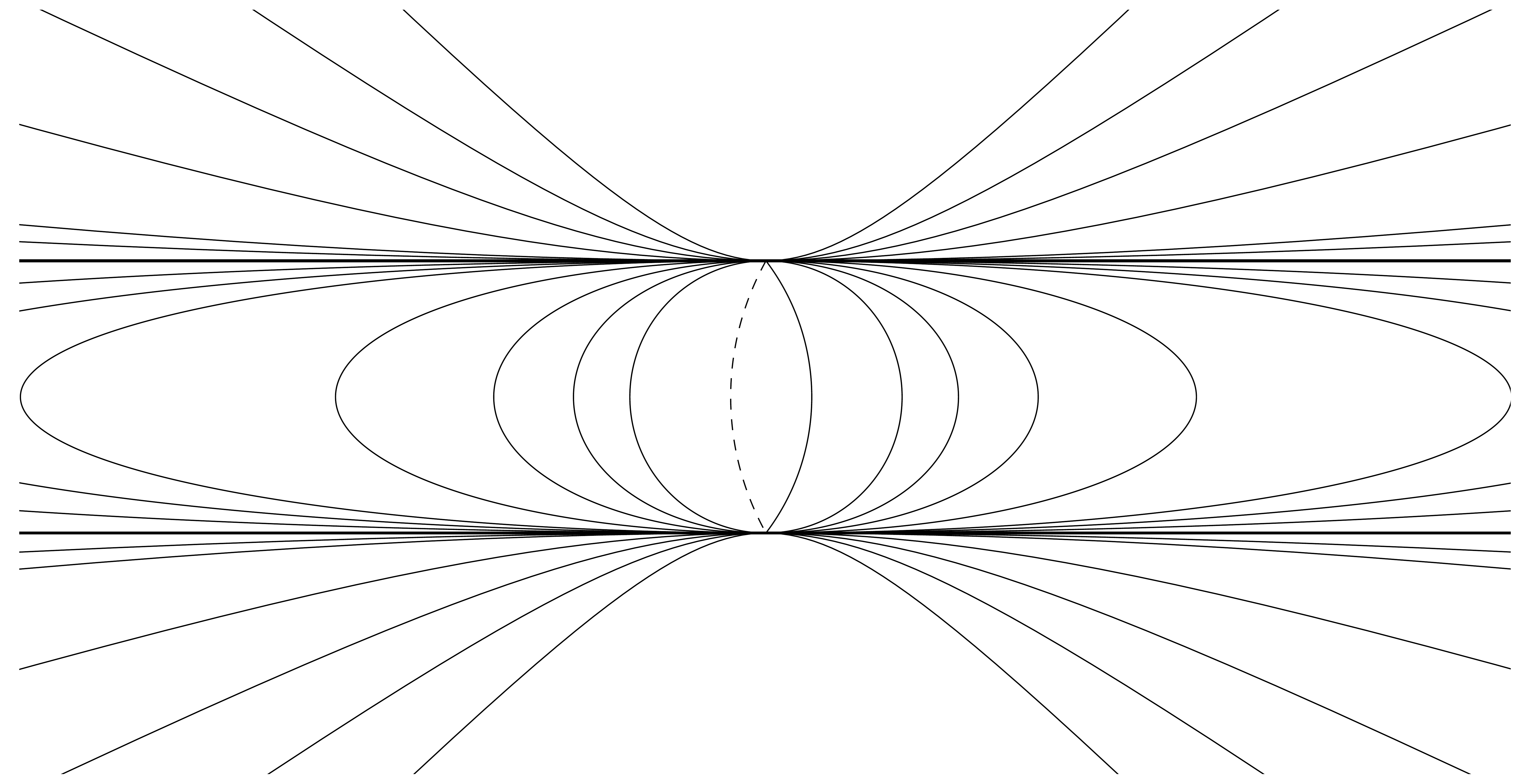}
\caption{\label{fig:blow2d} Degenerations of the bottom line in Scheme \eqref{eq:deg2d}, inside the ambient space, in the double cover.}
\end{SCfigure}

\paragraph{Geometric transitions with limits Minkowski and co-Minkowski plane.}

Given the above constructions, it is immediate to see that one can mimic the blow-up of a point for the de Sitter plane
$$\dS^2=\{x\in \M^{3} \,|\, b_{2,1}(x,x)=1\}/\{\pm \mathrm{Id}\}\,.$$
Using {again} the transformations $g_t$, the limit will be {again} represented by an affine chart defined by $\{x_3^2=1\}/\{\pm\mathrm{Id}\}$, and the conjugacy limit of $\mathrm{Isom}(\dS^2)\cong\mathrm{PO}(2,1)$ will be
$$\mathrm{Isom}(\M^2)=\left\{\begin{bmatrix} a_{11} & a_{12} & b_1 \\ a_{21} & a_{22} & b_2 \\ 0 & 0 & \pm 1 \end{bmatrix}:\begin{pmatrix} a_{11} & a_{12} \\ a_{12} & a_{22} \end{pmatrix}\in\mathrm{O}(1,1)\right\}\,.$$
Hence we can say that the blow-up of a point permits to obtain $(\M^2,\mathrm{Isom}(\M^2))$  as a limit of $(\dS^2,\mathrm{Isom}(\dS^2))$.

Also in this case we shall obtain the dual transition. As in the case of Euclidean/co-Euclidean plane, there are two possible transitions having limit the co-Minkowski plane. Indeed, both hyperbolic plane and de Sitter plane contain  hyperbolic lines (space-like in $\H^2$, time-like in $\dS^2$). For the hyperbolic plane, one checks directly that
$$g_t^*\mathrm{Isom}(\H^2)(g_t^*)^{-1}\to\mathrm{Isom}(^*\M^2)=\left\{\begin{bmatrix} a_{11} & a_{12} & 0 \\ a_{21} & a_{22} & 0 \\ v_1 & v_2 & \pm 1 \end{bmatrix}:\begin{pmatrix} a_{11} & a_{12} \\ a_{12} & a_{22} \end{pmatrix}\in\mathrm{O}(1,1)\right\}\,.$$

Observe that, in terms of geometric transition, if $(\mathbb{N},\mathrm{Isom}(\mathbb{N}))$ is a conjugacy limit of $(\mathbb{M},\mathrm{Isom}(\mathbb{M}))$, then also the anti-isometric space  $(\overline{\mathbb{N}},\mathrm{Isom}(\overline{\mathbb{N}}))$ is a limit of $(\mathbb{M},\mathrm{Isom}(\mathbb{M}))$ (and also of $(\overline{\mathbb{M}},\mathrm{Isom}(\overline{\mathbb{M}}))$, of course), by conjugating for the same projective transformations. For instance, both $\dS^2$ and $\overline{\dS^2}$ limit to $^*\M^2$. However, we prefer to say that $(\overline{\dS^2},\mathrm{Isom}(\overline{\dS^2}))$ limits to $(^*\M^2,\mathrm{Isom}(^*\M^2))$ by blowing-up a space-like line, since with this choice space-like lines (of hyperbolic type) converge to space-like lines of $^*\M^2$.

For the same reason, blowing-up a point one gets that $\overline{\dS^2}$ limits to $\overline{\M^2}$. However, in this special case, the space $\overline{\M^2}$, anti-isometric to $\M^2$, is also isometric to $\M^2$, thus we have a geometric transition $(\overline{\dS^2},\mathrm{Isom}(\overline{\dS^2}))\to({\M}^2,\mathrm{Isom}({\M^2}))$. This is formally not the same as the transition of $\dS^2$ to $\M^2$ (in $\dS^2$ we have space-like elliptic lines which converge to space-like lines of $\M^2$, while in $\overline{\dS^2}$ the space-like lines are hyperbolic and converge to space-like lines), although it is obtained for instance by applying the usual transformations $g_t$. Recall also that the dual of $\overline{\dS^2}$, considered as the space of space-like lines, is $\overline{\dS^2}$ itself 
{(see Figure~\ref{fig:dualh2}).}

By the same argument as in the previous paragraph (using the ambient metric of $\M^{3}$ on the left, and its anti-isometric $\overline{\M^{3}}$ on the right), one obtains the following diagram which shows the behavior of transitions and dualities, {see Figure~\ref{fig:deg2dm}}:

\begin{equation}\label{eq:deg 2d lor}
\xymatrix{
(\dS^2,\mathrm{Isom}(\dS^2)) \ar[rr]^-{\text{blow-up point}} \ar@{<->}[d]^-{*} & & (\M^2,\mathrm{Isom}(\M^2)) \ar@{<->}[d]^-{*} & & (\overline{\dS^2},\mathrm{Isom}(\overline{\dS^2})) \ar[ll]_-{\text{blow-up point}} \ar@{<->}[d]^-{*} \\
(\H^2,\mathrm{Isom}(\H^2)) \ar[rr]^-{\text{blow-up line}} & & (^*\M^2,\mathrm{Isom}(^*\M^2)) & & (\overline{\dS^2},\mathrm{Isom}(\overline{\dS^2})) \ar[ll]_-{\text{blow-up line}}
}
\end{equation}

\begin{SCfigure}
     \includegraphics[angle=90,origin=c,scale=0.07]{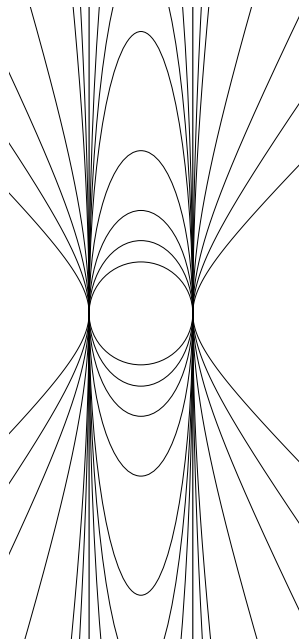}
\caption{\label{fig:deg2dm}
Degenerations of the bottom line in Scheme \eqref{eq:deg 2d lor}, in  an affine chart {(Recall that $\overline{\dS^2}$ is the same as $\AdS^2$}).}
\end{SCfigure}

In words,

\begin{fact}
The dual in $^*\M^2$ of a rescaled limit $x_\infty\in\M^2$ of points $x(t)\in\dS^2$ is the rescaled limit of the dual lines $x(t)^*$ in $\H^2$. The dual in $\M^2$ of a rescaled limit $x_\infty\in\,^*\M^2$ of points $x(t)\in\H^2$ is the rescaled limit of the dual lines $x(t)^*$ in $\dS^2$. The converse is also true, by exchanging the role of points and lines.
\end{fact}

\begin{fact}
The dual in $^*\M^2$ of a rescaled limit $x_\infty\in\M^2$ of points $x(t)\in\overline{\dS^2}$ is the rescaled limit of the dual lines $x(t)^*$ in $\overline{\dS^2}$ and vice versa (by exchanging the roles of points and lines).
\end{fact}

\subsection{Limits of 3-dimensional model spaces} \label{subsec Limits of 3-dimensional model spaces}

At this stage, the reader will not be surprised to find that the transition procedures described in the previous paragraph extend also to the three-dimensional case (and to higher dimensions, although this will not be considered in this survey). For instance, the following diagram summarizes the transitions which have limits in the Euclidean space or the co-Euclidean space:

\begin{equation}\label{eq: deg 3d eucl coeucl}
\xymatrix{
(\Ell^3,\mathrm{Isom}(\Ell^3)) \ar[rr]^-{\text{blow-up point}} \ar@{<->}[d]^-{*} & & (\E^3,\mathrm{Isom}(\E^3)) \ar@{<->}[d]^-{*} & & (\H^3,\mathrm{Isom}(\H^3)) \ar[ll]_-{\text{blow-up point}} \ar@{<->}[d]^-{*} \\
(\Ell^3,\mathrm{Isom}(\Ell^3)) \ar[rr]^-{\text{blow-up plane}} & & (^*\E^3,\mathrm{Isom}(^*\E^3)) & & (\dS^3,\mathrm{Isom}(\dS^3)) \ar[ll]_-{\text{blow-up plane}}
}
\end{equation}

Indeed, the way to rescale elliptic space or de Sitter space to get a limit in co-Euclidean space is by \emph{blowing-up a plane}\index{blow-up! of a plane}. Space-like planes in $\Ell^3$ and $\dS^3$ are indeed copies of $\Ell^2$, and $^*\E^3$ is the degenerate geometry having an embedded $\Ell^2$ plane. It is thus an exercise to rewrite the statements of Facts~\ref{rescaleEllE2} and \ref{rescaleHE2} in the three-dimensional setting, by making use of the duality points/planes.

The other diagram we considered in dimension 2 also has a generalization here. In fact, it will now become clear that the right model spaces which have a limit in Minkowski space (by blowing-up a point) are de Sitter and Anti-de Sitter space; whereas their duals (hyperbolic space and Anti-de Sitter space itself) have a limit in co-Minkowski space, {see Figure~\ref{fig:blow3d}}:

\begin{equation}\label{eq: deg 3d}
\xymatrix{
(\dS^3,\mathrm{Isom}(\dS^3)) \ar[rr]^-{\text{blow-up point}} \ar@{<->}[d]^-{*} & & (\M^3,\mathrm{Isom}(\M^3)) \ar@{<->}[d]^-{*} & & ({\AdS}^3,\mathrm{Isom}({\AdS}^3)) \ar[ll]_-{\text{blow-up point}} \ar@{<->}[d]^-{*} \\
(\H^3,\mathrm{Isom}(\H^3)) \ar[rr]^-{\text{blow-up plane}} & & (^*\M^3,\mathrm{Isom}(^*\M^3)) & & ({\AdS}^3,\mathrm{Isom}({\AdS}^3)) \ar[ll]_-{\text{blow-up plane}}
}
\end{equation}
 
\begin{SCfigure}
     \includegraphics[angle=270,origin=c,scale=0.07]{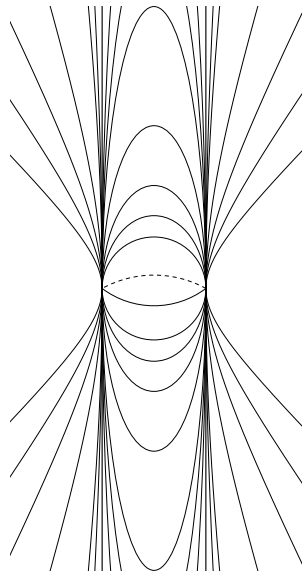}
\caption{\label{fig:blow3d}
Degenerations of the bottom line in Scheme \eqref{eq: deg 3d}, in  an affine chart.
}
\end{SCfigure}

In fact, in dimension 2 we did not define an Anti-de Sitter plane, as it would be anti-isometric to the de Sitter plane. Thus we write here the statement which encodes the relationship between transitions and dualities for Anti-de Sitter space, leaving the analogous statements to the reader:

\begin{fact}
The dual in $^*\M^3$ of a rescaled limit $x_\infty\in\M^3$ of points $x(t)\in{\AdS}^3$ is the rescaled limit of the dual  planes $x(t)^*$ in ${\AdS}^3$ and vice versa (by exchanging the roles of points and planes).
\end{fact}

\paragraph{Blowing-up time-like planes.}
However, we observe that in dimension three there are more complicated pheonomena which might occur. In particular,  one might be interested in the blow-up of a time-like  plane in a Lorentzian model-space. Of course this is the same of blowing-up a space-like plane in the anti-isometric model space. Intrinsically, the time-like plane can be a copy of $\dS^2$ or $\overline{\dS^2}$. Let us analyze the following diagram:

\[
\xymatrix{
(\overline{\dS^3},\mathrm{Isom}(\overline{\dS^3})) \ar[rr]^-{\text{blow-up point}} \ar@{<->}[d]^-{*} & & (\overline{\M^3},\mathrm{Isom}(\overline{\M^3})) \ar@{<->}[d]^-{*} & & (\overline{\AdS^3},\mathrm{Isom}(\overline{\AdS^3})) \ar[ll]_-{\text{blow-up point}} \ar@{<->}[d]^-{*} \\
({\dS}^3,\mathrm{Isom}({\dS}^3)) \ar[rr]^-{\text{blow-up plane}} & & (^*\overline{\M^3},\mathrm{Isom}(^*\overline{\M^3})) & & ({\overline{\AdS^3}},\mathrm{Isom}({\overline{\AdS^3}})) \ar[ll]_-{\text{blow-up plane}}
}
\]

On the upper line, there is nothing surprising. We have just re-written the usual blow-up of a point, but by considering the anti-isometric copies of $\dS^3$, $\M^3$ and $\AdS^3$. In fact, the definition of geometric transition does not distinguish between a model space and its anti-isometric copy. Anyway, we decided to stick to the convention to choose the sign of the metric in such a way that space-like lines converge to space-like lines, and so on. 

On the lower line, we have already encountered the duality which appears on the left. In fact,  in $\M^4$ the linear hyperplanes which define time-like planes in $\dS^3$ also define planes in $\H^3$ (see Remark \ref{remark duality}). Hence the space of time-like planes in $\dS^3$, or equivalently the space of space-like planes in $\overline{\dS^3}$, is naturally $\dS^3$ itself.

In the right-hand side, recall  that $\overline{\AdS^3}$ can be defined as:
$$\overline{\AdS^3}=\{x\in \R^{2,2} \,|\, b_{2,2}(x)=1 \} / \{\pm \mathrm{Id}\}\,.$$
Thus a space-like plane in $\overline{\AdS^3}$ is a copy of $\dS^2$, and the dual space of $\overline{\AdS^3}$, considered as the space of space-like planes of $\overline{\AdS^3}$, is $\overline{\AdS^3}$ itself. 

It remains to understand what is the limit space in the center of the lower line. We denoted it by $^*\overline{\M^3}$ to indicate that it is the dual of $\overline{\M^3}$ (and not the space {$\overline{^*\M^3}$} anti-isometric to $^*{\M}^3$!)
We will omit the details of the definition and the proof of the commutativity of the last diagram presented. However, observe that in the lower line, both ${\dS}^3$ and $\overline{\AdS^3}$ contain a totally geodesic copy of ${\dS}^2$. Thus one can define a transition procedure which \emph{blows-up a time-like plane}, stretching the transverse directions. Topologically the limit space is expected to be ${\dS}^2\times\R$. By a construction similar to that of $^*\M^3$, one can identify this space to the space of time-like planes in $\M^3$ (or of space-like planes in $\overline{\M^3}$). In an affine chart, this would be the exterior of the cylinder which represents $^*\M^3$.

\paragraph{Comments and references} {\small{
\begin{itemize}
\item Let us denote by $\Gal^2$ the \emph{Galilean plane}\index{Galilean plane}, i.e. the projective plane minus a parabolic line. {See \cite{yaglom}.} Below are shown  the possible degenerations of the three Riemannian {and the three Lorentzian} plane geometries. This diagram is the one  
in Section~5.3 of \cite{CDW}, adapted to our terminology.
\begin{displaymath}
\xymatrix{ \Ell^2 \ar[d] \ar[drr] & \H^2 \ar[dl]\ar[d] & \overline{\dS^2}\ar[dl]\ar[d]\ar[dr] & \\ \E^2\ar[dr] & ^*\M^2\ar[d] & ^*\E^2\ar[dl] & \M^2 \ar[dll] \\  & \Gal^2 & & }
\end{displaymath}
This essentially shows that the geometric transitions considered in this paper are all the possible transitions in $\R\mathrm{P}^3$, except the further space $\Gal^2$ which is, in some vague sense, doubly-degenerate. Essentially, the isometry group for $\Gal^2$ is the subset of triangular matrices which preserve an affine chart. %In higher dimension, there are many more possible cases of more degenerate spaces, but in this survey we will not consider them, thus restricting to simple degeneracies.
\item In \cite{CDW}, the possible degenerations are classified in every dimension. Already in dimension 3, apart from those considered above, there are other \emph{simple} degenerations, essentially obtained by blowing up a line. These spaces are described as model spaces $(X,\mathrm{Isom}(X))$ and in some cases they contain as subgeometries (i.e. as a geometry $(X,G)$ where $G$ is a subgroup of $\mathrm{Isom}(X)$) other 3-dimensional geometries in the sense of Thurston. For instance $\mathrm{Sol}$ geometry is a possible limit of hyperbolic structures \cite{kozai}. One could draw the corresponding diagram in dimension 3, which would already be pretty complicated, and find several other spaces as double degenerations. For example, an affine space endowed with the action of a group of lower triangular matrices, with unitary elements on the diagonal, is a generalization of $\Gal^2$.
\item 
For the effect of the two-dimensional transition hyperbolic-Euclidean-spherical, the effects on angles, area etc. of triangles are studied in \cite{AP3}.

\end{itemize}
}
}

\section{Connection and volume form} \label{sec connection volume}

In this section we discuss the definitions of the Levi-Civita connection and volume form for model spaces, starting by the general setting of Riemannian or pseudo-Riemannian manifolds and then specializing to the cases of constant curvature manifolds. We will then give a construction of a geometric connection and a volume form on the degenerate cases of co-Euclidean and co-Minkowski space.

\subsection{Non-degenerate model spaces}

Let $(M,g)$ be a three-dimensional Riemannian or pseudo-Riemannian manifold. Many key examples have already been introduced: $\E^3$, $\H^3$ and $\Ell^3$ are Riemannian, $\M^3$, $\dS^3$ and $\AdS^3$ are pseudo-Riemannian. Recall that the non-degenerate metric $g$ uniquely determines the \emph{Levi-Civita connection}\index{connection (Levi-Civita)}, which we will denote by  $\nabla^{g}$ or $\nabla^{M}$ if there is no ambiguity about the metric tensor $g$.
Indeed, the Levi-Civita connection is the unique connection which satisfies the following properties:
\begin{itemize}
 \item $\nabla^{g}$ is \emph{symmetric}, i.e. $\nabla^{g}_X Y-\nabla^{g}_Y X=[X,Y]$ for any pair of vector fields $X,Y$;
 \item $\nabla^{g}$ is \emph{compatible with the metric}, i.e. $Z.g(X,Y)=g(\nabla^{g}_Z X,Y)+g(X,\nabla^{g}_Z Y)$ for any vector fields $X,Y,Z$.
\end{itemize}
The latter condition can also be expressed by saying that $\nabla^{g} g=0$, namely the metric tensor is parallel with respect to $\nabla^{g}$. Recall also that a \emph{volume form}\index{volume form} is defined on $(M,g)$, provided $M$ is oriented. Namely, the volume form is a 3-form $\omega$ such that
$\omega(X,Y,Z)=1$ if $X,Y,Z$ is an oriented (with respect to the orientation of $M$) orthonormal frame for $g$. The volume form is also characterized by being the unique 3-form $w$ such that:
\begin{itemize}
\item $\omega$ is \emph{parallel} with respect to $\nabla^g$, i.e. $\nabla^g \omega=0$;
\item $\omega(v_1,v_2,v_3)=1$, where $v_1,v_2,v_3$ is an oriented orthonormal triple of vectors at a fixed point $x_0$.
\end{itemize}
Indeed, as the parallel transport  preserves an oriented orthonormal basis, the volume form is parallel for the connection, i.e. $\nabla^g \omega=0$, i.e.
\begin{equation}\label{eq:forme vol par}Z.\omega(X_1,\ldots,X_n)=\omega(\nabla_Z^g X_1,\ldots,X_n)+{\cdots +}\omega(X_1,\ldots,\nabla_Z^g X_n)~. \end{equation}

\paragraph{Explicit construction of the ambient Levi-Civita connection and volume form.}
Let us now go back one step, and see how the Levi-Civita connection of the three-dimensional manifolds considered here can be defined. For instance, the Levi-Civita connection of Euclidean space is simply given by differentiation of the standard coordinates of a vector field, that is,
$$\nabla^{\E^3}_v w=Dw(v)\,,$$
for any pair of smooth vector fields $v,w$. The connection of Minkowski space is defined analogously, and thus on the same affine space, the Levi-Civita connections of the Euclidean and the Minkowski metric coincide. Moreover, the {standard} volume form of {$\R^3$ coincides with the volume forms (induced by the metric) of}  $\E^3$ and $\M^3$, {i.e.} in the standard coordinates $(x,y,z)$: 
\begin{equation}\label{eq:meme forme volume}\omega_{\E^3}=\omega_{\M^3}=\d x\wedge \d y\wedge \d z\,{=\omega_{\R^3}}.\end{equation}

Of course, the same definitions can be given for any $n$-dimensional vector space endowed with a non-degenerate symmetric bilinear form $b$. Now we will use this fact to define the Levi-Civita connection and the volume form of the model spaces. We have already observed that for all {non degenerate} model spaces $\mathbb{M}=\mathcal{M}/\{\pm\mathrm{Id}\}$ which we have defined as the (projectivization of the) subset $\mathcal{M}$ of $(\R^4,b)$ where the quadratic form defined by $b$ takes the value $1$ (or $-1$), the normal vector to $\mathcal{M}$ at a point $x$ is precisely $x$ itself. Thus,  by the definition of the Levi-Civita connection of an embedded hypersurface in a higher-dimensional manifold, we obtain, for vector fields $v,w$ tangent to $\mathcal{M}$,
$$Dw(v)=\nabla^{\mathcal M}_v w-b(v,w)x\,.$$ We used that the  identity and the second fundamental form coincide (up to a sign) with the first fundamental form (see also Section \ref{sec geometry surfaces} below). Clearly this definition descends to the definition of the Levi-Civita connection of $\mathbb M$.

Also the volume form of $\mathcal{M}$ (and thus of $\mathbb{M}$) can be defined in terms of the ambient volume form of {$\R^4$}. Indeed, given any triple of vectors $v,w,u$ in $T_x \mathcal{M}$, $x$ is also the normal vector to $\mathcal{M}$, and thus one can define
$$\omega_{\mathcal{M}}(v,w,u)=\omega_{{\R^4}}(x,v,w,u)\,.$$

\subsection{Degenerate cases}\label{sec deg case}

We will now introduce a natural connection and volume form on $^*\E^3$ and $^*\M^3$. We will discuss the meaning of naturality in the following, but of course we can anticipate that a natural connection/volume form will be preserved by the isometry group of co-Euclidean (resp. co-Minkowski) space, as introduced in Definitions \ref{defi isom coeucl} and \ref{defi isom comink}. %We will show that this notion is the rescaled limit of the known connections and curvatures for $\AdS^3$ and $\H^3$ under the transitional procedures we discussed in Section \ref{sec:geometric transition}. See for instance \cite{Lee:2004wp} for more details about the standard theory in Riemannian geometry.

\paragraph{The connection of co-Euclidean space.}

To define a connection on $^*\E^3$, we start by defining a connection on its double cover, namely
$$\cE^3=\{(x_1,x_2,x_3,x_4)\,|\,x_1^2+x_2^2+x_3^2=1\}{=
\{x\in\R^ 4\,|\,b^*(x,x)=1\}}\,,$$
and thus we will consider $\mathcal S^2\times \R$ as a model for $\cE^3$. In the classical case of model spaces {defined by a non degenerate symmetric quadratic form $b$}, the key ingredient to define the Levi-Civita connection was the existence of the normal vector field $\textsc{N}$, so as to be able to write
$$Dw(v)=\nabla^{\mathcal M}_v w+b(v,w)\textsc{N}\,.$$
Clearly the normal vector field has the property that it is preserved by the group of isometries of the ambient quadratic form $b$. In this degenerate case, the bilinear form $b^*$ is degenerate, and thus it does not enable us to determine a unit normal vector field. However, %it is still true that the isometry group of $\mathcal M$, which is the group of linear transformations of the form
there is a well-defined transverse vector field to $\cE^3$, namely the vector field which at the point $x\in \cE^3\subset \R^4$ is defined by
$$\textsc{N}_x=x\in \R^4\,,$$
Tautologically, this vector field $\textsc{N}$ is preserved by the group $\mathrm{Isom}(^*\E^3)$, which means that if $A\in\mathrm{Isom}(^*\E^3)$, then $A_*(\textsc{N}_x)=\textsc{N}_{A(x)}$. Thus one can use the vector field $\textsc{N}$ to decompose the ambient derivative of two vector fields in a tangential and a ``normal'' component.

\begin{df}[Co-Euclidean connection] \label{defi coeucl connection}
Given two vector fields $v,w$ in $\cE^3$, we define the connection $\nabla^{\cE^3}$ by means of:
$$Dw(v)=\nabla^{\cE^3}_v w+b^*(v,w)\textsc{N}\,,$$
The co-Euclidean connection is the connection $\nabla^{^*\E^3}$ induced on $^*\E^3$ by $\nabla^{\cE^3}$.
%where $D_V W$ is the usual flat connection of the ambient $\R^{2,0,1}$ obtained by differentiating each component, and $(\bullet)^T$ denotes the projection on the half-pipe model determined by the splitting
%\begin{equation} \label{eq:splitting}
%\R^{2,0,1}=T_x^*\M^3\oplus \R x\,.
%\end{equation}
\end{df}

First, it should be clear that the co-Euclidean connection is preserved by the group of isometries of co-Euclidean space. We prove it now, and in fact this also follows from the characterization given in Proposition~\ref{prop properties coeucl connection} below.

\begin{lemma}
The co-Euclidean connection $\nabla^{^*\E^3}$ is invariant for the group $\mathrm{Isom}(^*\E^3)$.
\end{lemma}
\begin{proof}
From the definition, we have
$$\nabla^{\cE^3}_{A_*v}A_* w=D(A_*w)(A_* v)-b^*(A_* v,A_* w)\textsc{N}_{A(x)}=A_*(Dw(v)-b^*(v,w)\textsc{N}_x)=\nabla^{\cE^3}_v w\,,$$
since the ambient connection $D$ and the vector field $\textsc{N}$ are invariant for the action of $\mathrm{Isom}(^*\E^3)$.
\end{proof}

We will denote by $\textsc{T}$ the vector field on $\cE^3$ defined by $(0,1)$ in $T_x \cE^3\cong T_x\mathcal{S}^2\times \R$. It is a degenerate vector field invariant for the group $\mathrm{Isom}(\cE^3)$, as seen from the form inside the brackets of {\eqref{eq:isom co}}. Observe that $\textsc{T}$ does not descend to a global vector field on ${^*\E^3}$, but we will still talk about the vector field $\textsc{T}$ (by an abuse of notation) as the vector field induced on any simply connected open subset of ${^*\E^3}$.

\begin{proposition} \label{prop properties coeucl connection}
The connection $\nabla^{^*\E^3}$  is the unique connection on $^*\E^3$ such that:
\begin{itemize}
\item $\nabla^{^*\E^3}$ is \emph{symmetric}, i.e. $\nabla^{^*\E^3}_X Y-\nabla^{^*\E^3}_Y X=[X,Y]$ for any pair of vector fields $X,Y$; \item $\nabla^{^*\E^3}$ is \emph{compatible} with the degenerate metric $g^*$ of $^*\E^3$, i.e. $Z.g^*(X,Y)=g^*(\nabla^{^*\E^3}_Z X,Y)+g^*(X,\nabla^{^*\E^3}_Z Y)$ for any vector fields $X,Y,Z$.

\item $\nabla^{^*\E^3}$ \emph{preserves every space-like plane of} $^*\E^3$, i.e. for $V,W$ vector fields on a space-like plane $P$, $\nabla^{^*\E^3}_V W$ is tangent to $P$;
\item The vector field $\textsc{T}$ is \emph{parallel} with respect to $\nabla^{^*\E^3} $, i.e. $\nabla^{^*\E^3} \textsc{T}=0$.
\end{itemize}
In particular, the restriction of $\nabla^{^*\E^3}$ to any  space-like plane coincides with the Levi-Civita connection for the induced metric.
%Moreover, the connection is invariant for the group $\mathrm{isom}(^*\M^3)$. Geodesics and totally geodesic planes coincide with lines and planes defined in the projective geometry of $^*\M^3$.
\end{proposition}
\begin{proof}
It suffices to prove the statement for the double cover $\cE^3$, since clearly $-\mathrm{Id}$ acts on $\cE^3$ as an isometry of $\cE^3$, and preserving $\nabla^{\cE^3}$. It is straightforward to check that $\nabla^{\cE^3}$ defines a connection on $\cE^3$.
Symmetry follows from the observation that
$\nabla^{\cE^3}_X Y-\nabla^{\cE^3}_Y X$ is the tangential component of $(D_X Y-D_Y X)=[X,Y]$, which equals $[X,Y]$ itself. 
Also compatibility is very simple: for every vector $Z$ tangent to $\cE^3$, 
\begin{equation*}
Z.g^*(X,Y)=g^*( D_Z X ,Y)+g^*( X,D_Z Y)
=g^*( \nabla^{\cE^3}_Z X,Y)+g^*( X,\nabla^{\cE^3}_Z Y)\,.
\end{equation*}
For the third point, let $P$ be a plane of $\cE^3$ obtained as the intersection of $\cE^3$ with a linear hyperplane $P'$ of $\R^4$. Given vector fields $V,W$ on $P$, $D_V W$ is tangent to $P'$ and thus the projection to $\cE^3$ is still in $P$.
Finally, it is clear from the construction that the derivative of $\textsc{T}$ in any direction vanishes. 

Let us now assume that the four conditions hold. In the coordinate system provided by $\mathcal S^2\times\R$, the restriction of $\nabla^{\cE^3}$ to every plane $\mathcal S^2\times\{*\}$ preserves the plane itself (by the third point) and coincides with the Levi-Civita connection of $\mathcal S^2$, by the second point. 
 
Hence it is easily seen that the Christoffel symbols $\Gamma_{ij}^k$ are those of the Levi-Civita connection when $i,j,k$ correspond to coordinates of $\mathcal{S}^2$. Otherwise, using the first and fourth hypothesis, the $\Gamma_{ij}^k$ vanish. Hence the connection is uniquely determined.%\\ The last claims in the statement follow directly from the construction.
\end{proof}

\begin{cor}
The geodesics for the co-Euclidean connection $\nabla^{^*\E^3}$ coincide with the lines of $^*\E^3$.
\end{cor}
\begin{proof}
Again, we prove the statement for the double cover $\cE^3$. Given a space-like line $l$ of $\cE^3$, using the action of $\mathrm{Isom}(\cE^3)$ we can assume that $l$ is contained in the slice $\mathcal S^2\times\{0\}$. Since the connection on such a slice coincides with the Levi-Civita connection, and lines of $\cE^3$ are geodesics for this copy of $\mathcal S^2$, $l$ is a geodesic for $\nabla^{\cE^3}$. If $l$ is not space-like, then it is of the form $\{*\}\times\R$. Since $\nabla^{\cE^3}_\textsc{T} \textsc{T}=0$, it is clear by construction that $l$ is geodesic, provided it is parametrized in such a way that its tangent vector is a fixed multiple of $\textsc{T}$ for all time. 

Since there is a line of $\cE^3$ through every point of $\cE^3$ with every initial velocity, this shows that all geodesics for the connection $\nabla^{\cE^3}$ are lines of $\cE^3$.
\end{proof}

\paragraph{The volume form of co-Euclidean space.} By means of the transverse vector field $N$ in $\R^4$, we can also perform the usual construction to define a volume form for co-Euclidean space. 
Indeed, we can give the following definition:

\begin{df}[Co-Euclidean volume form] \label{defi coeucl volume}
The volume form of $\cE^3$ is the 3-form $\omega_{\cE^3}$ such that, given vectors $v,w,u$ in $T_x\cE^3$, 
$$\omega_{\cE^3}(v,w,u)=\omega_{\R^4}(\textsc{N},v,w,u)\,.$$
The co-Euclidean volume form is the volume form $\omega_{^*\E^3}$ induced on $^*\E^3$ by $\omega_{\cE^3}$.
%where $D_V W$ is the usual flat connection of the ambient $\R^{2,0,1}$ obtained by differentiating each component, and $(\bullet)^T$ denotes the projection on the half-pipe model determined by the splitting
%\begin{equation} \label{eq:splitting}
%\R^{2,0,1}=T_x^*\M^3\oplus \R x\,.
%\end{equation}
\end{df}

Both the volume form of the ambient $\R^4$ and the vector field $\textsc{N}$ are invariant for $\mathrm{Isom}(^*\E^3)$, hence clearly:

\begin{lemma}
The volume form $\omega_{\cE^3}$ is invariant for the group $\mathrm{Isom}(^*\E^3)$.
\end{lemma}

Of course, as there is no Riemannian or pseudo-Riemannian metric invariant by the group of isometries of $^*\E^3$ {(Fact~\ref{fact no met coeucl})}, there is no volume arising as the volume associated to a metric. However, the volume form $\omega_{^*\E^3}$ has the property that $\omega_{^*\E^3}(v,w,\textsc{T})=1$ provided $v,w$ are orthonormal space-like vectors for the degenerate metric of ${^*\E^3}$, and $\textsc{T}$ is the ``unitary'' degenerate vector field, so that the triple $(v,w,\textsc{T})$ is positively oriented. As in the classical case, since parallel transport preserves both the degenerate metric (this follows from the compatibility with the metric) and the vector field $\textsc{T}$ (see Proposition \ref{prop properties coeucl connection}), the volume form $\omega_{^*\E^3}$ has the following characterization:

\begin{proposition} 
The volume form $\omega_{^*\E^3}$  is the unique volume form on $^*\E^3$ such that:
\begin{itemize}
\item $\omega_{^*\E^3}$ is \emph{parallel} with respect to $\nabla^{^*\E^3}$, i.e. $\nabla^{^*\E^3} \omega_{^*\E^3}=0$;
\item $\omega_{^*\E^3}(v,w,\textsc{T})=1$, where $v,w,\textsc{T}$ is an oriented triple at a fixed point $x_0$, such that $v,w$ are orthonormal space-like vectors.
\end{itemize}
\end{proposition}

\paragraph{The case of co-Minkowski space.}
Very similar constructions can be used to define a connection and a volume form for co-Minkowski space. Indeed,  we define the vector field $\textsc{N}$, which is transverse to $\cM^3\cong \mathcal H^2\times\R$ in the ambient space $\R^4$, as the vector field $\textsc{N}_x=x\in\R^4$.
Clearly this definition is invariant by the isometries of $\cM^3$, hence in particular by the involution which identifies the two connected components of $\mathcal H^2\times\R$. For simplicity, we identify $^*\M^3\cong\H^2\times\R$ as one of the two connected components of $\mathcal H^2\times\R$.  
Recalling that the bilinear form for co-Minkowski space $b^*$ has the form $x_1^2+x_2^2-x_3^2$ in the $(x_1,x_2,x_3,x_4)$-coordinates, using the ambient connection and volume form of $\R^4$, one obtains the definitions of co-Minkowski connection and volume form:

\begin{df}[Co-Minkowski connection] \label{defi comink connection}
Given two vector fields $v,w$ in $^*\M^3$, we define the connection $\nabla^{^*\M^3}$ by means of:
$$Dw(v)=\nabla^{^*\M^3}_v w+b^*(v,w)\textsc{N}\,.$$
\end{df}

\begin{df}[Co-Minkowski volume form]
The volume form of $^*\M^3$ is the 3-form $\omega_{^*\M^3}$ such that, given vectors $v,w,u$ in $T_x ^*\M^3$, 
$$\omega_{^*\M^3}(v,w,u)=\omega_{\R^4}(\textsc{N},v,w,u)\,.$$
\end{df}

We report on the key properties below, without giving the proofs, as they are completely analogous to the co-Euclidean case. 

\begin{lemma}
The co-Minkowski connection $\nabla^{^*\M^3}$ and volume form $\omega_{\cE^3}$ are invariant for the isometry group $\mathrm{Isom}(^*\M^3)$.
\end{lemma}

Here we denote by $\textsc{T}$ the normalized degenerate vector field (as we recall that there is a notion of length of the degenerate direction which is preserved by the group of isometries of $^*\M^3$). In the $(x_1,x_2,x_3,x_4)$-coordinates, $\textsc{T}$ can be written as $(1,0,0,0)$.

\begin{proposition} \label{prop properties half pipe connection}
The connection $\nabla^{^*\M^3}$  is the unique connection on $^*\M^3$ such that:
\begin{itemize}
\item $\nabla^{^*\M^3}$ is \emph{symmetric}, i.e. $\nabla^{^*\M^3}_X Y-\nabla^{^*\M^3}_Y X=[X,Y]$ for any pair of vector fields $X,Y$; \item $\nabla^{^*\M^3}$ is \emph{compatible} with the degenerate metric $g^*$ of $^*\M^3$, i.e. $\nabla^{^*\M^3} g^*=0$;

\item $\nabla^{^*\M^3}$ \emph{preserves every space-like plane of} $^*\M^3$, i.e. for $V,W$ vector fields on a space-like plane $P$, $\nabla^{^*\M^3}_V W$ is tangent to $P$;
\item The vector field $\textsc{T}$ is \emph{parallel} with respect to $\nabla^{^*\M^3} $, i.e. $\nabla^{^*\M^3} \textsc{T}=0$.
\end{itemize}
\end{proposition}

\begin{cor}
Geodesics for the co-Minkowski connection $\nabla^{^*\M^3}$ coincide with lines of $^*\M^3$.
\end{cor}

\begin{proposition} \label{prop properties comink connection}
The volume form $\omega_{^*\M^3}$  is the unique volume form on $^*\M^3$ such that:
\begin{itemize}
\item $\omega_{^*\M^3}$ is \emph{parallel} with respect to $\nabla^{^*\M^3}$, i.e. $\nabla^{^*\M^3} \omega_{^*\M^3}=0$;
\item $\omega_{^*\M^3}(v,w,\textsc{T})=1$, where $v,w,\textsc{T}$ is an oriented triple at a fixed point $x_0$, such that $v,w$ are orthonormal space-like vectors.
\end{itemize}
\end{proposition}

\paragraph{Relation with geometric transition.} 
We will now show that the connections of co-Euclidean and co-Minkowski space are natural also in the sense of geometric transition, that is, they are the limits of the Levi-Civita connections under the transitions described in Subsection \ref{subsec Limits of 2-dimensional model spaces} and \ref{subsec Limits of 3-dimensional model spaces}. We will focus on dimension 3 for definiteness, and consider the transition of $\Ell^3$ and $\dS^3$ to $^*\E^3$, as described in Scheme \eqref{eq: deg 3d eucl coeucl}, which makes use of the projective transformations 
\begin{equation} \label{proj resc dim 4}
g_t^*=\begin{bmatrix} 1 & 0 & 0 & 0 \\ 0 & 1 & 0 & 0 \\ 0 & 0 &  1 & 0 \\ 0 & 0 & 0 & 1/t  \end{bmatrix}\in \mathrm{PGL}(4,\R)\,.
\end{equation}

\begin{proposition} \label{prop limit connection volume coeucl}
Let $X_t,Y_t,Z_t$ be smooth families of smooth vector fields on $\Ell^3$ (resp. $\dS^3$) such that $X_0,Y_0,Z_0$ are tangent to a  plane $P$. The limit of $g_t^*X_t$ is a vector field on $^*\E^3$ which we denote by $\dot X$, and analogously for $\dot Y$ and $\dot Z$. Then:
\begin{itemize}
\item The limit of $ g_t^*\nabla^{\Ell^3}_X Y$ (resp. $ g_t^*\nabla^{\dS^3}_X Y$) is $\nabla^{^*\E^3}_{\dot X}\dot Y$.
\item The limit of $\omega^{\Ell^3}(X_t,Y_t,Z_t)$ (resp. $\omega^{\dS^3}(X_t,Y_t,Z_t)$)  is $\omega^{^*\E^3}(\dot X,\dot Y,\dot Z)$.
\end{itemize}
\end{proposition}
\begin{proof}
First, let us assume that $P$ is defined by $x_4=0$ in $\R\mathrm{P}^3$, as in Subsection \ref{subsec Limits of 3-dimensional model spaces}, so that we can use the projective transformations
in \eqref{proj resc dim 4}. We can as usual perform the computation in the double cover $\cE^3$.  Hence in those coordinates, the vector field $X_0,Y_0,Z_0$ have $x_4$-coordinate equal to 0. 
Hence, the limit of $g_t^*X_t$ is:
$$\lim_{t\to 0}g_t^* X_t(x)=\lim_{t\to 0}\begin{bmatrix} (X_t(x))_1 \\ (X_t(x))_2 \\ (X_t(x))_3 \\ (X_t(x))_4/t  \end{bmatrix}=\begin{bmatrix} (X_0(x))_1 \\ (X_0(x))_2 \\ (X_0(x))_3 \\ (\dot X(x))_4 \end{bmatrix}\,,$$
namely, it is a vector field on $^*\E^3$ whose vertical projection is identified with the vector field $X_0(x)$.

Now, observe that $g_t^*$ maps {$\mathcal S^3$} to 
$$\mathcal S^3_t:=g_t^*(\mathcal S^3)=\{x\,|\,x_1^2+x_2^2+x_3^2+t^2x_4^2=-1\}\,.$$ 
Moreover, $g_t^*$ is an isometry if we endow $\mathcal S^3_t$ with the metric induced by the quadratic form $$b_{4,0}^t(x,x)=x_1^2+x_2^2+x_3^2+t^2x_4^2\,.$$
 Hence $g_t^*$ maps the normal vector $\textsc{N}(x)=x$ of {$\mathcal S^3$} to $\textsc{N}_t(x)=g_t^*(x)$, which is precisely the normal vector to $\mathcal S^3_t$ for $b_{4,0}^t$. 
 
Now, as $X_0,Y_0,Z_0$ are tangent to $P$, their base point $x_t$ are rescaled to $g_t^* x_t$ which converge to $x_\infty$, base point for $\dot X,\dot Y,\dot Z$. For the same reason, the normal vector $\textsc{N}_t(x_t)=g_t^*(x_t)$ to  $\mathcal S^3_t$ at $x_t$ converges to $x_\infty$, which is (in the notation of Definitions \ref{defi coeucl connection} and \ref{defi comink connection}) the vector $\textsc{N}(x_\infty)$ used as a transverse vector field to $\cE^3$.

Hence in conclusion,
$$g_t^*\nabla^{\mathcal S^3}_{X_t} Y_t=\nabla^{\mathcal S^3_t}_{g_t^*X_t} g_t^*Y_t=DY_t(X_t)-b_{4,0}^t (X_t,Y_t)\textsc{N}_t(x_t)$$
converges to
$$\nabla^{\cE^3}_{\dot X} \dot Y=D\dot Y(\dot X)-b^*(\dot X,\dot Y)\textsc{N}(x_\infty)\,.$$
Analogously,
$$\omega_{\mathcal S^2}(X_t,Y_t,Z_t)=\omega_{\R^4}(\textsc{N},X_t,Y_t,Z_t)=\omega_{\E^4}(\textsc{N}_t,g_t^*X_t,g_t^*Y_t,g_t^*Z_t)$$
converges to
$$\omega_{\cE^3}(\dot X,\dot Y,\dot Z)=\omega_{\R^4}(\textsc{N}(x_\infty),\dot X,\dot Y,\dot Z)\,.$$
Of course, the proof is completely analogous for the rescaling from $\dS^3$.
\end{proof}

Actually, the proof is analogous also for the case of geometric transition of $\H^3$ and $\AdS^3$ to $^*\M^3$, hence one can prove:

\begin{proposition}
Let $X_t,Y_t,Z_t$ be smooth families of smooth vector fields on $\H^3$ (resp. $\AdS^3$) such that $X_0,Y_0,Z_0$ are tangent to a  plane $P$. The limit of $g_t^*X_t$ is a vector field on $^*\M^3$ which we denote by $\dot X$, and analogously for $\dot Y$ and $\dot Z$. Then:
\begin{itemize}
\item The limit of $ g_t^*\nabla^{\H^3}_X Y$ (resp. $ g_t^*\nabla^{\AdS^3}_X Y$) is $\nabla^{^*\M^3}_{\dot X}\dot Y$.
\item The limit of $\omega^{\H^3}(X_t,Y_t,Z_t)$ (resp. $\omega^{\AdS^3}(X_t,Y_t,Z_t)$)  is $\omega^{^*\M^3}(\dot X,\dot Y,\dot Z)$.
\end{itemize}
\end{proposition}

\subsection{The infinitesimal Pogorelov map}

\paragraph{Weyl formula for connections.}

In an affine chart, a (non-degenerate) model space $\mathbb M$
has the strong property that its (unparametrized) geodesics are the same as in the ambient $\R^n$  endowed with a Euclidean metric. This will imply that, in an open set, the  Levi-Civita connection $\nabla$ of $\mathbb M$ can be written in terms of the usual connection on $\R^n$. The manifold $\mathbb M$ may not be orientable, but on an affine chart we can consider the same orientation as $\R^n$. Hence the pseudo-Riemannian metric of $\mathbb M$ gives a volume form $\omega$ in the affine chart.

Let $N$ be an orientable manifold with two torsion-free connections $\nabla$ and $\widehat \nabla$ and with two volume forms $\omega$ and $\widehat \omega$ such that $\nabla \omega=0$ and $\widehat\nabla \widehat \omega =0$.
Let $\lambda$ be the function on $N$ such that
$$\widehat \omega = \lambda \omega~.$$
Let us introduce $$D(X,Y)=\widehat \nabla_XY - \nabla_XY~.$$ {This function is linear over smooth functions on both argument, so it is a $(1,2)$-tensor. As the connections are symmetric, it is easy to check that $D$ is symmetric.}

\begin{lemma} The contraction of $D$ is a one-form that is given by $\lambda$:
\begin{equation}\label{contraction}
C^1_1 D = \d\ln \lambda~.
\end{equation}
\end{lemma}
\begin{proof}

Recall that  the volume form $\omega$ is parallel for $\nabla$ \eqref{eq:forme vol par}, that gives,  in the given local coordinates on $N$,
{$$\partial_i\omega_{1,\ldots,n}=\omega\left(\sum_{k=1}^n\Gamma_{i1}^ke_k,\ldots,e_n\right)+ \cdots+\omega\left(e_1,\ldots,\sum_{k=1}^n\Gamma_{in}^ke_k\right)~.$$
But 
$\omega$ is alternating and $n$-linear:
$$\partial_i\omega_{1,\ldots,n}=\omega(\Gamma_{i1}^1e_1,\ldots,e_n)+ \cdots+\omega(e_1,\ldots,\Gamma_{in}^ne_n)= \left(\sum_{\alpha=1}^n\Gamma_{i\alpha}^{\alpha}\right) \omega_{1,\ldots,n}\,, $$}
namely
$$\sum_{\alpha=1}^n\Gamma_{i\alpha}^{\alpha} = \partial_i \ln(\omega_{1,\ldots,n}) $$
and in particular,
$$ \sum_{\alpha=1}^n \widehat\Gamma_{i\alpha}^{\alpha}-\sum_{\alpha=1}^n\Gamma_{i\alpha}^{\alpha}=\partial_i \ln \lambda~. $$
\end{proof}

\begin{lemma}[Weyl Formula]\label{weyl}\index{Weyl formula}
Let us suppose that $\nabla$ and $\widehat\nabla$ have the same (unparameterized) geodesics. Then 
$$\widehat \nabla_XY -  \nabla_XY= X.\left(\ln \lambda^{\frac{1}{n+1}}\right)Y + Y.\left(\ln \lambda^{\frac{1}{n+1}}\right) X~. $$
\end{lemma}
\begin{proof}

Observe that in this formula both sides are linear in $X$ and $Y$. To compute the left-hand side, fix $p$ and $X_p\in T_p N$ and let $c:I\to N$ be a parameterized  geodesic for $\nabla$ with tangent vector $X_p$ at $c(0)=p$. Extend $X_p$ to $X$, the tangent vector of $c$, which is parallel along $c$, so that $\nabla_X X=0$.
Let $\widehat X$ be the tangent vector of a reparametrization of $c$ that turns it into a geodesic $\widehat c$ for $\widehat \nabla$ (hence $\widehat \nabla_{\widehat X} \widehat X=0$).
Then there is a function {$f$} on $c$ (which depends on the choice of $X_p$ and $\widehat X_p$) such that $X=f \widehat X$.
We compute:
$$(\widehat \nabla_{X} X - \nabla_{X} X)(p) = (\widehat \nabla_{X} X)(p)
=  \widehat \nabla_{f(p) \widehat X_p} (f \widehat X )= f(p) (\widehat X_p.f)\widehat X_p   + f(p)^2 \widehat \nabla_{\widehat X_p}\widehat X_p =f(p) (\widehat X_p.f)\widehat X_p\,,$$
so $D(X,X) = X.\ln (f)X $.
 Let us define  $\phi(X)=X.\ln(f_X)/2$, so that 
 \begin{equation} \label{eq spivak phi}
 D(X,X)=2\phi(X)X\,.
 \end{equation} 
 Let us show that $\phi$ is a 1-form. In fact, given any symmetric tensor $D$ such that $D(X,X)$ is a multiple of $X$, the function $\phi$ satisfying \eqref{eq spivak phi} is uniquely determined and clearly satisfies $\phi(\lambda X)=\lambda\phi(X)$. For the additivity, from the symmetry of $D$, we have
  \begin{equation} \label{eq spivak D}
 D(X+Y,X+Y)=D(X,X)+D(Y,Y)+2D(X,Y)\,,
 \end{equation} 
 hence 
 $$\phi(X+Y)(X+Y)=\phi(X)X+\phi(Y)Y+D(X,Y)\,.$$
 On the other hand, developing the same expression for $D(X-Y,X-Y)$ gives
 $$\phi(X-Y)(X-Y)=\phi(X)X+\phi(Y)Y-D(X,Y)\,.$$
 Putting together the two expressions, one obtains
 $$(\phi(X+Y)+\phi(X-Y)-2\phi(X))X=(\phi(X-Y)-\phi(X+Y)+2\phi(Y))Y\,.$$
 Since it now suffices to consider $X$ and $Y$ linearly independent, one has $\phi(X+Y)+\phi(X-Y)-2\phi(X)=\phi(X-Y)-\phi(X+Y)+2\phi(Y)=0$ and therefore $\phi(X+Y)=\phi(X)+\phi(Y)$. Hence, using the linearity of $\phi$ which was just proved, from \eqref{eq spivak phi} and \eqref{eq spivak D} one obtains:
 $$D(X,Y) = \phi(X)Y + \phi(Y)X~.$$
Finally, contracting on both side using \eqref{contraction} leads to
$\d\ln \lambda = (n+1){\phi}$.
\end{proof}

\paragraph{Killing fields.}

If $\nabla$ is the Levi-Civita connection of a pseudo-Riemannian metric $g$ on a manifold $N$, a vector field is called a \emph{Killing field}\index{Killing field} if it generates a $1$-parameter group of isometries.  A vector field $K$ is a Killing field if and only if $\textsc{L}_Xg=0$ ($\textsc{L}$ is the Lie derivative), or if and only if $\nabla K$ is a  skew-symmetric $(1,1)$-tensor:
\begin{equation}\label{def:killing}g(\nabla_XK,Y)+g(X,\nabla_YK)=0~, \end{equation}
for every $X,Y$. It is easy to see that this condition is actually equivalent to the condition that $g(\nabla_XK,X)=0$ for every $X$.

\paragraph{The infinitesimal Pogorelov map.}

Let us suppose that $\widehat \nabla$ is also a Levi-Civita connection for a pseudo-Riemannian metric $\widehat g$ on $N$.
Let $L:=L_{g,\widehat g}$ be the map $TN\to TN$ defined by

\begin{equation} \label{defi operator L}
g(X,Y)=\widehat g (L(X),Y)
\end{equation}
for every $Y$.
The \emph{infinitesimal Pogorelov map}\index{Pogorelov map! infinitesimal} $\mathrm{P}:=\mathrm{P}_{g,\widehat g}:TN\mapsto TN$ is defined by
$$ \mathrm{P}(X)= \lambda^{\frac{2}{n+1}}L(X)~. $$

\begin{lemma}\label{lem:killing}
The infinitesimal Pogorelov map $\mathrm{P}_{g,\widehat g}$ sends Killing fields of $g$ to Killing fields of $\widehat g$: if $K$ is a Killing field of $(N,g)$, then $\mathrm{P}(K)$ is a Killing field of $(N,\widehat g)$.
\end{lemma}
\begin{proof}
By definition,
\begin{equation} \label{eq defi killing}
\widehat g(\mathrm{P}(K), X) = g (\lambda^{\frac{2}{n+1}}K,X)~,
\end{equation} which implies
$$X.\widehat g(\mathrm{P}(K),X)=X.g(\lambda^{
\frac{2}{n+1}}K,X)~, $$
and using the fact that $K$ is a Killing field of $(N,g)$, we arrive at
\begin{equation} \label{eq proof lemma Killing}
\widehat g(\widehat\nabla_X\mathrm{P}(K),X)= (X.\lambda^{\frac{2}{n+1}})g(K,X)-{ g (\lambda^{\frac{2}{n+1}}K,\widehat \nabla_X X-\nabla_XX)~.}
\end{equation}
{Observe that by  Weyl formula (Lemma~\ref{weyl})}, 
$${g (\lambda^{\frac{2}{n+1}}K,\widehat \nabla_X X-\nabla_XX)=g(\lambda^{\frac{2}{n+1}}K,2X.(\ln \lambda^{\frac{1}{n+1}})X))
=(X.\lambda^\frac{2}{n+1})g(K,X)~.}$$ 
Hence from Equation \eqref{eq proof lemma Killing} one gets
$\widehat g(\widehat \nabla_X\mathrm{P}(K),X)= 0,$
i.e.~$\mathrm{P}(K)$ is a Killing field of $(M,\widehat g)$.
\end{proof}

%\begin{example}\label{ex: killing hyp}{\rm
\paragraph{The infinitesimal Pogorelov map from  hyperbolic  to Euclidean space.}
Let us consider the usual Klein model of the hyperbolic space $\mathbb{H}^n$, i.e. the affine chart $\{x_{n+1}=1\}$, in which
$\mathbb{H}^n$ is the open unit ball
$$B^n=\{x\in \R^n \,|\, b_{n,0}<1 \}~.$$

We will compare the hyperbolic metric $g$ on $B^n$ with the standard Euclidean metric. For $x\in B^n$, the point $r(x):=\rho(x)\binom{x}{1}$, with
$$\rho(x)=\frac{1}{\sqrt{1-b_{n,0}(x,x)}} $$
belongs to $\mathcal{H}^n$ in $\M^{n+1}$.\footnote{If $t$ is the hyperbolic distance from  $r(0)$ to $r(x)$, then $\rho(x)=\cosh t~. $} So if $X\in T_xB^n$, then 
$$\d r(X)=\rho(x)\binom{\rho(x)^2b_{n,0}(x,X)x+X}{\rho(x)^2b_{n,0}(x,X)} $$
belongs to  $T_{r(x)}\mathcal{H}^n$, and it is then straightforward that the expression of the hyperbolic metric in the Klein model is 
\begin{equation}\label{klein metric}g_x(X,Y)=\rho(x)^2b_{n,0}(X,Y)+\rho(x)^4b_{n,0}(x,X)b_{n,0}(x,Y)~. \end{equation}

The radial direction at $x$ is the direction defined by the origin of $\mathbb{R}^n$ and $x$. From \eqref{klein metric}, a vector is orthogonal to the radial direction for the Euclidean metric if and only if it is orthogonal for the hyperbolic metric. A vector orthogonal to the radial direction is called lateral. Then for a tangent vector $X$ of $B^n$ at a point $x$,
\begin{itemize}
\item if $X$ is radial, then  its hyperbolic norm is $\rho(x)^2$ times its Euclidean norm.
\item if $X$ is lateral, then its hyperbolic norm is $\rho(x)$ times its Euclidean norm.
\end{itemize}

 With respect to the definition of $L=L_{g,b_{n,0}}$ in {\eqref{defi operator L}},
$$L_x(X)=\rho(x)^2X+\rho(x)^4b_{n,0}(x,X)x $$
and  a lateral (resp. radial) vector is an eigenvector for $L$ with eigenvalue  $\rho(x)^2$
(resp. $\rho(x)^4$). In an orthonormal basis for
$b_{n,0}$,  $\operatorname{det}(g_x)=\det (L_x)=\rho(x)^{2(n+1)}$, so
$${\omega_{\R^n}=\rho^{-(n+1)} {\omega_{\H^n}} ~.} $$

%
%, and
%$\lambda_x$, the density of the Euclidean volume with respect to the hyperbolic volume at the point $x$, i.e. $\operatorname{det}(g_x)^{-1/2}$, is equal to $\rho(x)^{-(n+1)}$.

Let $K$ be a Killing field of $\H^n$. Then by Lemma~\ref{lem:killing},
$\mathrm{P}(K) = \rho^{-2}L_\cdot(K)$ is a Euclidean Killing field.  More precisely,
$$\mathrm{P}(K)_x=  K_x + \rho(x)^2b_{n,0}(x,K_x)x~.$$

In particular,
\begin{itemize}
\item if $K$ is lateral, the Euclidean norm of $\mathrm{P}(K)$ is equal to $\rho^{-1}$ times the hyperbolic norm of $K$;
\item if $K$ is radial, the Euclidean norm of $\mathrm{P}(K)$ is equal to the hyperbolic norm of $K$.
\end{itemize}

%}\end{example}

\paragraph{The infinitesimal Pogorelov map from Anti-de Sitter to Minkowski space.}

Let us consider the model of the Anti-de Sitter space $\AdS^n$ given by the affine chart $\{x_{n+1}=1\}$. Recall that $\AdS^n$ is the projective quotient of
$$\mathcal{A}d\mathcal{S}^n=\{(x,x_{n+1})\in\R^{n+1} \,|\, b_{n-1,2}=-1 \} $$
and its image in the affine chart is
$$H^n=\{x\in \R^n \,|\, b_{n-1,1}<1 \}~.$$

We will compare the Anti-de Sitter Lorentzian  metric $g$ on $H^n$ with the Minkowski metric on $\R^n$. For $x\in H^n$, the point $\rho(x)\binom{x}{1}$, with
$$\rho(x)=\frac{1}{\sqrt{1-b_{n-1,1}(x,x)}} $$
belongs to $\mathcal{A}d\mathcal{S}^n$.
A computation similar to the hyperbolic/Euclidean case gives
 \begin{equation*}
g_x(X,Y)=\rho(x)^2b_{n-1,1}(X,Y)+\rho(x)^4b_{n-1,1}(x,X)b_{n-1,1}(x,Y)~. \end{equation*}

As for the Klein model of the hyperbolic space, a vector has a radial and a lateral component, and this decomposition does not depend on the metric.
We also have
$$L_x(X)=\rho(x)^2X+\rho(x)^4b_{n-1,1}(x,X)x~, $$
whose determinant in an orthonormal basis for
$b_{n-1,1}$, is $\rho(x)^{2(n+1)}$, so (recall \eqref{eq:meme forme volume})
$${ \omega_{\R^n}=\rho^{-(n+1)}\omega_{\AdS^n}}~.$$

Let $K$ be a Killing field of $\AdS^n$. Then by Lemma~\ref{lem:killing},
$\rho^{-2}L_\cdot(K)= K + \rho^2b_{n-1,1}(\cdot,K)\cdot$ is a Minkowski Killing field.

\paragraph{Comments and references} \small
\begin{itemize}
\item A symmetric connection such that there exists (locally) a parallel volume form is characterized by the fact that its Ricci tensor is symmetric, \cite[Proposition 3.1]{nomizu}.
\item Let  $\nabla$ be a torsion-free complete connection on a manifold $N$, such that a local parallel volume form exists.
An \emph{affine field}\index{affine field} of $\nabla$  is a vector field that generates a $1$-parameter group of transformations that preserves the connection.
Let $R$ be the curvature tensor of $\nabla$:
$$R(X,Y)Z=\nabla_X\nabla_YZ-\nabla_Y\nabla_XZ-
\nabla_{[X,Y]}Z~. $$
Then $K$ is an affine field if and only if (see e.g. \cite{kobayashi})
$$\nabla_Y \nabla_X K= - R(K,Y)X~.$$
Let $\widehat \nabla$ be another connection on $N$ with the same properties as $\nabla$, both having the same unparametrized geodesics. Let $\widehat R$ be its curvature tensor. Then a direct computation using the Weyl formula shows that:
$$\widehat R(X,Y)Z=R(X,Y)Z + (\nabla^2 f -\d f^2)(X,Z)Y - (\nabla^2 f -\d f^2)(Y,Z)X~, $$
where $\nabla^2$ is the Hessian for $\nabla$.
See  pp. 126,127 in \cite{kobayashi} for the relations between infinitesimal affine transformations and  Killing fields.
\item Lemma~\ref{lem:killing} was proved in \cite{kne30} and independently in \cite{volkov}. See also the end of Section \ref{sec:inf rig}.% We don't know if there exists a result similar to  Lemma~\ref{lem:killing} for infinitesimal affine transformations.
\item The simplest infinitesimal Pogorelov map is the one from Minkowski space to Euclidean space: it suffices to multiply the last coordinate of the Killing field by $-1$. This was noted in \cite{GPS82}.
The term infinitesimal Pogorelov map comes from the fact that it was defined as a first-order version of  the so-called Pogorelov mapping, see the end of Section~\ref{sec:inf rig}.  See  \cite{LS00,sch06,DMS} for some applications.
\end{itemize}

\normalsize

\section{Geometry of surfaces in 3-dimensional spaces} \label{sec geometry surfaces}

The purpose of this section is the study of embedded surfaces, with particular attention to convex surfaces, in 3-dimensional model spaces. We will first review the classical theory of embeddings of surfaces in 3-manifolds, with special attention to the constant curvature cases we have introduced so far. After that, we will define analogous notions in the degenerate spaces (for instance, the second fundamental form and the shape operator), in particular co-Euclidean and co-Minkowski spaces, and show that these notions have a good behavior both with respect to the geometry of $^*\E^3$ and $^*\M^3$, although the ambient metric is a degenerate metric, and with respect to duality and geometric transition.

\subsection{Surfaces in non degenerate constant curvature 3-manifolds}
Let $(M,g)$ be a Riemannian or pseudo-Riemannian manifold. Given a smooth immersion $\sigma:S\to M$ with image a space-like surface $\sigma(S)$ in $(M,g)$, recall that the \emph{first fundamental form}\index{fundamental form! first} is the pull-back of the induced metric, namely
$$\I(v,w)=g(\sigma_*(v),\sigma_*(w))\,.$$
The Levi-Civita connection $\nabla^S$ of the first fundamental form $\I$ of $S$ is obtained from the Levi-Civita connection of the ambient manifold: given vector fields $v,w$ on $S$,
$\nabla^S_v w$ is the orthogonal projection to the tangent space of $S$ of $\nabla^{M}_{\sigma_* v} (\sigma_* w)$.

Let us denote by $\textsc{N}$  a unit normal vector field on $S$, namely $\textsc{N}$ is a smooth vector field such that for every point $x\in S$, $\textsc{N}_x$ is orthogonal to $T_{\sigma(x)}\sigma(S)$, and such that $|g(\textsc{N},\textsc{N})|=1$. The \emph{second fundamental form}\index{fundamental form! second} $\II$ is a bilinear form on $S$ defined by
\begin{equation}\label{eq:nabla sub}\nabla^{M}_{\sigma_* v}(\sigma_*\hat w)=\nabla^S_{v}\hat w+\II(v,w)\textsc{N}\end{equation}
where $\hat w$ denotes any smooth vector field on $S$ extending the vector $w$. This form turns out to be symmetric and it only depends on the vectors $v$ and $w$, not on the extension of any of them.
The \emph{shape operator} of $S$ is the $(1,1)$-tensor $B\in\mathrm{End}(TS)$, self-adjoint with respect to $\I$, such that
\begin{equation} \label{second fundamental form}\II(v,w)=\I(B(v),w)=\I(v,B(w))\,. \end{equation}
 By a standard computation, it turns out that
 \begin{equation}\label{eq:shape op}B(v)=\pm\nabla^M_v \textsc{N}\,,\end{equation}
the sign depending on whether $M$ is Riemannian or Lorentzian, where we have used the differential of the embedding $\sigma$ to identify a vector $v\in T_x S$ to a vector tangent to the embedded surface $\sigma(S)$.
Indeed, by applying the condition of compatibility of the metric of the Levi-Civita connection to the condition $|g(\textsc{N},\textsc{N})|=1$, it is easily checked that $\nabla^{M}_{v}\textsc{N}$ is orthogonal to $\textsc{N}$, hence is tangent to $\sigma(S)$.
Since $B$ is self-adjoint with respect to $\I$, it is diagonalizable at every point. The eigenvalues of $B$ are called \emph{principal curvatures}.

The \emph{extrinsic curvature} is the determinant of the shape operator, i.e. the product of the principal curvatures. The \emph{mean curvature} is the trace of the shape operator, that is the sum of the principal curvatures. 

\paragraph{Fundamental theorem of immersed surfaces.}

The \emph{embedding data} of a smooth immersed surface in a 3-manifold are usually considered the first fundamental form and the shape operator, or equivalently, the first and second fundamental forms. However, these two objects are not independent and satisfy some coupled differential equations usually called the \emph{Gauss--Codazzi} equations. We start by expressing such equations in the setting of the constant curvature 3-manifolds introduced above, namely Euclidean and Minkowski space, and the model spaces.

The \emph{Codazzi equation}\index{Codazzi equation} can be expressed in the same fashion for all ambient spaces, and it says that the exterior derivative of the shape operator $B$, with respect to the Levi-Civita connection of the first fundamental form (denoted by $\nabla^\I$), vanishes. In formulae,
\begin{equation}\label{eq:codazzi eq}
\d^{\nabla^\I}\!B:=(\nabla^\I_{v} B)(w)-(\nabla^\I_{w} B)(v)=\nabla^\I_{v} B(w)-\nabla^\I_{w} B(v)-B[v,w]=0\,,
\end{equation}

The \emph{Gauss equation}\index{Gauss equation} is a relation between the intrinsic curvature of the first fundamental form, and the extrinsic curvature of the immersion. The form of the Gauss equation, however, depends on the ambient metric (whether it is Riemannian or pseudo-Riemannian) and on its curvature. In particular, if $K_\I$ denotes the curvature of the first fundamental form, in Euclidean space the Gauss formula is:
$$K_\I=\det B\,,$$
while in Minkowski space, for immersed space-like surfaces, the correct form is:
$$K_\I=-\det B\,.$$
As we said, if the ambient manifold has nonzero curvature, there is an additional term in the equation. For instance, for $\mathcal S^3$ or $\Ell^3$ (curvature $1$ everywhere):
$$K_\I=1+\det B\,,$$
while for surfaces in hyperbolic space we have:
$$K_\I=-1+\det B\,.$$
Finally for space-like surfaces in Lorentzian manifolds of nonzero constant curvature: in de Sitter space the Gauss equation is
$$K_\I=1-\det B\,,$$
and in Anti-de Sitter space of course
$$K_\I=-1-\det B\,.$$
We resume all the statements in the following:
\begin{fact}
Given a smooth immersed surface in a three-dimensional model space (or in Euclidean/Minkowski space), the first fundamental form and the shape operator satisfy the corresponding Gauss--Codazzi equations.
\end{fact}

Clearly, if one post-composes a smooth immersion with an isometry of the ambient space, the embedding data remain unchanged. A classical theorem in Euclidean space, which can be extended to all the other cases of constant curvature, states that the Gauss--Codazzi equations are necessary and sufficient conditions in order to have the embedding data of a smooth surface, and the embedding data determine the surface up to global isometry. See \cite{BGM,spi3} for a reference.

\begin{theorem}[Fundamental theorem of immersed surfaces]\index{fundamental theorem of immersed surfaces}
Given a simply connected surface $S$ and a pair $(\I,B)$ of a Riemannian metric and a symmetric $(1,1)$-tensor on $S$, if $(\I,B)$ satisfy the Gauss--Codazzi equations in Euclidean (resp. Minkowski, elliptic, hyperbolic, de Sitter or Anti-de Sitter) space, then there exists a smooth immersion of $S$ having $(\I,B)$ as embedding data. Any two such immersions differ by post-composition by a global isometry.
\end{theorem}

\paragraph{Duality for smooth surfaces.}
We can use the description of duality introduced in Subsection \ref{duality} to talk about duality for convex surfaces. In fact, consider a surface $S$, which is the boundary of a convex set. Then its dual is again a convex set with boundary a surface $S^*$. For instance, the dual of the boundary of a convex set in $\H^3$  is a space-like surface in $\dS^3$ (and vice versa). Similarly one can consider a convex set in $\AdS^3$, whose boundary is composed of two space-like surfaces, and dually one obtains a convex set  in $\AdS^3$ bounded by two 
space-like surfaces.

 Let us consider a convex set with boundary a smooth (or at least $C^2$) embedded surface $S$, such that $B$ is positive definite at every point (which in particular implies strict convexity). 
The \emph{third fundamental form}\index{fundamental form! third} of $S$ is
$$\III(v,w)=\I(B(v),B(w))\,.$$
First, observe that, since  $B$ is positive definite, $\III$ is a Riemannian metric. The reader can check, as an exercise, that in his/her favorite duality of ambient spaces, $(\III,B^{-1})$ are the embedding data of a space-like surface, namely they satisfy the Gauss--Codazzi equations provided $(\I,B)$ satisfy the Gauss--Codazzi equations. It is helpful to use the following formula for the curvature of $\III(v,w)=\I(B\cdot,B\cdot)$:
$$K_{\III}=\frac{K_\I}{\det B}\,,$$
which holds under the assumption that $B$ satisfies the Codazzi equation (see
\cite{labourieCP} or \cite{Schlenker-Krasnov}).
\begin{fact}
The pair $(\III,B^{-1})$ are the embedding data of the dual surface $S^*$.
\end{fact}
Let us check the fact in the hyperbolic-de Sitter case. If $S$ is a convex surface in $\H^3$, for $x\in S$, $\textsc{N}(x)$ is a point inside the unit tangent sphere of $T_x\H^3$, and $\III$ is by definition the pull-back by $\textsc{N}$ on $S$ of the spherical metric. But on the other hand,
the hyperplane tangent to $S$ at $x$ is naturally identified, via the double cover in $\M^{4}$, to a time-like hyperplane, and $\textsc{N}(x)$ is a unit vector orthogonal to this hyperplane, hence a point in $d\mathcal{S}^n$, the double cover of de Sitter space. So $\III$ is exactly the induced metric on the dual surface, and by the involution property of the duality, its shape operator is the inverse of the one of $B$.

\subsection{Geometry of surfaces in co-Euclidean space}

We are now ready to define the second fundamental form of any space-like surface in $^*\E^3$. For simplicity (since this is a local notion) we will use the double cover $\cE^3\cong \mathcal S^2\times \R$. A space-like surface in $\cE^3$ is locally the graph of a function $ u:\Omega\to\R$, for $\Omega\subseteq\mathcal S^2$, and the first fundamental form is just the spherical metric on the base $\mathcal S^2$. 

Given a space-like immersion $\sigma:S\to\cE^3$, again there is no notion of normal vector field, since the metric of $\cE^3$ is degenerate. However, we have already remarked several times that the vector field $\textsc{T}$ (of zero length for the degenerate metric, see Section~\ref{sec deg case}) is well-defined, i.e. it is invariant by the action of the isometry group. Given two vector fields $\hat v,\hat w$ on $S$, using symmetry and compatibility with the metric, it is easy to prove that the tangential component of $\nabla^{\cE}_{\sigma_*v}(\sigma_*\hat w)$ in the splitting
$$T_{(x,t)}\mathcal \cE^3=T_{(x,t)}\sigma(S)\oplus\langle \textsc{T}\rangle\,,$$ coincides with the Levi-Civita connection of the first fundamental form. This enables to give the following definition:

\begin{df}
Given a space-like immersion $\sigma:S\to\cE^3$, the second fundamental form of $S$ is defined by
$$\nabla^{\cE^3}_{\sigma_*v}(\sigma_*\hat w)=\nabla^{\I}_{v}(\hat w)+\II(v,w) \textsc{T}\,,$$
for every pair of vectors $v,w\in T_x S$, where $\hat w$ is any extension on $S$ of the vector $w\in T_x S$.
\end{df}

It is easy to check, as in the classical Riemannian case, that $\II$ is linear in both arguments, and thus defines a $(0,2)$-tensor.

\begin{df}
The shape operator of $\sigma:S\to\cE^3$ is the  symmetric $(1,1)$-tensor $B$ such that $\II(v,w)=\I(B(v),w)$ for every $v,w\in T_x S$, where $\I$ is the first fundamental form of $S$. The extrinsic curvature of $S$ is the determinant of the shape operator.
\end{df}

As usual the definition does not depend on the extension of any of the vectors $v$ and $w$ and the second fundamental form is symmetric.

\begin{lemma} \label{lemma shape operator sph hessu-uid}
Given a space-like embedded graph $S$ in $\cE^3$, consider the embedding $\sigma:\Omega\to\cE^3\cong \mathcal S^2\times\R$ defining $S$ as a graph:
$$\sigma(x)=(x, u(x))$$
for $u:\Omega\to\R$ and $\Omega\subseteq\mathcal S^2$. Then the shape operator of $S$ for the embedding $\sigma$ is 
\begin{equation*} %\label{second fundamental form coE}
B={\Hess^{\mathcal{S}^2}}u+u\, \mathrm{Id}\,,
\end{equation*} 
where ${\Hess^{\mathcal{S}^2}}u=\nabla^{\mathcal S^2}\!\mathrm{grad} u$ denotes the {spherical} Hessian of $u$.
\end{lemma}

{Before proceeding to the proof of this proposition, let us relate functions on $\Omega\subset \mathcal{S}^2\subset \E^3$ to functions on an open set of $\E^3$. As the unit outward vector field of $\mathcal{S}^2$ is the identity, it follows from
 \eqref{eq:nabla sub} and \eqref{eq:shape op} that, at a point $x\in \mathcal{S}^2$, for $X,Y$ tangent to the sphere,
\begin{equation}\label{eq: sphere connection}D_{X} Y=\nabla^{\mathcal S^2}_{X}Y-\langle X,Y\rangle x\,\end{equation}
where $\langle\cdot,\cdot\rangle:=b_{3,0}$. From this it follows immediately that, for $U$ the 1-homogenous extension of $u$,
 $$\langle \Hess U (X),Y\rangle= X.Y.U - (D_XY).U $$$$= \Hess^{\mathcal{S}^2} U|_{\mathcal{S}^2}(X,Y)+\langle X,Y\rangle x.U \,,$$
and as $U$ is 1-homogenous, by the Euler theorem,
 $x.U=U$, so on $\S^2$:
 \begin{equation}\label{restr hessien2}\Hess U = \Hess^{\mathcal{S}^2}\! u + u \operatorname{Id}~. \end{equation}
 } 

%  ****************************
%Let us consider spherical coordinates on $\E^3$:
%any
%orthonormal frame  on
%$\mathcal{S}^2$ is extended to  an orthonormal frame of $\E^3$ with the decomposition
%$r^2g_{\mathcal{S}^2}+\mathrm{d}r\otimes \mathrm{d}r$ of the Euclidean metric.
%The Hessian of a function $U$ on $\E^3$ and spherical Hessian of its restriction to $\mathcal{S}^2$ are related by
%\begin{equation}\label{restr hessien}
%D^2U=\Hess\, U+\frac{\partial U}{\partial r}g_{\mathcal{S}^2}~.
%\end{equation}
%If $u$ is a function on $\mathcal{S}^2$, let us consider $U$, its one homogenous extension to $\E^3$: 
%$U(x)=\|x\|u(\pi(x))$,
%By the Euler homogenous function theorem and \eqref{restr hessien} give
%\begin{equation}\label{restr hessien2}D^2U=\Hess\, u+ug_{\mathcal{S}^2}~. \end{equation}

\begin{proof}
Fix a point $x_0\in\Omega$, and denote by $b^*=x_1y_1+x_2y_2+x_3y_3$ the ambient degenerate quadratic form. By composing with an element in $\mathrm{Isom}(\mathcal \cE^3)$ of the form $(x,t)\mapsto(x,t+f(x))$, where $f(x)=\langle x,p_0\rangle$, we can assume $S$ is tangent to the horizontal plane $\mathcal S^2\times\{0\}$ at $x_0$. Indeed, observe that by the usual duality, this statement is equivalent to saying that the point in $\E^3$ dual to $T_{\sigma(x_0)}S$ is the origin and the plane in $\E^3$ dual to $x_0$ is $x_0^\perp$. Such s condition can be easily obtained by applying a translation of $\E^3$. {From \eqref{restr hessien2}, } ${\Hess^{\mathcal{S}^2}} f+f \mathrm{Id}=0$, and thus it suffices to prove the statement in this case. 

We consider $\mathcal S^2\times\{0\}$ inside the copy of $\E^3$ obtained as $\{x_4=0\}$ in $\R^4$. Let $\hat v$, $\hat w$ be two vector fields on $\Omega\subseteq \mathcal S^2$. Then at $x_0\in\mathcal S^2$ one has
as in \eqref{eq: sphere connection} 
 $$D_{\hat v}\hat w=\nabla^{\mathcal S^2}_{\hat v}\hat w-\langle \hat v,\hat w\rangle x_0\,.$$
Consider now the vector fields $\sigma_*(\hat v)=\hat v+\d u(\hat v)\textsc{T}$ and $\sigma_*(\hat w)=\hat w+\d u(\hat w)\textsc{T}$ on $S$. We choose extensions $V$ and $W$ in a neighbourhood of $S$ which are invariant by translations $t\mapsto t+t_0$ in the degenerate direction of $\mathcal S^2\times\R$. We can now compute
\begin{align*}
D_V W&=D_{\hat v}\hat w+D_{\hat v}(\d u(\hat w)\textsc{T})+\d u(\hat v)D_{\textsc{T}}W \\
&=\nabla^{\mathcal S^2}_{\hat v}\hat w-\langle \hat v,\hat w\rangle x_0+{\hat v}.\langle  \mathrm{grad} u,\hat w\rangle \textsc{T} \\
&=\nabla^{\mathcal S^2}_{\hat v}\hat w-\langle \hat v,\hat w\rangle x_0+\left(\langle{\Hess^{\mathcal{S}^2}} u(\hat v),\hat w\rangle +\langle \mathrm{grad} u,\nabla^{\mathcal S^2}_{\hat v}\hat w\rangle \right)\textsc{T}\,,
\end{align*}
where in the first equality we have substituted the expressions for $V$ and $W$, and in the second equality we have used that $D\textsc{T}=0$ and that the chosen extension $W$ is invariant along the direction of $\textsc{T}$. Recalling that the connection $\nabla^{\cE^3}_V W$ at $x_0$ is defined as the tangential component of $D_V W$ with respect to the transverse vector $\textsc{N}_{x_0}=x_0$ and using that $\mathrm{grad} u$ vanishes at $x_0$ by construction, we get at $x_0$: 
$$\nabla^{\cE^3}_V W=\nabla^{\mathcal S^2}_{\hat v}\hat w+\I({\Hess^{\mathcal{S}^2}}u(v),w)\,.$$
By Proposition \ref{prop properties coeucl connection}, $\nabla^{\mathcal S^2}_{\hat v}\hat w$ is tangent to the slice ${\mathcal S^2}\times\{0\}$ itself, and thus the second fundamental form is:
$$\II(v,w)=\I({\Hess^{\mathcal{S}^2}} u(v),w)\,.$$
Since $u(x_0)=0$, this concludes the proof.
\end{proof}

\begin{cor}
A space-like surface in $^*\E^3$ is a plane if and only if $B\equiv 0$.
\end{cor}
\begin{proof}
Totally geodesic planes in $\cE^3$ are graphs of functions of the form $u(x)=\langle x,p\rangle$, which are functions such that ${\Hess^{\mathcal{S}^2}}\, u+ u\,\mathrm{Id}$ vanishes. Since the statement is local, this is true also in $^*\E^3$.
\end{proof}

Recall that the induced metric on a space-like surface $S$ in $\cE^3$, written as a graph over $\Omega\subset\mathcal S^2$, is just the spherical metric of $\mathcal S^2$. Hence the first fundamental form $\I$ of $S$ satisfies:
\begin{equation} \label{baby gauss}
K_\I=1\,.
\end{equation}
Of course this statement is local, and thus holds in $^*\E^3$ as well. Moreover, it turns out that the shape operator satisfies the Codazzi equation with respect to $\I$:
\begin{equation} \label{baby codazzi}
\d^{\nabla^\I}\!B=0\,.
\end{equation}

Indeed, let $U$ be the $1$ homogenous extention of $u$. Let $\tilde B=\Hess U$. Each component 
$\tilde B^i$, $i=1,2,3$, is a one-form on $\R^3$, $$\tilde B^i=\tilde B^i_1 \d x^1+\tilde B^i_2\d x^2 + \tilde B^i_3 \d x^3~,$$
with $U_{ij}=\tilde B^i_j$. By the Schwarz lemma, $\tilde B^i$ is closed, and by \eqref{eq:codazzi eq} and  \eqref{eq: sphere connection}, $B$ satisfies the Codazzi equation on  $\mathcal{S}^2$.
See also \cite[2.61]{GHL} for the relations between closdeness of one forms and the Codazzi equation.

The above two equations \eqref{baby gauss} and \eqref{baby codazzi} can be interpreted as a \emph{baby-version} of the Gauss--Codazzi equations for co-Euclidean geometry. Of course this is a very simple version, since the two equations are not really coupled: the first equation is independent of $B$. In fact, the proof of the fundamental theorem of immersions will be very simple. But before that, let us show that any  tensor satisfying the Codazzi equation on $\mathcal{S}^2$ can be locally written as $\Hess^{\mathcal S^2}\! u+u\,\mathrm{Id}$.

Let $\Omega$ be a domain in $\mathcal S^2$ and let $\tilde \Omega$ be the cone over $\Omega$ from the origin of $\E^3$. Note that $\tilde\Omega$ is star-shaped from the origin.
Let $B$ be a $(1,1)$-tensor on $\Omega$.
 Let $\tilde B$ be 	a $(1,1)$-tensor of $\tilde \Omega$ which satisfies $\tilde B_x(x)=0$, and for $\langle X,x\rangle=0$, $\tilde B_x(X)=B_{x/\|x\|}(X)$, where  $\|\cdot\|$ is the usual norm on $\E^3$.
  
Suppose  that $B$ satisfies the Codazzi equation on $\Omega$, and consider $\tilde B$ as an $\R^3$-valued one form.
From \eqref{eq:codazzi eq} and \eqref{eq: sphere connection}, $\tilde B$ satisfies the Codazzi equation on $\tilde \Omega$, hence is a  closed form, and  by the Poincar\'e Lemma, 
there exists $F^i : \R^3\to \R$ such that 
$\tilde B^i=\d F^i$.  
Suppose moreover that $B$ is symmetric. This readily implies that $\tilde B$ is symmetric, and hence that the one form 
$\omega=F^1\d x^1+F^2\d x_2^2+
F^3\d x_3^2$ is closed, so by the Poincar\'e Lemma, there exists a function $U:\tilde \Omega \to \R$ with 
$\omega = \d U$, i.e.  $\tilde B=\Hess U$. 
From  \eqref{restr hessien2}, we obtain the following fact, previously noted by D. Ferus in \cite{ferus}:
\begin{fact}\label{fact ferus}
Let $B$ be a  tensor satisfying the Codazzi equation on a surface of constant curvature equal to one. Then locally there exists a function $u$ on $\mathcal{S}^2$ such that $B=\Hess^{\mathcal S^2} u + u\, \mathrm{Id}$.
\end{fact}

\begin{proposition}[Fundamental theorem of immersed surfaces in co-Euclidean geometry] \label{fund thm coE}

Given a simply connected surface $S$ and a pair $(\I,B)$ of a Riemannian metric and a symmetric  $(1,1)$-tensor on $S$, if $(\I,B)$ are such that $\I$ has constant curvature $+1$,
$$K_\I=1\,,$$
and $B$ is satisfies the Codazzi conditions for $\I$,
$$\d^{\nabla^\I}\!B=0\,.$$
Then there exists a smooth immersion of $S$ in $^*\E^3$ having $(\I,B)$ as embedding data. Any two such immersions differ by post-composition by a global isometry of $^*\E^3$.
\end{proposition}

\begin{proof}
Again, it suffices to prove first the immersion statement in $\cE^3$. By Fact~\ref{fact ferus} there exists a function $u:S\to\R$ such that $B={\Hess^{\mathcal{S}^2}}\, u+ u\, \mathrm{Id}$. Let $\mathrm{dev}:\tilde S\to\mathcal S^2$ be a developing map (that is, a local isometry) for the spherical metric $\I$ on $S$. Then we define
$$\sigma:S\to\cE^3\cong \mathcal S^2\times\R$$
by means of
$$\sigma(x)=(\mathrm{dev}(x), u(x))\,.$$
Since $\mathrm{dev}$ is a local isometry, and the degenerate metric of $\cE^3$ reduces to the spherical metric on the $\mathcal S^2$ component, the first fundamental form of $\sigma$ is $\I$. By Lemma \ref{lemma shape operator sph hessu-uid}, the shape operator is $B$.

Given any other immersion $\sigma'$ with embedding data $(\I,B)$, the projection to the first component is a local isometry, hence it differs from $\mathrm{dev}$ by postcomposition by a global isometry $A$ of $\mathcal S^2$. By composing with an isometry of $\cE^3$ which acts on $\mathcal S^2$ by means of $A$ and leaves the coordinate $t$ invariant, we can assume the first component of $\sigma$ and $\sigma'$ in $\mathcal S^2\times\R$ coincide.

Hence we have $\sigma'(x)=(\mathrm{dev}(x), u'(x))$ where $u'$ is such that $B={\Hess^{\mathcal{S}^2}}\,u'+u'\, \mathrm{Id}$. Therefore ${\Hess^{\mathcal{S}^2}}(u-u')+(u-u')\mathrm{Id}=0$, which implies  $$u(x)-u'(x)=\langle \mathrm{dev}(x),p_0\rangle$$ for some $p_0\in\E^3$. This shows that $\sigma$ and $\sigma'$ differ by the isometry
$$(x,t)\mapsto (x,t+\langle x,p_0\rangle)$$
which is an element of $\mathrm{Isom}(\cE^3)$. This concludes the proof.
\end{proof}

Let us remark that (recalling the discussion of Subsection \ref{coeucl}) if $S$ is a space-like surface which is homotopic to $\mathcal S^2\times\{\star\}$ in $\cE^3\cong\mathcal S^2\times \R$ and its shape operator $B$ is positive definite, then $S$ is the graph of $u$ where $u$ is a support function of a convex body whose boundary is $S$. 
Let us now compute the embedding data in Euclidean space of the dual surface of $S$. The formulae we will obtain are exactly the same as in the $\AdS^3$-$\AdS^3$ duality and in the $\H^3$-$\dS^3$ duality.

\begin{cor}\label{cor dual plat}
Given a space-like smooth surface $S$ in $\E^3$ (resp. $^*\E^3$) which bounds a convex set and with positive definite shape operator, if the embedding data of $S$ are $(\I,B)$, then the boundary of the dual convex set in $^*\E^3$ (resp. $\E^3$) has embedding data $(\III,B^{-1})$.
\end{cor}
\begin{proof}
 If the result is proved for one case, then the other case is obvious, since it suffices to exchange the roles of $(\I,B)$ and $(\III,B^{-1})$. Let us consider a surface $S$ in $
 \E^3$, and let $S^*$ be the dual surface in $^*\E^3$. 
 
Let us assume that in the double cover $\cE^3$, $S^*$ is locally the graph of $u$ over an open subset $\Omega$ of $\mathcal S^2$. To check that the first fundamental form of the dual surface is the third fundamental form, it suffices to compute the pull-back of the spherical metric of $\mathcal S^2$ by means of the first component of the dual embedding $\sigma:S\to {^*\E^3}$ for $S^*$ --- that is, the map which associates to a tangent plane $P=T_x S$ to $S$ the dual point $P^*\in$ $^*\E^3$. Again in the double cover, the first component of the dual embedding is precisely the Gauss map $G$ of $S$, and the derivative of the Gauss map is the shape operator, hence one gets
$$\langle \d\sigma(v),\d\sigma(w)\rangle=\langle \d G(v),\d G(w)\rangle=\I(B(v),B(w))\,.$$
Moreover, one can directly show that the inverse of the shape operator of $S$, by means of the inverse of the Gauss map $G^{-1}:\mathcal S^2\to S\subset\E^3$, is $B^{-1}={\Hess^{\mathcal{S}^2}}u+u\, \mathrm{Id}$, where $u:\mathcal S^2\to\R$ is the support function of the convex set bounded by $S$. Using Lemma \ref{lemma shape operator sph hessu-uid}, this concludes the proof.
\end{proof}

\subsection{Geometry of surfaces in co-Minkowski space}

By means of a very analogous construction, one can study the geometry of surfaces in co-Minkowski space. We will not repeat all the details, but just state the parallel results. Here $\langle\cdot,\cdot\rangle:=b_{2,1}$ will denote the Minkowski product of $\M^3$, and we will implicitly identify $^*\M^3$, which is the quotient of $\cM^3\cong\mathcal H^2\times \R$ by the action of $\{\pm\mathrm{Id}\}$, to one of its connected components, thus considering the model of $\H^2\times\R$. The definition of second fundamental form of an immersion $\sigma:S\to$~$^*\M^3$ makes again use of the splitting:
$$T_{x}\mathcal \cM^3=T_{x}\sigma(S)\oplus\langle \textsc{T}\rangle\,.$$

\begin{df}
Given a space-like immersion $\sigma:S\to{^*\M^3}$, the second fundamental form of $S$ is defined by
$$\nabla^{^*\M^3}_{\sigma_*v}(\sigma_*\hat w)=\nabla^{\I}_{v}(\hat w)+\II(v,w) \textsc{T}\,,$$
for every pair of vectors $v,w\in T_x S$, where $\hat w$ is any extension on $S$ of the vector $w\in T_x S$.
\end{df}

\begin{df}
The shape operator of $\sigma:S\to{^*\M^3}$ is the  symmetric  $(1,1)$-tensor such that $\II(v,w)=\I(B(v),w)$ for every $v,w\in T_x S$, where $\I$ is the first fundamental form of $S$. The extrinsic curvature of $S$ is the determinant of the shape operator.
\end{df}

\begin{lemma} \label{lemma shape operator hp hessu-uid}
Given a space-like embedded graph $S$ in $^*\M^3$, consider the embedding $\sigma:\Omega\to{^*\M^3}\cong \H^2\times\R$ defining $S$ as a graph:
$$\sigma(x)=(x,u(x))$$
for $ u:\Omega\to\R$ and $\Omega\subseteq\H^2$. Then the shape operator of $S$ for the embedding $\sigma$ is 
\begin{equation*} %\label{second fundamental form half-pipe}
B=\mathrm{Hess^{\H^2}} u- u\, \mathrm{Id}\,,
\end{equation*} 
where $\Hess^{\H^2} u=\nabla^{\H^2}\!\mathrm{grad} \, u$ denotes the hyperbolic Hessian of $ u$.
\end{lemma}

\begin{cor}
A space-like surface in $^*\M^3$ is   a plane if and only if $B\equiv 0$.
\end{cor}

Again, one can observe that the embedding data of a smooth space-like surface in $^*\M^3$ satisfy the very simple condition that the first fundamental form $\I$ is hyperbolic, and the shape operator satisfies the  Codazzi equation for $\I$.

\begin{proposition}[Fundamental theorem of immersed surfaces in co-Minkowski geometry] \label{fund thm halfpipe}

Given a simply connected surface $S$ and a pair $(\I,B)$ of a Riemannian metric and a  symmetric $(1,1)$-tensor on $S$, if $(\I,B)$ are such that $\I$ has constant curvature $-1$,
$$K_\I=-1\,,$$
and $B$ satisfies the Codazzi conditions for $\I$,
$$\d^{\nabla^\I}\!B=0\,.$$
then there exists a smooth immersion of $S$ in $^*\M^3$ having $(\I,B)$ as embedding data. Any two such immersions differ by post-composition by a global isometry of $^*\M^3$.
\end{proposition}

Finally, the formulae for the embedding data of dual surfaces work well in the $\M^3$-$^*\M^3$ duality:

\begin{cor}
Given a space-like smooth surface $S$ in $\M^3$ (resp. $^*\M^3$) which bounds a convex set and with positive definite shape operator, if the embedding data of $S$ are $(\I,B)$, then the boundary of the dual convex set in $^*\M^3$ (resp. $\M^3$) has embedding data $(\III,B^{-1})$.
\end{cor}

\subsection{Geometric transition of surfaces}

We now conclude by showing that the notion of curvature in co-Euclidean and co-Minkowski geometry, under the procedures of geometric transition we have encountered, is the rescaled limit of the usual notions in the model spaces which degenerate to $^*\E^3$ and $^*\M^3$. In particular, we will focus on the transitions defined in Subsection \ref{subsec Limits of 3-dimensional model spaces}, having limit in the co-Euclidean and co-Minkowski space.

\begin{proposition} \label{prop rescaling curvature}
Suppose $\sigma_t$ is a $C^2$ family of space-like smooth immersions of a simply connected surface $S$ into $\Ell^3$ or $\dS^3$ (resp. $\AdS^3$ or $\H^3$), such that $\sigma_0$ is contained in a totally geodesic plane $P$. Let $$\sigma=\lim_{t\to 0}(g_t^*\circ \sigma_t)$$
be the rescaled immersion in $^*\E^3$ (resp. $^*\M^3$) obtained by blowing-up the plane $P$.
Then:
\begin{itemize}
\item The first fundamental form of $\sigma$ coincides with the first fundamental form of $\sigma_0$:
$$\I(v,w)=\lim_{t\to 0}\I_t(v,w)\,;$$
\item The second fundamental form of $\sigma$ is the first derivative of the second fundamental form of $\sigma_t$:
$$\II(v,w)=\lim_{t\to 0} \frac{\II_t(v,w)}{t}\,;$$
\item The shape operator $B$ of $\sigma$ is the first derivative of the shape operator $B_t$ of $\sigma_t$:
$$B(v)=\lim_{t\to 0}\frac{B_t(v)}{t}\,;$$
\item The extrinsic curvature $K^{ext}$ of $\sigma$ is the second derivative of the extrinsic curvature $K^{ext}_t=\det B_t$ of $\sigma_t$:
$$K^{ext}(x)=\lim_{t\to 0}\frac{K^{ext}_t(x)}{t^2}\,.$$
\end{itemize}
\end{proposition}
\begin{proof}
As in Subsection \ref{subsec Limits of 3-dimensional model spaces} (recall also the proof of Proposition \ref{prop limit connection volume coeucl}), we can assume that the totally geodesic plane $P$ is defined by $x_4=0$ in $\R\mathrm{P}^3$, and use the projective transformations
$$g_t^*=\begin{bmatrix} 1 & 0 & 0 & 0 \\ 0 & 1 & 0 & 0 \\ 0 & 0 &  1 & 0 \\ 0 & 0 & 0 & 1/t  \end{bmatrix}\in \mathrm{PGL}(4,\R)\,.$$
Hence, the limit of $g_t^*\sigma_t$ is:
$$\lim_{t\to 0}g_t^* \sigma_t(x)=\lim_{t\to 0}\begin{bmatrix} (\sigma_t(x))_1 \\ (\sigma_t(x))_2 \\ (\sigma_t(x))_3 \\ (\sigma_t(x))_4/t  \end{bmatrix}=\begin{bmatrix} (\sigma_0(x))_1 \\ (\sigma_0(x))_2 \\ (\sigma_0(x))_3 \\ (\dot\sigma(x))_4/t  \end{bmatrix}\,,$$
namely, $\sigma(x)=(\sigma_0(x), u(x))$ for some function $u$ which encodes the derivative of the vertical component of $\sigma_t$. It is then clear, from the degenerate form of the metric of both $^*\E^3$ and $^*\M^3$, that in both cases the first fundamental form of the rescaled limit $\sigma$ is the limit of the first fundamental form of $\sigma_t$.

For the second point, we will focus on the case of $\Ell^3$ (with limit in $\E^3$) for definiteness. Since the statement is local, we can consider the computation in the dual covers $\mathcal S^3$ and $\cE^3$. %Consider the second fundamental form of $\sigma_t$. Given two vectors $v,w\in T_x  S$, consider extensions $\hat v$ and $\hat w$ on a neighborhood of $x$. We observe that, as $t\to 0$, the connection of $\H^3$ converges to the connection of $^*\M^3$. Indeed, in the double cover, $g_t^*$ maps $\mathcal H^3$ isometrically to 
%$$\mathcal H^3_t:=g_t^*(\mathcal H^3)=\{x:x_1^2+x_2^2-x_3^2+t^2x_4^2=-1\}\,.$$ 
%Recall that the connection of $g_t^*(\mathcal H^3)$ can be obtained by projecting the flat connection $D_{v}\hat w$ of the ambient space tangentially to $\mathfrak{r}_t(\H^3)$. As $t\to 0$, $g_t^*(\H^3)$ converges to the co-Minkowski space, hence the tangent projection converges exactly to that of $^*\M^3$. 
Consider the unit normal vector fields $\textsc{N}_t$ to $\sigma_t( S)$, chosen so that at time $t=0$ the vector field is $(0,0,0,1)$ and $\textsc{N}_t$ varies continuously with $t$. Recall that the second fundamental form of $\sigma_t(S)$ satisfies:
$$\nabla^{\mathcal S^3}_{\sigma_{t}v}(\sigma_{t}\hat w)-\nabla^{\I_t}_{v}(\hat w)=\II_t(v,w)\textsc{N}_t=\frac{\II_t(v,w)}{t}t\textsc{N}_t\,.$$
Now applying the transformation $g_t^*$, we have:
\begin{itemize}
\item $g_t^*\nabla^{\mathcal S^3}_{\sigma_{t}v}(\sigma_{t}\hat w)$ converges to $\nabla^{\cE^3}_{\sigma v}\sigma w$ by Proposition \ref{prop limit connection volume coeucl};
\item $g_t^* \nabla^{\I_t}_{v}(\hat w)$ converges to a tangential component to $\sigma(S)$, since $ \nabla^{\I_t}_{v}(\hat w)$ is tangent to $\sigma_t(S)$;
\item $t g_t^*(N_t)$ converges to $\textsc{T}=(0,0,0,1)$ since $g_t^*(\textsc{N}_t)=(\textsc{N}_t^1,\textsc{N}_t^2,\textsc{N}_t^3,{\textsc{N}_t^4}/{t})$ and $\textsc{N}_0=(0,0,0,1)$.
\end{itemize}
This shows at once that $\II_t(v,w)/t$ converges to the second fundamental form of $\sigma$ in $\cE^3$, and $g_t^* \nabla^{\I_t}_{v}(\hat w)$ converges to $\nabla^{\I}_{v}(\hat w)$. 

Since $\II_t(v,w)=\I_t(B_t(v),w)$ and $\II(v,w)=\I(B(v),w)$, the third point follows from the first two statements. The last point is a consequence of the third point and the fact that $K^{ext}_t=\det B_t$ and $K^{ext}=\det B$.

Clearly the proof is completely analogous for the convergence from $\dS^3$ to $^*\E^3$, or from the convergence from $\H^3$ or $\AdS^3$ to $^*\M^3$.
\end{proof}

%{\color{red}{add somewhere:
%\begin{equation} \label{isometries coeucl}
%\left(
%\begin{array}{ccc|c}
%  
%  & & & 0 \\
%   & \raisebox{-4.5pt}{{\huge\mbox{{$A$}}}}  & & 0  \\
%  & & & 0 \\ \hline
%  a_1 & a_2 & a_3 & 1
%\end{array}
%\right)~,
%\end{equation}
%for $A\in\SO(3)$}}

\subsection{Projective nature of infinitesimal rigidity}\label{sec:inf rig}

Let $S$ be a smooth space-like surface in  $\E^3$, $\M^ 3$, $\Ell^3$, $\H^3$, $\dS^3$ or $\AdS^3$.
 A vector field $Z$ on $S$ is called an \emph{infinitesimal isometric deformation}\index{infinitesimal isometric deformation} if it preserves the induced metric on $S$ at the first order: $\textsc{L}_Zg(X,Y)=0\quad \forall X,Y\in TS $,  which is easily shown to be equivalent to
$g(\nabla_X Z,X)=0\quad \forall X\in TS$, if $g$ is the metric of the ambient space and $\nabla$ its Levi--Civita connection.
 An infinitesimal isometric deformation  is said to be \emph{trivial} if it is the restriction to $S$ of a Killing field of the ambient space.

Now let us suppose that $S$ embeds in an affine chart, 
and that this affine representation is also an affine chart for another model space, in which the image of $S$ is space-like. 
By a straightforward adaptation of the  proof of Lemma~\ref{lem:killing}, it is easy to see that the infinitesimal Pogorelov map sends 
infinitesimal isometric deformations of $S$ for the first ambient metric to infinitesimal isometric deformations of $S$ for the second ambient metric. Moreover,  Lemma~\ref{lem:killing} also says that trivial deformations are sent to trivial deformations. 
A surface $S$ is \emph{infinitesimally rigid}\index{infinitesimal rigidity} if all its infinitesimal isometric deformations are trivial.

{It follows from the above discussion that the infinitesimal rigidity of a surface in a three dimensional affine space is independent of the choice of the projective distance on the ambient space.}

%\begin{fact}\label{fact:iid}
%Infinitesimal rigidity of a surface does not depend on the projective metric.
%\end{fact}

\paragraph{Comments and references} \small
\begin{itemize}
\item A classical problem in Riemannian geometry is the existence of zero mean curvature surfaces with prescribed boundary at infinity. For instance, in hyperbolic space the so-called asymptotic Plateau problem was proved by Anderson in \cite{anderson}. In Anti-de Sitter space, the analogous problem was tackled by Bonsante and Schlenker in \cite{bon_schl}. With the tools introduced in this paper, the analogous problem in $^*\M^3$ turns out to be very simple. Indeed, 
taking the trace in the expression
$B=\Hess^{\H^3} u-u\, \mathrm{Id}$
for the shape operator of a space-like surface, one obtains that a surface has zero mean curvature if and only if 
$
\Delta u-2u=0$
where $\Delta$ denotes the hyperbolic Laplacian. Hence a proof follows  from the existence and uniqueness of solutions to this linear PDE.
\item Volkov \cite{volkov} showed the following: a vector field $Z$ on $S$ is an infinitesimal isometric deformation if and only if for each  $x\in S$ there exists a  Killing field $K^x$ such that
$Z_x=(K^x)_x$ and $(\nabla_X Z)_x = (\nabla_X K^x)_x$ $\forall X\in T_xS~.$
\item {To simplify the presentation, we have considered different metrics on the same (affine) space. But we could have considered maps between different spaces, sending geodesics onto geodesics. In particular, it would follow from the result presented here that infinitesimal rigidity in the Euclidean space is invariant under projective transformations.} This fact 
is known since the end of the 19th century and called the Darboux--Sauer theorem, see e.g. 
 \cite[Chapitre IV]{dar96}. 
 This is also true for more general surfaces than smooth ones. For example consider polyhedral surfaces made of pieces of space-like planes. An infinitesimal isometric deformation is then defined as the data of a Killing field on each face, such that they coincide on common edges. 
Then Lemma~\ref{lem:killing}  implies that the infinitesimal Pogorelov map sends deformations to deformations and trivial deformations to trivial deformations. This invariance of infinitesimal rigidity is always true if one considers the more general case of frameworks, that includes polyhedral surfaces \cite{izm09,ivan}.
\item The \emph{Pogorelov map}\index{Pogorelov map} is a map taking two convex surfaces $S_1,S_2$ in $\Ell^3$ and giving two convex surfaces $\bar S_1,\bar S_2$ of $\E^3$. {If $S_1$ is isometric to $S_2$ (for the distance induced by the ambient metric) then $\bar S_1$ and $\bar S_2$ are also isometric. Moreover, if the isometry between $S_1$ and $S_2$ is the restriction to $S_1$ of an isometry of the ambient space, then the isometry between $\bar S_1$ and $ \bar S_2$ is also the restriction of an isometry of the ambient space.} The difficult part is to prove that $\bar S_1$ and $\bar S_2$ are convex \cite{Pogo}. The differential of this map at $S=S_1=S_2$ gives the infinitesimal Pogorelov map.
The map was extended to $\H^3$ \cite{Pogo} and to $\dS^3$  \cite{milka}  instead of $\Ell^3$, see also
\cite{sch98,sch00}. 
\end{itemize}

\normalsize

\printindex

\bibliographystyle{plain}
\bibliography{review-hyp}

\end{document}